\documentclass{amsart}
\usepackage{amssymb}
\usepackage{graphicx}
\usepackage{epstopdf}
\usepackage{subfig}
\usepackage{float}
\usepackage{placeins}

\theoremstyle{plain}
\newtheorem{theorem}{Theorem}[section]
\newtheorem{conjecture}[theorem]{Conjecture}
\newtheorem{proposition}[theorem]{Proposition}
\newtheorem{lemma}[theorem]{Lemma}
\newtheorem{corollary}[theorem]{Corollary}

\newtheorem{problem}[theorem]{Problem}
\theoremstyle{definition}
\newtheorem{definition}[theorem]{Definition}
\theoremstyle{remark}
\newtheorem{remark}[theorem]{Remark}
\newtheorem{example}[theorem]{Example}

\def\Z{\mathbb{Z}}
\def\R{\mathbb{R}}
\def\C{\mathbb{C}}
\def\I{\mathcal{I}}

\def\S{\mathcal{S}}
\def\A{\mathcal{A}}
\def\Vol{\mathrm{Vol}}
\def\Mat{\mathit{Mat}}

\title{Arrangements of equal minors in the positive Grassmannian}
\author{Miriam Farber \and Alexander Postnikov}
\address{Department of Mathematics, Massachusetts Institute of Technology,
77 Massachusetts Avenue, Cambridge MA 02139}
\email{mfarber@mit.edu}
\email{apost@math.mit.edu}
\thanks{M.F.\ is supported in part by the NSF graduate research fellowship grant 1122374.
A.P.\ is supported in part by the NSF grant DMS-1362336.}
\date{January 28, 2015}


\begin{document}

\begin{abstract}
We discuss arrangements of equal minors of totally positive matrices.  More
precisely, we investigate the structure of equalities
and inequalities between the minors.  We show that arrangements of equal
minors of largest value are in bijection with {\it sorted sets,} which earlier
appeared in the context of {\it alcoved polytopes\/} and Gr\"obner bases.
Maximal  arrangements of this form correspond to simplices of the alcoved
triangulation of the hypersimplex; and the number of such arrangements equals
the {\it Eulerian number}.  On the other hand, we prove in many cases that
arrangements of equal minors of smallest value are exactly {\it weakly
separated sets.}   Weakly separated sets, originally introduced by Leclerc and
Zelevinsky, are closely related to the {\it positive Grassmannian\/} and the
associated {\it cluster algebra.}    However, we also construct examples of arrangements of
smallest minors which are not weakly separated using {\it chain reactions\/}
of mutations of {\it plabic graphs.}
\end{abstract}

% We relate these arrangements with
%sequences of mutations on {\it plabic graphs.}

%and on more general 2-dimensional surfaces.

\maketitle

\section{Introduction}
\label{sec:in}

In this paper, we investigate possible equalities and inequalities between minors of totally positive matrices.
This study is closely related to
Leclerc-Zelevinsky's weakly separated sets \cite{LZ, OPS},
Fomin-Zelevinsky's cluster algebras \cite{FZ1, FZ2},
combinatorics of the positive Grassmannian \cite{Pos2},
alcoved polytopes \cite{LP} and triangulations of hypersimplices,
as well as other topics.

\medskip

One motivation for the study of equal minors came from a variant of the {\it matrix completion problem.}
%The {\it matrix completion problem\/} is
This is the problem about completing missing
entries of a partial matrix so that the resulting matrix satisfies a certain
property (e.g., it is positive definite or totally positive).  Completion
problems arise in various of applications, such as statistics, discrete
optimization, data compression, etc.

Recently, the following variant of the completion problem
was investigated in
\cite{FFJM} and \cite{FRS}.  It is well-known that one can
``slightly perturb'' a totally nonnegative
matrix (with all nonnegative minors)
and obtain a totally positive matrix (with all strictly positive minors).
It is natural to ask how to do this in a minimal way.
In other words, one would like to find the {\it minimal\/} number of matrix entries that one
needs to change in order to get a totally positive matrix.
The most degenerate totally nonnegative matrix all of whose entries are positive is the matrix filled
with all 1's.  The above question for this matrix can be equivalently reformulated as follows:
What is the {\it maximal\/} number of equal entries in a totally positive matrix?
(One can always rescale all equal matrix entries to 1's.)
It is then natural to ask about the maximal number of equal minors in a totally positive matrix.

In \cite{FFJM, FRS}, it was shown that the maximal number of
equal entries in a totally positive $n\times n$ matrix is $\Theta(n^{4/3})$,
and that the maximal number of equal $2\times 2$-minors in a
$2\times n$ totally positive matrix is $\Theta(n^{4/3})$.
It was also shown that the maximal number of equal $k\times k$ minors
in a $k\times n$ totally positive matrix is $O(n^{k-{k\over k+1}})$.
The construction is based on the famous
Szemer\'edi-Trotter theorem \cite{ST} (conjectured by Erd\"os) about the maximal number
of point-line incidences in the plane.

\medskip

Another motivation came from the study of combinatorics of the {\it positive
Grassmannian\/} \cite{Pos2}.  The nonnegative part of the Grassmannian $Gr(k,n)$
can be subdivided into {\it positroid cells,} which are defined by setting some subset
of the Pl\"ucker coordinates (the  maximal minors) to zero, and requiring the
other Pl\"ucker coordinates to be strictly positive.
The positroid cells and the corresponding arrangements of zero and
positive Pl\"ucker coordinates were combinatorially characterized in \cite{Pos2}.

We can introduce the finer subdivision of the nonnegative part of the Grassmannian,
where the strata are defined by all possible equalities and inequalities between the Pl\"ucker
coordinates.  This is a ``higher analog'' of the positroid stratification.
A natural question is:  How to extend the combinatorial constructions from
\cite{Pos2} to this ``higher positroid stratification'' of the Grassmannian?

One would like to get an explicit combinatorial description of all possible collections of equal minors.
In general, this seems to be a hard problem, which is still far from the complete solution.
However, in cases of minors of smallest and largest values, the problem leads to the structures
that have a nice combinatorial description.

In this paper we show that arrangements of equal minors of largest value are
exactly {\it sorted sets.}   Such sets correspond to the simplices of the alcoved
triangulation of the hypersimplex \cite{Sta, LP}.  They appear in the study of Gr\"obner bases
\cite{Stu} and in the study of alcoved polytopes \cite{LP}.

On the other hand, we show that arrangements of equal minors of smallest value
include {\it weakly separated sets\/} of Leclerc-Zelevinsky \cite{LZ}.   Weakly
separated sets are closely related to the positive Grassmannian and {\it plabic graphs} \cite{OPS, Pos2}.
In many cases, we prove that arrangements of smallest minors are exactly weakly separated sets.

However, we also construct examples of arrangements of smallest minors which are not weakly
separated,  and make a conjecture on the structure of such arrangements.
We construct these examples using certain {\it chain reactions\/} of mutations of plabic graphs,
and also vizualize them geometricaly using square pyramids and octahedron/tetrahedron moves.
%certain 2-dimensional surfaces in $\R^3$.

\medskip

We present below the general outline of the paper.
In Section~\ref{sec:from_Mat_to_Gr}, we discuss the positive Grassmannian $Gr^+(k,n)$.
In Section~\ref{sec:arr_minors}, we define arrangements of minors.
As a warm-up, in Section~\ref{sec:k=2},
we consider the case of the positive Grassmannian $Gr^+(2,n)$.
In this case, we show that maximal arrangements of smallest minors are in bijection
with triangulations of the $n$-gon, while the arrangements of largest minors are in bijection with
thrackles, which are the graphs where every pair of edges intersect.
In Section~\ref{sec:WS_sorted}, we define weakly separated sets and sorted sets.   They generalize
triangulations of the $n$-gon and thrackles.
We formulate our main result (Theorem~\ref{thm:sorted}) on arrangements of largest minors,
which says that these arrangements coincide with sorted sets.
We also give results
(Theorems~\ref{thm:WSeparated} and~\ref{thm:weakly_separated_123}) and
Conjecture~\ref{conj:max_smallest} on arrangements of smallest minors,
that relate these arrangements with weakly separated sets.
In Section~\ref{sec:inequalities_products_minors}, we use Skandera's inequalities \cite{Ska}
for products of minors to prove one direction ($\Rightarrow$) of Theorems~\ref{thm:sorted}
and~\ref{thm:weakly_separated_123}.
In Section~\ref{sec:cluster}, we discuss the cluster algebra associated with the Grassmannian.
According to \cite{OPS, Pos2},
maximal weakly separated sets form clusters of this cluster algebra.
We use Fomin-Zelevinsky's Laurent phenomenon \cite{FZ1} and
the positivity result of Lee-Schiffler \cite{LS} to prove Theorem~\ref{thm:WSeparated}.
In Section~\ref{sec:construction_largest}, we prove the other direction ($\Leftarrow$)
of Theorem~\ref{thm:sorted}.  In order to do this, for any sorted set,
we show how to construct an element of the Grassmannian, that is
a matrix with needed equalities and inequalites between the minors.
We actually show that any torus orbit on the positive Grassmannian
$Gr^+(k,n)$ contains the Eulerian number $A(n-1,k-1)$ of such special elements
(Theorem~\ref{thm:torus_action}).
We give examples for $Gr^+(3,5)$ and $Gr^+(3,6)$ that can be described as certain labellings
of vertices of the regular pentagon and hexagon by positive numbers.
The proof of Theorem~\ref{thm:torus_action} is based on the theory of alcoved polytopes \cite{LP}.
In Section~\ref{sec:sort_closed}, we extend the results on arrangements of largest minors
in a more general context of sort-closed sets.  In this setting, the number of maximal arrangements
of largest minors equals the normalized volume of the corresponding alcoved polytope.
In Section~\ref{sec:matrix_entries}, we discuss equalities between matrix entries in a totally
positive matrix, which is a special case of the construction from the previous section.
In Section~\ref{sce:nonnegative_Grass},
we discuss the case of the nonnegative Grassmannian $Gr^{\geq }(2,n)$.
If we allow some minors to be zero, then we can actually achieve a larger number ($\simeq n^2/3$)
of equal positive minors.
In Section~\ref{sec:not_weakly_separated}, we construct examples of arrangements of smallest
minors for $Gr^+(4,8)$ and $Gr^+(5,10)$, which are not weakly separated.
We formulate Conjecture~\ref{conjecture:minimal_minors} on the structure of pairs of
equal smallest minors, and prove it for $Gr^+(k,n)$ with $k\leq 5$.
Our construction uses plabic graphs, especially honeycomb plabic graph that have mostly
hexagonal faces.  We describe certain chain reactions of mutations (square moves) for these graphs.
We also give a geometric vizualization of these chain reactions using square pyramids.
%and the octahedron/tetrahedron moves.
In Section~\ref{sec:final_remarks}, we give a few final remarks.

\section{From totally positive matrices to the positive Grassmannian}
\label{sec:from_Mat_to_Gr}

%A matrix is called {\it totally positive\/} (resp., {\it totally nonnegative}) if all its minors, that is,
%determinants of square submatrices (of all sizes),  are positive (resp., nonnegative).

A matrix is called {\it totally positive\/} (resp., {\it totally nonnegative})
if all its minors, that is, determinants of square submatrices (of all sizes),
are positive (resp., nonnegative).  The notion of total positivity was
introduced by Schoenberg \cite{Sch} and Gantmacher and Krein \cite{GK} in the
1930s. Lusztig \cite{Lu1, Lu2} extended total positivity in the general
Lie theoretic setup
and defined the positive part for a reductive Lie group $G$ and
a generalized partial flag manifold $G/P$.

For $n\geq k\geq 0$, the {\it Grassmannian} $Gr(k,n)$ (over $\R$) is the space of $k$-dimensional linear subspaces in $\R^n$.
It can be identified with the space of real $k\times n$ matrices of rank $k$ modulo row operations.
(The rows of a matrix span a $k$-dimensional subspace in $\R^n$.)
The maximal $k\times k$ minors of $k\times n$ matrices form projective coordinates on the Grassmannian, called
the {\it Pl\"ucker coordinates.}
We will denote the Pl\"ucker coordinates by $\Delta_I$, where $I$ is a $k$-element subset in $[n]:=\{1,\dots,n\}$ corresponding
to the columns of the maximal minor.  These coordinates on $Gr(k,n)$ are not algebraically independent; they satisfy the
Pl\"ucker relations.

In \cite{Pos2},
the {\it positive Grassmannian\/} $Gr^+(k,n)$ was described as the subset of the Grassmannian $Gr(k,n)$ such
that all the Pl\"ucker coordinates are simultaneously positive: $\Delta_I >0$ for all $I$.
(Strictly speaking, since the $\Delta_I$ are projective coordinates defined up to rescaling,
one should say ``all $\Delta_I$ have the same sign.'')
Similarly, the {\it nonnegative Grassmannian\/} $Gr^{\geq} (k,n)$ was defined by the condition $\Delta_I\geq 0$ for all $I$.
This construction agrees with Lusztig's general theory of total positivity.
(However, this is a nontrivial fact that Lusztig's positive part of $Gr(k,n)$ is the same as $Gr^+(k,n)$
defined above.)

%Lie theoretic definition \cite{Lu2} of the positive part of a generalized
%flag variety $G/P$.
%This construction agrees with Lusztig's general Lie theoretic definition \cite{Lu} of the positive part of a generalized
%flag variety $G/P$.

The space of totally positive (totally nonnegative) $k\times m$ matrices $A=(a_{ij})$ can be embedded
into the  positive (nonnegative) Grassmannian $Gr^+(k,n)$ with $n = m+k$,
as follows, see \cite{Pos2}.    The element of the Grassmannian $Gr(k,n)$ associated with a $k\times m$ matrix $A$
is represented by the $k\times n$ matrix
$$
\phi(A) =
\left(
\begin{array}{cccccccccc}
1 & 0 & \cdots & 0 & 0 & 0 & (-1)^{k-1} a_{k1} & (-1)^{k-1} a_{k2} & \cdots & (-1)^{k-1} a_{km} \\
%0 & 1 & \cdots & 0 & 0 & 0 & (-1)^{k-2} a_{k-1,1} & (-1)^{k-2} a_{k-1,2} & \cdots & (-1)^{k-2} a_{k-1,m}\\
\vdots & \vdots & \ddots & \vdots & \vdots & \vdots & \vdots & \vdots & \ddots & \vdots\\
0 & 0 & \cdots & 1 & 0 & 0 & a_{31} & a_{32} & \cdots & a_{3m}\\
0 & 0 & \cdots & 0 &1 & 0 & -a_{21} & - a_{22} & \cdots & -a_{2m} \\
0 & 0 & \cdots & 0 & 0 & 1 & a_{11} & a_{12} & \cdots & a_{1m}
\end{array}
\right).
$$
Under the map $\phi$, all minors (of all sizes) of the $k\times m$ matrix $A$ are equal to the {\it maximal\/} $k\times k$-minors
of the extended $k\times n$ matrix $\phi(A)$.
More precisely, let $\Delta_{I,J}(A)$ denotes the minor of the $k\times m$ matrix $A$ in row set
$I=\{i_1,\dots,i_r\}\subset [k]$ and column set
$J=\{j_1,\dots,j_r\}\subset[m]$;
and let $\Delta_K(B)$ denotes the maximal $k\times k$ minor of a $k\times n$ matrix $B$ in column set $K\subset[n]$, where $n=m+k$.
Then
$$
\Delta_{I,J}(A) = \Delta_{([k]\setminus\{k+1-i_r,\dots,k+1-i_1\})\cup \{j_1+k,\dots,j_r+k\}}(\phi(A)).
$$

This map is actually a {\it bijection\/} between the space of totally positive $k\times m$ matrices and
the positive Grassmannian $Gr^+(k,n)$.  It also identifies the space of totally nonnegative $k\times m$
matrices with the subset of the totally nonnegative Grassmannian $Gr^{\geq } (k,n)$ such that the Pl\"ucker coordinate
$\Delta_{[k]}$ is nonzero.  Note, however, that the whole totally nonnegative Grassmannian $Gr^{\geq } (k,n)$
is strictly bigger than the space of totally nonnegative $k\times m$ matrices,
and it has a more subtle combinatorial structure.

%%However, the {\it nonnegative\/} Grassmannian $Gr^{\geq } (k,n)$ is strictly bigger than the space of totally nonnegative
%%$k\times m$ matrices;
%and it has a more subtle combinatorial structure.

This construction allows us to reformulate questions about equalities and inequalities between minors (of various sizes)
in terms of analogous questions for the positive Grassmannian, involving only maximal $k\times k$ minors
(the Pl\"ucker coordinates).   One immediate technical simplification is that, instead
of minors with two sets of indices (for rows and columns), we will use the Pl\"ucker coordinates $\Delta_I$ with one
set of column indices $I$.  More significantly, the reformulation of the problem in terms of the Grassmannian unveils
{\it symmetries\/} which are hidden on the level of matrices.

Indeed, the positive Grassmannian $Gr^+(k,n)$ possesses the {\it cyclic symmetry.}  Let $[v_1,\dots,v_n]$ denotes a point in
$Gr(k,n)$ given by $n$ column vectors $v_1,\dots,v_n\in \R^k$.
Then the map
$$
[v_1,\dots,v_n]\mapsto [(-1)^{k-1}\, v_n, v_1,v_2,\dots,v_{n-1}]
$$
preserves the positive Grassmannian $Gr^+(k,n)$.  This defines the action of the cyclic group $\Z/n\Z$ on the positive
Grassmannian $Gr^+(k,n)$.

We will see that all combinatorial structures that appear in
the study of the positive Grassmannian and arrangements of equal minors
have the cyclic symmetry related to this action of $\Z/n\Z$.

\section{Arrangements of minors}
\label{sec:arr_minors}

\begin{definition}
Let $\I=(\I_0,\I_1, \dots,\I_l)$ be an ordered set-partition of the set $[n]\choose k$ of all $k$-element subsets in $[n]$.
Let us subdivide the nonnegative Grassmannian $Gr^{\geq}(k,n)$ into the strata $S_\I$ labelled by such ordered set partitions
$\I$ and given by the conditions:
\begin{enumerate}
\item
$\Delta_I = 0$ for $I\in \I_0$,
\item
$\Delta_I = \Delta_J$ if $I,J\in \I_i$,
\item
$\Delta_I < \Delta_J$ if $I\in \I_i$ and $J\in \I_j$ with $i<j$.
\end{enumerate}

An {\it arrangement of minors\/} is an ordered set-partition $\I$ such that the stratum $S_\I$ is not empty.
\end{definition}

\begin{problem}
Describe combinatorially all possible arrangements of minors in $Gr^{\geq}(k,n)$.
Investigate the geometric and the combinatorial structure of the stratification
$Gr^{\geq}(k,n)=\bigcup S_\I$.
\end{problem}

For $k=1$, this stratification is equivalent to the subdivision of the linear space $\R^n$
by the hyperplanes $x_i=x_j$, which forms the {\it Coxeter arrangement\/} of type A,
also known as the {\it braid arrangement.}
The structure of the Coxeter arrangement is well studied.
Combinatorially, it is equivalent of the face structure of the {\it permutohedron.}

For $k\geq 2$, the above problem seems to be quite nontrivial.

\cite{Pos2} described the cell structure of the nonnegative Grassmannian $Gr^{\geq}(k,n)$,
which is equivalent to the description of possible sets $\I_0$.
This description already involves quite rich and nontrivial combinatorial structures.
It was shown that possible $\I_0$'s are in bijection with various
combinatorial objects: positroids, decorated permutations, L-diagrams, Grassmann necklaces, etc.
The stratification of $Gr^{\geq} (k,n)$ into the strata $S_\I$
is a finer subdivision of the {\it positroid stratification\/} studied
in \cite{Pos2}.  It should lead to even more interesting combinatorial objects.

In the present paper, we mostly discuss the case of the positive Grassmannian
$Gr^+(k,n)$, that is, we assume that $\I_0 = \emptyset$.   We concentrate on a
combinatorial description of possible sets $\I_1$ and $\I_l$.
In Section~\ref{sce:nonnegative_Grass} we also discuss some results for the nonnegative
Grassmannian $Gr^{\geq}(k,n)$.

%In the last section, we present an interesting relation between the structure
%of $\I_1$ and chain reactions on plabic graph, and also more general
%2-dimensional surfaces.

\begin{definition}
We say that a subset $\mathcal{J}\subset{[n]\choose k}$ is
an {\it arrangement of smallest minors\/} in $Gr^{+}(k,n)$, if there exists a nonempty stratum $S_{\I}$ such that
$\I_0=\emptyset$ and $\I_1=\mathcal{J}$.

We also say that $\mathcal{J}\subset{[n]\choose k}$ is an {\it arrangement of largest minors\/} in $Gr^{+}(k,n)$
if there exists a nonempty stratum $S_{\I}$ such that $\I_0=\emptyset$ and $\I_l=\mathcal{J}$.
%
%Also we will call a possible set $\I_l$
%(Here we assume that $\I_0=\emptyset$.)
\end{definition}

\medskip

As a warm-up, in the next section we describe all possible arrangements of smallest and largest minors
in the case $k=2$.  We will treat the general case in the subsequent sections.

\section{Case $k=2$: triangulations and thrackles}
\label{sec:k=2}

In the case $k=2$, one can identify 2-element sets $I=\{i,j\}$ that label the Pl\"ucker coordinates $\Delta_I$ with the edges
$\{i,j\}$ of the complete graph $K_n$ on the vertices $1,\dots,n$.
A subset in $[n]\choose 2$ can be identified with a subgraph $G\subset K_n$.

Let us assume that the vertices $1,\dots,n$ are arranged on the circle in the clockwise order.

\begin{definition}
For distinct $a,b,c,d\in[n]$, we say that the two edges $\{a,b\}$ and $\{c,d\}$ are {\it non-crossing}
if the corresponding straight-line chords $[a,b]$
and $[c,d]$ in the circle do not cross each other.
Otherwise, if the chords $[a,b]$ and $[c,d]$ cross each other, we say that the edges
$\{a,b\}$ and $\{c,d\}$ are {\it crossing.}
\end{definition}

For example, the two edges $\{1,4\}$ and $\{2,3\}$ are non-crossing; while the edge $\{1,3\}$ and $\{2,4\}$ are crossing.

\begin{theorem}
A nonempty subgraph $G\subset K_n$ corresponds to an arrangement of smallest minors in $Gr^+(2,n)$ if and only if every pair of edges in $G$
is non-crossing, or they share a common vertex.
\label{thm:non_crossing}
\end{theorem}

\begin{theorem}
%\label{crossing}
A nonempty subgraph $H\subset K_n$ corresponds to an arrangement of largest minors in $Gr^+(2,n)$ if and only if every pair of edges in $H$
is crossing, or they share a common vertex.
\label{thm:crossing}
\end{theorem}

In one direction ($\Rightarrow$), both Theorems~\ref{thm:crossing} and~\ref{thm:non_crossing}
easily follow from the {\it 3-term Pl\"ucker relation\/} for the Pl\"ucker coordinates
$\Delta_{ij}$ in $Gr^+(2,n)$:
$$
\Delta_{ac} \, \Delta_{bd} = \Delta_{ab}\,\Delta_{cd}+ \Delta_{ad}\,\Delta_{bc},\quad
\textrm{for } a<b<c<d.
$$
Here all the $\Delta_{ij}$ should be strictly positive.
Indeed, if $\Delta_{ac} = \Delta_{bd}$ then some of the minors $\Delta_{ab},\Delta_{bc},\Delta_{cd},\Delta_{ad}$
should be strictly smaller than $\Delta_{ac} = \Delta_{bd}$.
Thus the pair of crossing edges $\{a,c\}$ and $\{b,d\}$ cannot belong to an arrangement of smallest minors.
On the other hand, if, say, $\Delta_{ab} = \Delta_{cd}$, then $\Delta_{ac}$ or $\Delta_{bd}$ should be strictly greater than
 $\Delta_{ab} = \Delta_{cd}$.  Thus the pair of non-crossing edges $\{a,b\}$ and $\{c,d\}$
cannot belong to an arrangement of largest minors.  Similarly, the pair of non-crossing edges $\{a,d\}$ and $\{b,c\}$ cannot
belong to an arrangement of largest minors.
%In a similar fashion, we can show that any other pair of crossing (non-crossing) edges cannot occur in an arrangement
%of smallest (largest) minors.

In order to prove Theorems~\ref{thm:crossing} and~\ref{thm:non_crossing}
it remains to show that, for any nonempty subgraph of $K_n$ with no crossing (resp., with no
non-crossing) edges, there exists an element of $Gr^+(2,n)$ with the
corresponding arrangement of equal smallest (resp., largest) minors.
We will give explicit constructions of $2\times n$ matrices that represent
such elements of the Grassmannian.  Before we do this, let us discuss triangulations and thrackles.

\medskip

When we say that $G$ is a ``maximal'' subgraph of $K_n$ satisfying some property, we mean that it is maximal by inclusion of
edge sets, that is, there is no other subgraph of $K_n$ satisfying this property whose edge set contains the edge set of $G$.

Clearly, maximal subgraphs $G\subset K_n$ without crossing edges correspond to {\it triangulations\/} of the $n$-gon.
Such graphs contain all the ``boundary'' edges $\{1,2\}$, $\{2,3\}$,\dots, $\{n-1,n\}$, $\{n,1\}$ together with some $n-3$ non-crossing diagonals
that subdivide the $n$-gon into triangles,
see Figure~\ref{fig:triangulation_thrackle} (the graph on the left-hand side) and
Figure~\ref{fig:triangulations}.
Of course, the number of triangulations of the $n$-gon is the famous
{\it Catalan number\/} $C_{n-2} = {1\over n-1} {2(n-2)\choose n-2}$.

\begin{figure}[h]
\qquad \qquad
\includegraphics[height=1in,width=1in]{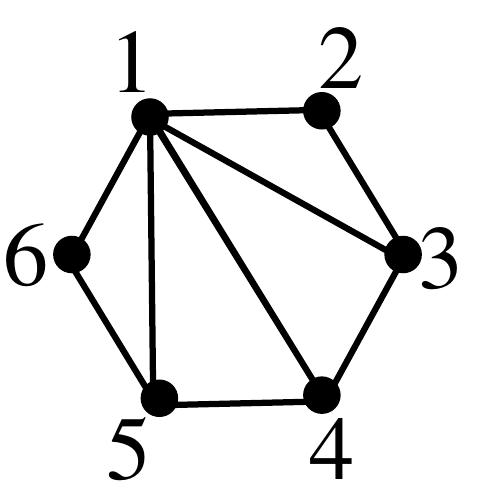}
\qquad \qquad
\includegraphics[height=1in,width=1in]{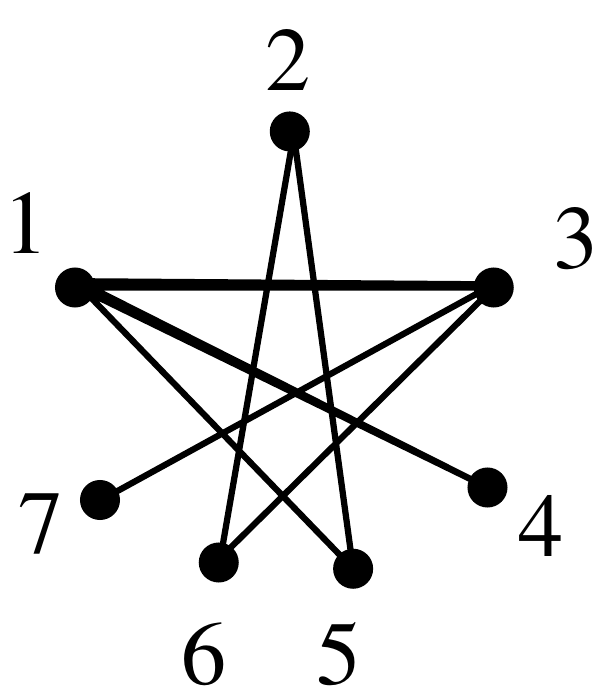}
\caption{
A {\it triangulation\/} (left) and a {\it thrackle\/} (right).
The edges of the triangulation correspond to
the arrangement of smallest minors
%equal Pl\"ucker coordinates
$\Delta_{12}=\Delta_{23}=\Delta_{34}=\Delta_{45}=\Delta_{56}=\Delta_{16}=\Delta_{13}=\Delta_{14}=\Delta_{15}$
%of {\it smallest\/} value in
in the positive Grassmannian $Gr^+(2,6)$;
while the edges of the thrackle correspond to the arrangement
of largest minors
%equal Pl\"ucker coordinates
$\Delta_{13}=\Delta_{14}=\Delta_{15}=\Delta_{25}=\Delta_{26}=\Delta_{36}=\Delta_{37}$
%of {\it largest\/} value
in $Gr^+(2,7)$.
This thrackle is obtained from the $5$-star by adding two leaves.
}
\label{fig:triangulation_thrackle}
\end{figure}

\begin{figure}[h]
\includegraphics[height=0.8in,width=.8in]{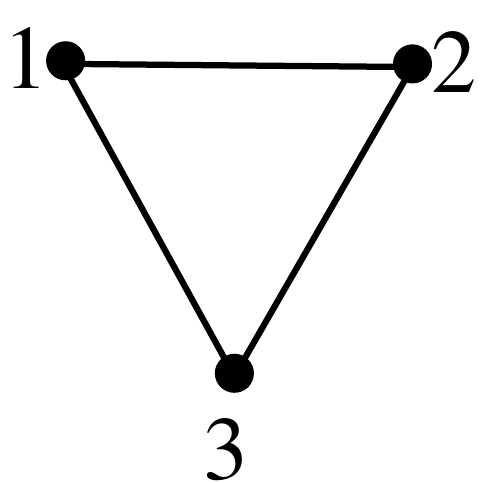}
\qquad
\includegraphics[height=0.8in,width=.8in]{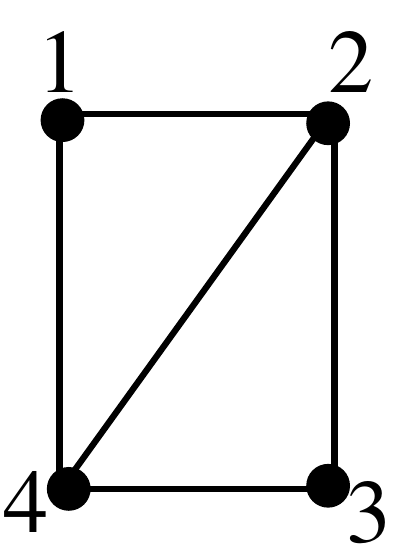}
\quad
\includegraphics[height=0.8in,width=.8in]{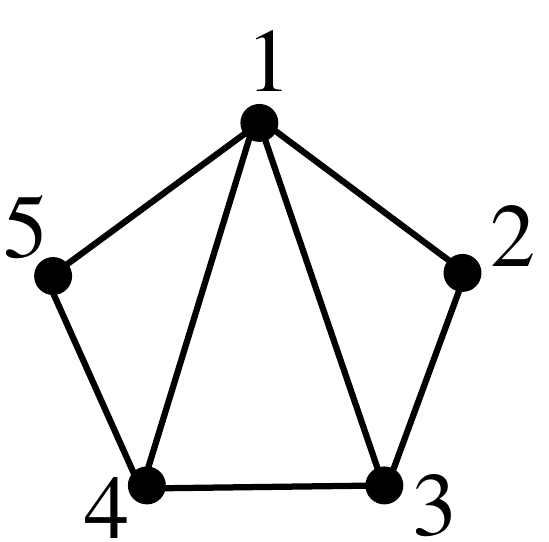}\\[.2in]
\includegraphics[height=0.8in,width=.8in]{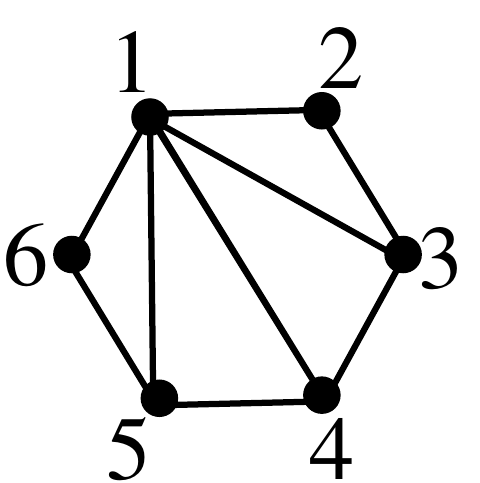}
\qquad
\includegraphics[height=0.8in,width=.8in]{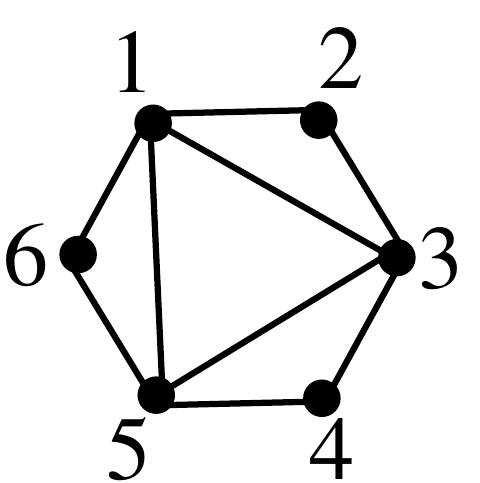}
\quad
\includegraphics[height=0.8in,width=.8in]{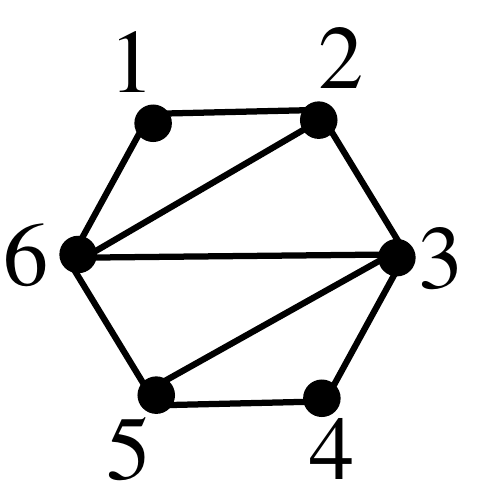}
\caption{All triangulations of $n$-gons for $n=3,4,5,6$ (up to rotations and reflections).}
\label{fig:triangulations}
\end{figure}

\begin{definition}
Let us call subgraphs $G\subset K_n$ such that every pair of edges in $G$ is crossing or shares a common vertex
{\it thrackles}\footnote{Our thrackles are a special case of Conway's thrackles.
The latter are not required to have vertices arranged on a circle.}.
\end{definition}

For an odd number $2r+1\geq 3$, let the {\it $(2r+1)$-star\/} be the subgraph of $K_{2r+1}$ such that each vertex $i$ is connected
by edges with the vertices $i+r$ and $i+r+1$, where the labels of vertices are taken modulo $2r+1$.
We call such graphs {\it odd stars.} Clearly, odd stars are thrackles.

We can obtain more thrackles by attaching some leaves to vertices of an odd star, as follows.
As before, we assume that the vertices $1,\dots,2r+1$ of the $(2r+1)$-star are arranged on a circle.
For each $i\in[2r+1]$, we can insert some number $k_i\geq 0$ of vertices arranged on the circle between the
vertices $i+r$ and $i+r+1$ (modulo $2r+1$)
and connect them by edges with the vertex $i$.  Then we should relabel all vertices of the
obtained graph by the numbers $1,\dots,n$ in the clockwise order starting from any vertex, where
$n=(2r+1)+\sum k_i$.  For example, the graph shown Figure~\ref{fig:triangulation_thrackle} (on the right-hand side)
is obtained from the $5$-star by adding two leaves.
More examples of thrackles are shown in Figure~\ref{fig:more_thrackles}.

\begin{figure}[h]
\includegraphics[height=0.8in,width=.8in]{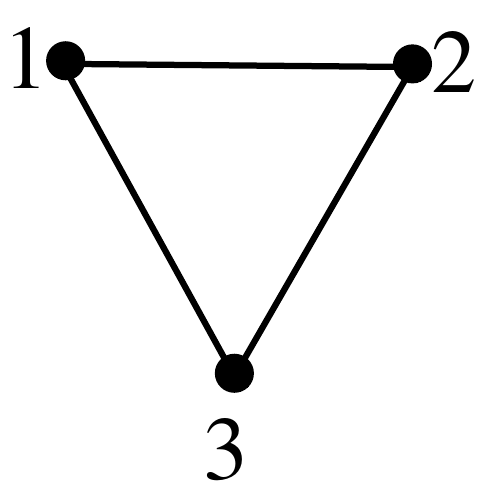}%
\qquad\quad
\includegraphics[height=0.8in,width=.8in]{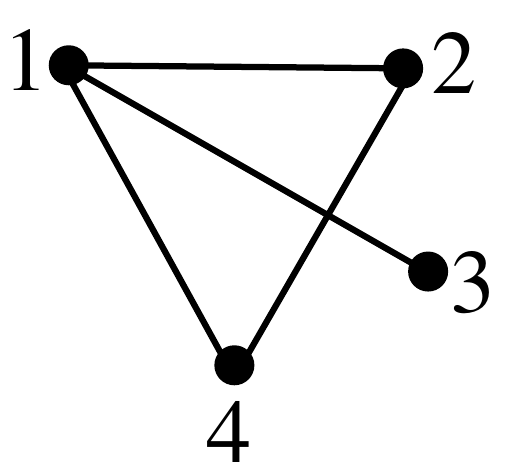}%
\qquad
\includegraphics[height=0.8in,width=.8in]{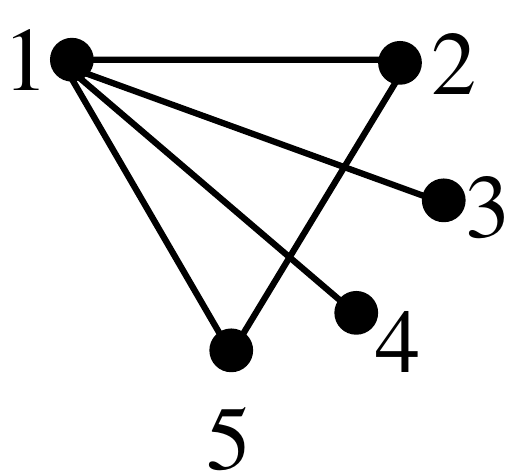}%
\qquad
\includegraphics[height=0.8in,width=.8in]{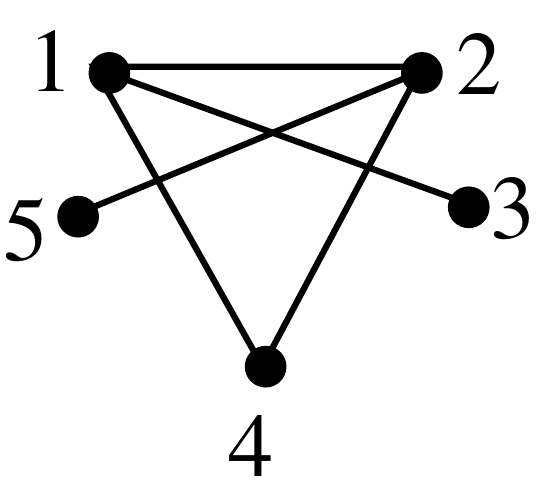}%
\qquad
\includegraphics[height=0.8in,width=.8in]{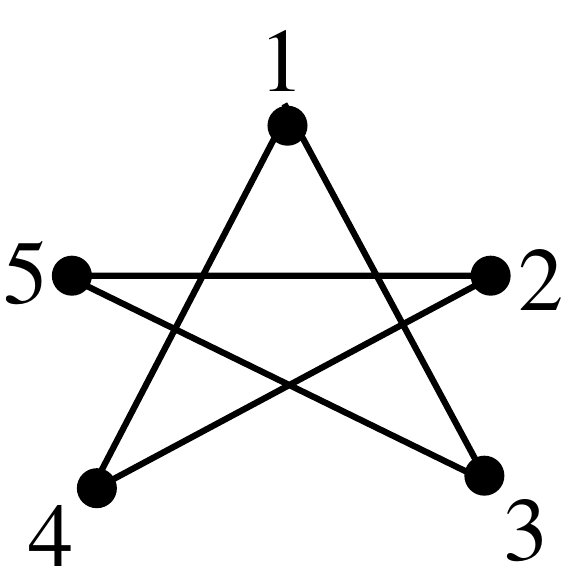}%
\qquad
\includegraphics[height=0.8in,width=.8in]{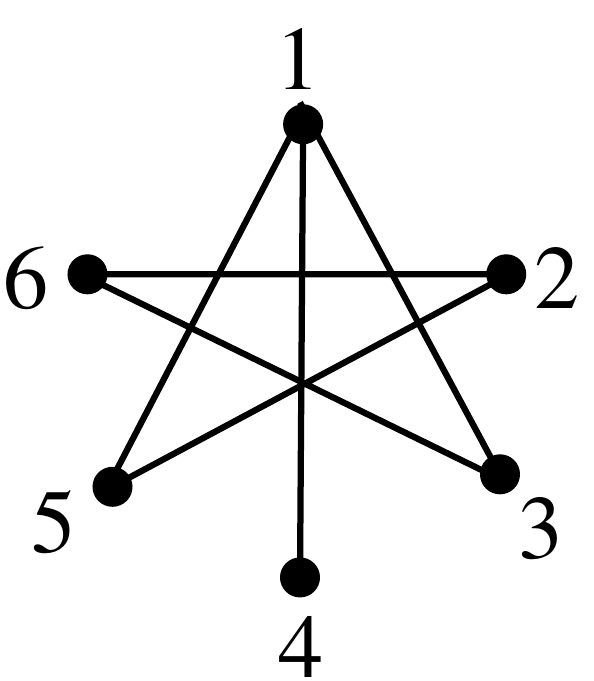}%
\qquad
\includegraphics[height=0.8in,width=.8in]{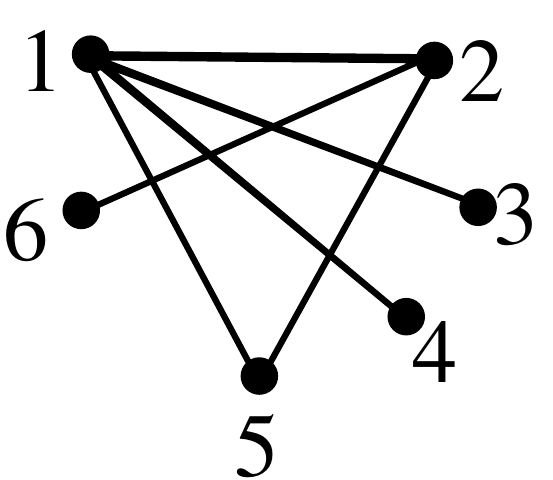}%
\qquad
\includegraphics[height=0.8in,width=.8in]{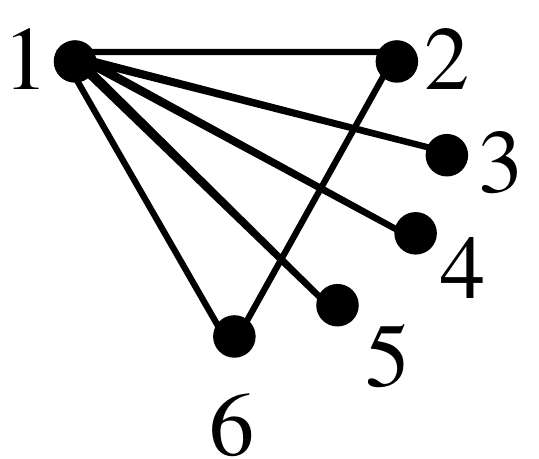}%
\qquad
\includegraphics[height=0.8in,width=.8in]{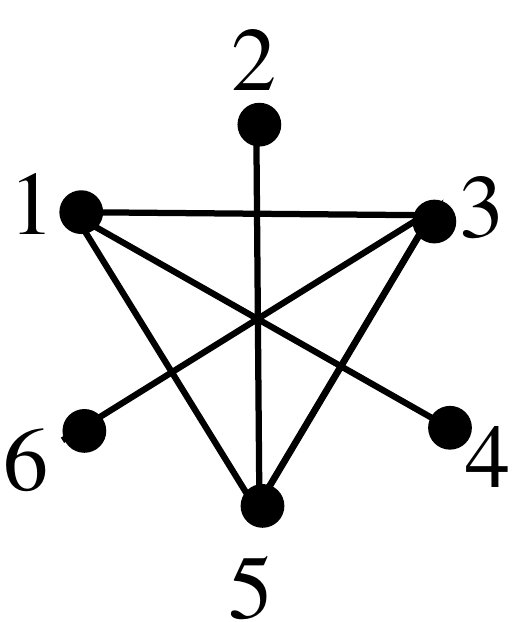}%
\qquad\quad
\caption{All maximal thrackles with 3, 4, 5, and 6 vertices (up to rotations and reflections).
These thrackles are obtained from the 3-star (triangle) and the 5-star by adding leaves.}
\label{fig:more_thrackles}
\end{figure}

%We can insert a new vertex $i$ (and increase labels of vertices

We leave the proof of the following claim as an exercise for the reader.

\begin{proposition}
\label{prop:odd_stars}
Maximal thrackles in $K_n$ have exactly $n$ edges.  They are obtained from an odd star by
attaching some leaves, as described above.
The number of  maximal thrackles in $K_n$ is $2^{n-1}-n$.
%is the Eulerian number $A(n-1,1) = 2^{n-1}-n$.
\end{proposition}

Remark that the number $2^{n-1}-n$ is the {\it Eulerian number\/} $A(n-1,1)$, that is, the number of
permutations $w_1,\dots,w_{n-1}$ of size $n-1$ with exactly one descent $w_{i-1}> w_i$.

%Figures~\ref{fig:triangulation_thrackle},
%and~\ref{fig:more_thrackles} show examples of triangulations of $n$-gons and thrackles.

Theorems~\ref{thm:non_crossing} and~\ref{thm:crossing}
imply the following results.

\begin{corollary}
Maximal arrangements of smallest minors in $Gr^+(2,n)$ correspond to
triangulations of the $n$-gon.   They contain exactly $2n-3$ minors.
The number of such maximal arrangements is the Catalan number $C_{n-2}={1\over n-1} {2(n-2)\choose n-2}$.
\end{corollary}

\begin{corollary}
Maximal arrangements of largest minors in $Gr^+(2,n)$
correspond to maximal thrackles in $K_n$.  They contain exactly $n$ minors.
The number of such maximal arrangements is the Eulerian number
$A(n-1,1) = 2^{n-1}-n$.
\end{corollary}

\medskip
Let us return to the proof of Theorem~\ref{thm:non_crossing}.
% and~\ref{thm:crossing} in the second direction.
The following claim is essentially well-known in the context of
Fomin-Zelevinsky's cluster algebra \cite{FZ1, FZ2}, more specifically, cluster
algebras of finite type $A$.
We will talk more about the connection with cluster algebras in Section~\ref{sec:cluster}.

\begin{proposition}
\label{prop:construction_for_triangulation}
Let $G\subset K_n$ be a graph corresponding to a triangulation of the $n$-gon.
Assign $2n-3$ positive real parameters $x_{ij}$ to the edges $\{i,j\}$ of $G$.

%{\rm (1)}
There exists a unique $2\times n$ matrix $A$ such that the minors $\Delta_{ij}(A)$ corresponding
to the edges $\{i,j\}$ of $G$ are $\Delta_{ij}(A) = x_{ij}$,
and that $a_{11}=1$ and $a_{21}=a_{1n}=0$.

%{\rm (2)}
All other minors $\Delta_{ab}(A)$ are Laurent polynomials in the $x_{ij}$ with positive integer coefficients
with at least two  monomials.
%Such matrix $A$ is unique up row operations (the left $SL_2$-action), i.e.,
%it is unique regarded as an element of the Grassmannian $Gr(2,n)$.
\end{proposition}

\begin{remark}
Without the conditions $a_{11}=1$ and $a_{21}=a_{1n}=0$, the matrix $A$ is
unique modulo the left $SL_2$-action.  Thus it is unique as an element of the
Grassmannian $Gr(2,n)$.  We added this condition to fix a concrete matrix
that represents this element of the Grassmannain.
\end{remark}

%For the sake of completeness, we prove this claim.

\begin{proof}
We construct $A$ by induction on $n$.  For $n=2$, we have
$A=\begin{pmatrix} 1 & 0 \\ 0 & x_{12}\end{pmatrix}$.
%For $n=3$, there is only one triangulation
%of the triangle, and we can take the matrix
%$\begin{pmatrix}
%1 & \frac{x_{23}}{x_{13}} & 0 \\
%0 & x_{12} & x_{13}
%\end{pmatrix}$.

Now assume that $n\geq 3$.  For any triangulation of the $n$-gon, there is a vertex $i\in\{2,\dots,n-1\}$ which is adjacent to
only one triangle of the triangulation.
%In other words,
This means that the graph $G$ contains the edges $\{i-1,i\}$, $\{i,i+1\}$, $\{i-1,i+1\}$.   Let $G'$ be the graph obtained from $G$
by removing the vertex $i$ together with the 2 adjacent edges $\{i-1,i\}$ and $\{i,i+1\}$;  it corresponds to a triangulation
of the $(n-1)$-gon (with vertices labelled by $1,\dots,i-1,i+1,\dots,n$).
By the induction hypothesis, we have already constructed a $2\times (n-1)$ matrix
$A'=(v_1,\dots, v_{i-1},v_{i+1},\dots,v_n)$ for the graph $G'$ with the required properties,
where $v_1,\dots,v_{i-1},v_{i+1},\dots, v_n$ are the column vectors of $A'$.
Let us take the $2\times n$ matrix
$$
A=(v_1,\dots,v_{i-1}, {x_{i,i+1}\over x_{i-1,i+1}}\, v_{i-1} + {x_{i-1,i}\over x_{i-1,i+1}}\, v_{i+1}, v_{i+1}, \dots, v_n).
$$

One easily checks that the matrix $A$ has the required properties.
Indeed,
all $2\times 2$ minors of $A$ whose indices do not include $i$ are the same as the corresponding minors of $A'$.
We have $\Delta_{i-1,i}(A)= {x_{i-1,i}\over x_{i-1,i+1}}\, \det(v_{i-1}, v_{i+1}) = x_{i-1,i}$
and $\Delta_{i,i+1}(A)= {x_{i,i+1}\over x_{i-1,i+1}}\, \det(v_{i-1}, v_{i+1}) = x_{i,i+1}$.
Also, for $j\ne i\pm 1$, the minor $\Delta_{ij}(A)$ equals
${x_{i,i+1}\over x_{i-1,i+1}}\, \det(v_{i-1},v_j)+ {x_{i-1,i}\over x_{i-1,i+1}}\, \det(v_{i+1},v_j)$,
which is a positive integer Laurent polynomial with at least two terms.

The uniqueness of $A$ also easily follows by induction.  By the induction hypothesis,
the graph $G'$ uniquely defines the matrix $A'$.  The columns $v_{i-1}$ and $v_{i+1}$ of $A'$ are linearly independent
(because all $2\times 2$ minors of $A'$ are strictly positive).   Thus the $i$th column of $A$ is a linear combination
$\alpha \, v_{i-1} + \beta\, v_{i+1}$.  The conditions $\Delta_{i-1,i}(A)=x_{i-1, i}$ and
$\Delta_{i,i+1}(A)=x_{i, i+1}$ imply that $\beta= x_{i-1,i}/\det(v_{i-1}, v_{i+1}) = x_{i-1,i}/x_{i-1,i+1}$ and
$\alpha= x_{i,i+1}/\det(v_{i-1}, v_{i+1}) = x_{i,i+1}/x_{i-1,i+1}$.
\end{proof}

%We leave it as an exercise for the reader to check that the condition $\Delta_{ij}(A)=x_{ij}$ uniquely defines $A$
%up to $SL_2$ action.

\begin{example}
Let us give some examples of matrices $A$ corresponding to triangulations.
Assume for simplicity that all $x_{ij}=1$.  According to the above construction, these matrices are obtained,
starting from the identity $2\times 2$ matrix, by repeatedly inserting sums of adjacent columns between these columns.
The matrices corresponding to the triangulations from Figure~\ref{fig:triangulations}
(in the same order) are
$$
\begin{pmatrix}
1 & 1 & 0\\
0 & 1 & 1
\end{pmatrix}
\quad
\begin{pmatrix}
1 & 1 & 1 & 0\\
0 & 1 & 2 & 1
\end{pmatrix}
\quad
\begin{pmatrix}
1 & 3 & 2 & 1 & 0\\
0 & 1 & 1 & 1 & 1
\end{pmatrix}
$$
$$
\begin{pmatrix}
1 & 4 & 3 & 2 & 1 & 0\\
0 & 1 & 1 & 1 & 1 & 1
\end{pmatrix}
\quad
\begin{pmatrix}
1 & 3 & 2 & 3 & 1 & 0\\
0 & 1 & 1 & 2 & 1 & 1
\end{pmatrix}
\quad
\begin{pmatrix}
1 & 1 & 1 & 2 & 1 & 0\\
0 & 1 & 2 & 5 & 3 & 1
\end{pmatrix}
$$
\end{example}

\begin{remark}
The inductive step in the construction of matrix $A$ given in the proof of
Proposition~\ref{prop:construction_for_triangulation} depends on a choice
of ``removable'' vertex $i$.  However, the resulting matrix $A$ is independent on
this choice.  One can easily prove this directly from the construction using
a variant of ``Diamond Lemma.''
\end{remark}

% (not only as an element of $Gr(2,n)$ but also as an actual matrix),
%according to the following ``Diamond Lemma''.

%\begin{lemma} The matrix $A$ depends only on the graph $G$ of a triangulation
%of the $n$-gon, and it does not depend on a choice of order of removing the
%vertices.
%\end{lemma}

%This claim belongs to a general class of results known as the {\it diamond lemmas.}

%\begin{proof} Induction on $n$.
%Suppose that $G$ has are two removable vertices $i$ and $j$.  We need to show that two schemes
%of constructing $A$, one starting from removing the vertex $i$, and the other startring from removing the vertex $j$,
%produce the same matrix.  Let $G'$, $G''$, and $G'''$ be the graphs obtained from $G$ by removing, respectively, the vertex $i$,
%the vertex $j$, and the both vertices $i$ and $j$.  Notice that $|i-j|\geq 2$, and vertex $j$ remains removable in
%the graph $G'$, and the vertex $i$ remains removable in $G''$.   By induction, the matrix $A'$ for the graph $G'$
%is independent on a choice of the order of removing the vertices.  For example, we can construct $A'$ by first removing
%the vertex $j$ from $G'$, i.e., $A'$ is obtained from the matrix $A'''$ for the graph $G'''$.
%Similary, the matrix $A''$ for $G''$ can be construted by first removing the vertex $i$ from $G''$, i.e.,
%it is obtained from the same matrix $A'''$.  We have two schemes to construct the matrix $A$ from the matrix $A'''$
%by first removing $i$ and then $j$, and vise versa.  One easily checks that both schemes produce the same matrix $A$.
%This proves the lemma.
%\end{proof}

We can now finish the proof of Theorem~\ref{thm:non_crossing}

\begin{proof}[Proof of Theorem~\ref{thm:non_crossing}]
Let $G\subset K_n$ be a graph with no crossing edges.  Let us pick a maximal graph $\tilde G\subset K_n$
without crossing edges (i.e., a triangulation of the $n$-gon)
that contains all edges $G$.  Construct the matrix $A$ for the graph $\tilde G$
as in Proposition~\ref{prop:construction_for_triangulation} with
$$
x_{ij}=\left\{\begin{array}{cl}
1 & \textrm{if $\{i,j\}$ is an edge of $G$}\\
1+\epsilon &\textrm{if $\{i,j\}$ is an edge of $\tilde G\setminus G$}
\end{array}\right.
$$
where $\epsilon>0$ is a small positive number.

The minors of $A$ corresponding to the edges of $G$ are equal to 1, the minors corresponding to the edges
of $\tilde G\setminus G$ are slightly bigger than 1, and all other $2\times 2$
minors are bigger than 1 (if $\epsilon$ is sufficiently small) because
they are positive integer Laurent polynomials in the $x_{ij}$ with at least two
terms.
\end{proof}

Let us now prove Theorem~\ref{thm:crossing}.
%in case of maximal thrackles.
%The case of non-maximal thrackles will be covered by more general Theorem~\ref{thm:sorted}.

\begin{proof}[Proof of Theorem~\ref{thm:crossing}]
For a thrackle $G$, we need to construct a $2\times n$ matrix $B$ such
that all $2\times 2$ minors of $B$ corresponding to edges of $G$ are equal to
each other, and all other $2\times 2$ minors are strictly smaller.

First, we consider the case of maximal thrackles.
According to Proposition~\ref{prop:odd_stars}, a maximal thrackle
$G$ is obtained from an odd star by attaching some leaves to its
vertices.  Assume that it is the $(2r+1)$-star with $k_i\geq 0$ leaves attached
to the $i$th vertex, for $i=1,\dots, 2r+1$.  We have $n=(2r+1)+\sum k_i$.

%The reader should not be confused with the use of the same word ``edge'' for edges of the graph $G$ and edges
%of the $m$-gon in the following argument.

Let $m=2(2r+1)$.  Consider a regular $m$-gon with center at the origin.
To be more specific, let us take the $m$-gon with the vertices
$u_i = (\cos(2\pi {i \over m}), \sin(2\pi {i\over m} ))$, for $i = 1,\dots,m$.
%For $j=1,\dots,m$,
Let us mark the $k_i$ points on the side  $[u_{i+r}, u_{i+r+1}]$ of the $m$-gon
that subdivide this side into $k_i+1$ equal parts, for $i=1,\dots,m$.
(Here the indices are taken modulo $m$.  We assume that $k_{i+m/2}=k_i$.)
Let $v_1,\dots,v_n,-v_1,\dots, - v_n$ be all vertices of the $m$-gon and all marked points ordered counterclockwise
starting from $v_1=u_1$.
(In order to avoid confusion between edges of the graph $G$ and edges of the $m$-gon, we
use the word ``side'' of the latter.)

For example, Figure~\ref{fig:10gon} shows the 10-gon (with extra marked points) that corresponds
to the thrackle shown on Figure~\ref{fig:triangulation_thrackle}.

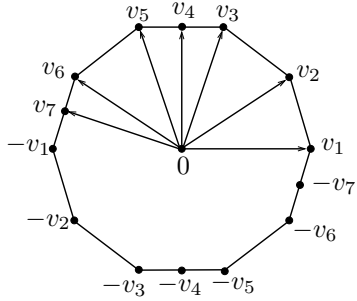
\begin{figure}[h]
\begin{picture}(0,0)%
\includegraphics{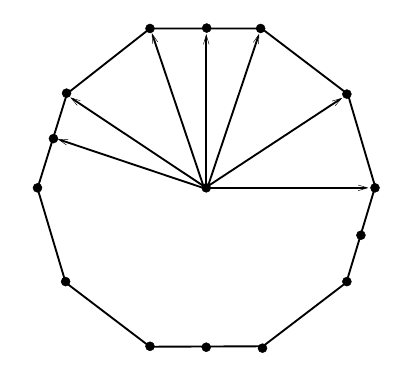}%
\end{picture}%
\setlength{\unitlength}{1184sp}%
\begingroup\makeatletter\ifx\SetFigFont\undefined%
\gdef\SetFigFont#1#2#3#4#5{%
  \reset@font\fontsize{#1}{#2pt}%
  \fontfamily{#3}\fontseries{#4}\fontshape{#5}%
  \selectfont}%
\fi\endgroup%
\begin{picture}(6233,6052)(2401,-6710)
%\put(5596,-3991){\makebox(0,0)[lb]{\smash{{\SetFigFont{5}{6.0}{\rmdefault}{\mddefault}{\updefault}{$0$}%
%}}}}
\put(5596,-4191){$0$}
\put(8619,-3684){$v_1$}
\put(8101,-2146){$v_2$}
\put(6413,-901){$v_3$}
\put(5491,-841){$v_4$}
\put(4531,-909){$v_5$}
\put(2771,-2087){$v_6$}
\put(2581,-2889){$v_7$}
\put(2016,-3715){$-v_1$}
\put(2429,-5267){$-v_2$}
\put(4021,-6667){$-v_3$}
\put(8417,-4533){$-v_7$}
\put(8022,-5505){$-v_6$}
\put(6307,-6602){$-v_5$}
\put(5192,-6632){$-v_4$}
\end{picture}
\caption{The 10-gon that corresponds to the thrackle (from Figure~\ref{fig:triangulation_thrackle})
obtained by attaching two leaves 4 and 7 to the 5-star with vertices 1, 2, 3, 5, 6.
The vertices $v_1,v_2,v_3,v_5,v_6$ of the $10$-gon correspond to the vertices of the 5-star,
and the points $v_4$ and $v_7$ on the sides of the $10$-gon correspond to the leaves
of the thrackle.}
\label{fig:10gon}
\end{figure}

We claim that the $2\times n$-matrix $B$ with the column vectors $v_1,\dots,v_n$ has the
needed equalities and inequalities between minors.
Indeed, the minor $\Delta_{ij}(B)$ equals the volume $\Vol(v_i,v_i)$ of the parallelogram generated by the vectors $v_i$ and $v_j$,
for $i<j$.
If $\{i,j\}$ is an edge of the thrackle $G$, then at least one of the vectors $v_i$ or $v_j$, say $v_i$,
is a vertex of the $m$-gon, and (the end-point of) the other vector $v_j$ lies on the side of the $m$-gon
which is farthest from the line spanned by the vector $v_i$.
In this case, the volume $\Vol(v_i,v_j)$ has the maximal possible value.
(It is equal to the half distance between a pair of opposite sides of the $m$-gon, which is equal to
$\sin({\pi r\over 2r+1})$.)

Otherwise, if $\{i,j\}$ is not an edge of the thrackle, then the volume $\Vol(v_i,v_j)$
is strictly smaller than the maximal volume.
Indeed, if $i$ is not a leaf of the thrackle (that is, $v_i$ is a vertex of the $m$-gon)
then $v_j$ does not belong to the side of the $m$-gon which is farthest from the line spanned by $v_i$,
so $\Vol(v_i,v_j)$ is smaller than the maximal value.   On the other hand, if $i$ is a leaf of the thrackle
(that is, $v_i$ lies on a side of the $m$-gon), then there is a unique vertex $v_{j'}$, $j'>i$, of the $m$-gon which lies as far from the line spanned by $v_i$ as possible.
If $j'=j$ then we can use the same argument as above,
and if $j'\ne j$, then $\Vol(v_i,v_j)<\Vol(v_i,v_{j'})$.

This proves the theorem for maximal thrackles.

Let us now assume that $G$ is not a maximal thrackle.  Pick a maximal thrackle $\tilde G$ that contains $G$.
Construct the vectors $v_1, \dots,v_n$ for $\tilde G$ as described above.
Let us show how to slightly modify the vectors $v_i$ so that some minors (namely, the minors
corresponding to the edges in
$\tilde G\setminus G$) become smaller, while the minors corresponding to the edges of $G$ remain the same.

Suppose that we want remove an edge $\{i,j\}$ from $\tilde G$, that is, $\{i,j\}$ is not an edge of
$G$.  If this edge is a leaf of $\tilde G$, then one of its vertices, say $i$, has degree 1,
and $v_i$ is a marked point on a side of the $m$-gon, but not a vertex of the $m$-gon.
If we rescale the vector $v_i$, that is, replace it by the vector $\alpha\, v_i$, then the minor $\Delta_{ij}$
will be rescaled by the same factor $\alpha$, while all other minors corresponding to edges of $\tilde G$
will remain the same.  If we pick the factor $\alpha$ to be slightly smaller than $1$,
then this will make the minor $\Delta_{ij}$ smaller.  Actually, this argument shows that we can independently
rescale the minors for all leaves of $\tilde G$.

Now assume that $\{i,j\}$ is not a leaf of $\tilde G$.  Then $\tilde G$ contains two non-leaf edges $\{i,j\}$ and
$\{i,j'\}$ incident to $i$.  If $G$ does not contain both edges $\{i,j\}$ and $\{i,j'\}$ then we can also rescale
the vector $v_i$ by a factor $\alpha$ slightly smaller than 1.  This will make both minors $\Delta_{ij}$ and $\Delta_{ij'}$ smaller.
If $G$ does not contain the edge $\{i,j\}$ but contains the edge $\{i,j'\}$, then we
can slightly move the point $v_i$ along one of the two sides of the $m$-gon which is parallel
to the vector $v_{j'}$
towards the other vertex of this side of the $m$-gon.
This will make the minor $\Delta_{ij}$ smaller but preserve the minor $\Delta_{ij'}$.
This deformation of $v_i$ will also modify the minors for the leaves incident to $i$.
However, as we showed above, we can always
independently rescale the minors for all leaves of $\tilde G$ and make them equal to any values.

This shows that we can slightly decrease the values of minors for any subset of edges of $\tilde G$.
This finishes the proof.
\end{proof}

\section{Weakly separated sets and sorted sets}
\label{sec:WS_sorted}

In this section, we show how to extend triangulations and thrackles to the case of general $k$.
\medskip

As before, we assume that the vertices $1,\dots,n$ are arranged on the circle in the clockwise order.

\begin{definition}
Two $k$-element sets $I,J\in {[n]\choose k}$ are called {\it weakly separated\/} if their set-theoretic differences
$I\setminus J = \{a_1,\dots,a_r\}$
and $J\setminus I=\{b_1,\dots,b_r\}$ are separated from each other by some diagonal in the circle,
i.e., $a_1 < \cdots < a_s < b_1< \dots < b_r < a_{s+1} < \cdots < a_r$ (or the same inequalities with $a$'s
and $b$'s switched).

A subset of $[n]\choose k$ is called {\it weakly separated\/} if every two
elements in it are weakly separated.
\end{definition}

Weakly separated sets were originally introduced by
Leclerc-Zelevinsky \cite{LZ} in the study of quasi-commuting quantum minors.
It was conjectured in \cite{LZ} that all maximal (by containment) weakly
separated sets have the same number of elements (the {\it Purity Conjecture}),
and that they can be obtained from each other by a sequence of {\it mutations.}
The purity conjecture was proved independently by Danilov-Karzanov-Koshevoy
\cite{DKK} and in \cite{OPS}.

\cite{OPS} presented a bijection between maximal weakly separated sets
and {\it reduced plabic graphs.}  The latter appear in the study of the
positive Grassmannian \cite{Pos2}.  Leclerc-Zelevinsky's purity conjecture and
the mutation connectedness conjecture follow from the properties of plabic
graphs proved in \cite{Pos2}.

More precisely, it was shown in \cite{OPS}, cf. \cite{DKK}, that
any maximal by containment weakly separated subset of $[n] \choose k$ has
exactly $k(n-k)+1$ elements. We will talk more about the connection between weakly separated sets and plabic graphs in Section~\ref{sec:not_weakly_separated}.

\begin{definition}
Two $k$-element sets $I,J\in {[n]\choose k}$ are called {\it sorted\/} if their set-theoretic differences $I\setminus J =
\{a_1,\dots,a_r\}$
and $J\setminus I=\{b_1,\dots,b_r\}$ are interlaced on the circle, i.e., $a_1<b_1<a_2<b_2<\cdots<a_r<b_r$
(or the same inequalities with $a$'s and $b$'s switched).

A subset of $[n]\choose k$ is called {\it sorted\/} if every two elements in it are sorted.
\end{definition}

Sorted sets appear in the study of Gr\"obner bases \cite{Stu} and in the theory of {\it alcoved polytopes\/} \cite{LP}.
Any maximal (by containment) sorted subset in $[n]\choose k$ has exactly $n$ elements.
Such subsets were identified with simplices of the {\it alcoved triangulation\/}
of the hypersimplex $\Delta_{k,n}$, see~\cite{LP, Stu}.
The number of maximal sorted subsets in $[n]\choose k$ equals the {\it Eulerian number\/}
$A(n-1,k-1)$, that is, the number of permutations $w$ of size $n-1$ with exactly $k-1$ descents, $\mathrm{des}(w) = k-1$.
(Recall, that a {\it descent\/} in a permutation $w$ is an index $i$ such that $w(i)>w(i+1)$.)
An explicit bijection between sorted subsets in $[n]\choose k$ and permutations of size $n-1$
with $k-1$ descents was constructed in \cite{LP}.

\begin{remark}
For $k=2$, a pair $\{a,b\}$ and $\{c,d\}$ is weakly separated if the edges $\{a,b\}$ and $\{c,d\}$ of $K_n$ are non-crossing
or share a common vertex.  On the other hand, a pair $\{a,b\}$ and $\{c,d\}$ is sorted
if the edges $\{a,b\}$ and $\{c,d\}$ of $K_n$ are crossing or share a common vertex.
Thus maximal weakly separated subsets in $[n]\choose 2$ are exactly the
graphs corresponding to triangulations of the $n$-gon,
while sorted subsets in $[n]\choose 2$ are exactly thrackles discussed in Section~\ref{sec:k=2}.
\end{remark}

Here is our main result on arrangements of largest minors.

\begin{theorem}
A nonempty subset of $[n]\choose k$ is an arrangement of largest minors
in $Gr^+(k,n)$ if and only if it is a sorted subset.
Maximal arrangements of largest minors contain exactly $n$ elements.
The number of maximal arrangements of largest minors
in $Gr^+(k,n)$ equals the Eulerian number $A(n-1,k-1)$.
\label{thm:sorted}
\end{theorem}

Regarding arrangements of smallest minors, we will show the following.

%\begin{conjecture}
%A subset of $[n]\choose k$ is an arrangement of smallest minors in $Gr^+(k,n)$ if and only if it is a
%weakly separated subset.
%\label{conj:weakly_separated}
%\end{conjecture}

%In particular, this conjecture implies that any maximal arrangement of smallest minors in $Gr^+(k,n)$
%contains exactly $k(n-k)+1$ minors.

%We proved the $\Leftarrow$ direction Conjecture~\ref{conj:weakly_separated} (for arbitrary $k,n$),
%and also the $\Rightarrow$ direction of this conjecture for $k=1,2,3,n-1,n-2,n-3$.

\begin{theorem}
Any nonempty weakly separated set in $[n]\choose k$ is an arrangement of smallest minors in $Gr^+(k,n)$.
\label{thm:WSeparated}
\end{theorem}

\begin{theorem}
For $k=1,2,3, n-1, n-2, n-3$,
a nonempty subset of $[n]\choose k$ is an arrangement of smallest minors in $Gr^+(k,n)$ if and only if it is
a weakly separated subset.
Maximal arrangements of smallest minors contain exactly $k(n-k)+1$ elements.
%of $[n]\choose k$.
\label{thm:weakly_separated_123}
\end{theorem}

Note that the symmetry $Gr(k,n)\simeq Gr(n-k,n)$ implies that the cases $k=1,2,3$ are
equivalent to the cases $k=n-1,n-2,n-3$.

% Theorem~\ref{thm:weakly_separated_123}
%also holds for $k\geq n-3$.

In Section~\ref{sec:not_weakly_separated}, we will construct, for $k\geq 4$,
examples of arrangements of smallest minors which are not weakly separated.
We will describe the conjectural structure of such arrangements (Conjecture~\ref{conjecture:minimal_minors})
and prove it for $k=4, 5, n-4, n-5$.
%and give a conjecture for the general case.

These examples show that it is not true in general that all maximal (by containment) arrangements of
smallest minors are weakly separated.  However, the following  conjecture says
that maximal by {\it size\/}
%(not by containment)
arrangements of smallest minors are exactly maximal weakly separated sets.

%Thus Conjecture~\ref{conj:weakly_separated} is true for these values of $k$.

%\medskip

%Recall that any maximal (by containment) weakly separated set in $[n]\choose k$ has  $k(n-k)+1$
%elements, see \cite{OPS}.
% cf. , DKK}.

\begin{conjecture}
Any arrangement of smallest minors in $Gr^+(k,n)$ contains at most $k(n-k)+1$ elements.
Any arrangement of smallest minors in $Gr^+(k,n)$ with $k(n-k)+1$ elements
is a (maximal) weakly separated set in $[n]\choose k$.
\label{conj:max_smallest}
\end{conjecture}

%\begin{remark}
%As we mentioned above, it is not true that any arrangement of smallest minors is weakly separated.
%%In Section~\ref{sec:not_weakly_separated},
%%we will construct example of such arrangements.
%Conjecture~\ref{conj:max_smallest} says that maximal by size (not by containment) arrangements of
%smallest minors are exactly maximal weakly separated sets.
%\end{remark}

In order to prove
Theorems~\ref{thm:sorted} and~\ref{thm:weakly_separated_123} in one direction ($\Rightarrow$),
%and also Theorem~\ref{thm:WSeparated},
%and Conjecture~\ref{conj:weakly_separated} in one direction ($\Rightarrow$),
we need to show that, for a pair of elements $I$ and $J$ in an arrangement of largest (smallest) minors,
the pair $I,J$ is sorted (weakly separated).

In order to prove these claims in the other direction ($\Leftarrow$)
and also Theorem~\ref{thm:WSeparated}, it is enough to construct, for each sorted (weakly separated)
subset, matrices with the corresponding collection of equal largest (smallest) minors.

\medskip

In Section~\ref{sec:inequalities_products_minors}, we discuss
inequalities between products of minors and use them to prove
Theorems~\ref{thm:sorted} and~\ref{thm:weakly_separated_123} in one direction ($\Rightarrow$).
That is, we show arrangements of largest (smallest) minors should be sorted (weakly separated).
%our method of
%proving that the subsets of equal minors should be sorted or weakly separated,
%and present the proof of the first direction of
In Section~\ref{sec:cluster}, we prove Theorem~\ref{thm:WSeparated} (and hence the other direction ($\Leftarrow$) of
Theorem~\ref{thm:weakly_separated_123}) using the theory
of cluster algebras.
In Section~\ref{sec:construction_largest}, we prove the other direction ($\Leftarrow$) of
Theorem~\ref{thm:sorted} using the theory of alcoved polytopes \cite{LP}.

\section{Inequalities for products of minors}
\label{sec:inequalities_products_minors}

As we discussed in Section~\ref{sec:k=2}, in the case $k=2$, in one direction, our
results follow from the inequalities for products of minors
of the form $\Delta_{ac} \, \Delta_{bd} > \Delta_{ab}\, \Delta_{cd}$ and
$\Delta_{ac} \, \Delta_{bd} > \Delta_{ad}\, \Delta_{bc}$, for $a<b<c<d$.

There are more general inequalities of this form found by Skandera \cite{Ska}.  \medskip

For $I,J\in{[n]\choose k}$ and an interval $[a,b]:=\{a,a+1,\dots,b\}\subset [n]$, define
$$
r(I,J; a,b) = | \left(|(I\setminus J)\cap [a,b]| - |(J\setminus I)\cap [a,b]|\right)|.
$$
Notice that the pair $I,J$ is sorted if and only if $r(I,J;a,b)\leq 1$ for all $a$ and $b$.
In a sense, $r(I,J; a,b)$ is a measure of ``unsortedness'' of the pair $I,J$.

\begin{theorem} {\rm Skandera \cite{Ska}} \
For $I,J,K,L\in{[n]\choose k}$, the products of the Pl\"ucker coordinates satisfy the inequality
$$
\Delta_I \, \Delta_J  \geq \Delta_K\,\Delta_L
$$
for all points of the nonnegative Grassmannian $Gr^{\geq}(k,n)$,
if and only if the multiset union of $I$ and $J$ equals to the multiset union of $K$ and $L$; and,
for any interval $[a,b]\subset [n]$, we have
$$
r(I,J;a,b)\leq r(K,L;a,b).
$$
\label{thm:skandera}
\end{theorem}

\begin{remark}
Skandera's result \cite{Ska} is given in terms of minors (of arbitrary sizes) of
totally nonnegative matrices.  Here we reformulated this result in terms of
Pl\"ucker coordinates (i.e., {\it maximal\/} minors) on the nonnegative
Grassmannian using the map $\phi:\Mat(k,n-k)\to Gr(k,n)$ from
Section~\ref{sec:from_Mat_to_Gr}.  We also used a different notation to express
the condition for  the sets $I,J,K,L$.
We leave it as an exercise for the reader to check that above Theorem~\ref{thm:skandera}
is equivalent to \cite[Theorem 4.2]{Ska}.
\end{remark}

Roughly speaking, this theorem says that the product
of minors $\Delta_I \, \Delta_J$ should be ``large'' if the pair $I, J$  is ``close'' to being sorted;
and the product should be ``small'' if the pair $I, J$  is ``far'' from being sorted.

% is weakly separated.

Actually, we need a similar result with {\it strict\/} inequalities.  It also
follows from results of Skandera's work \cite{Ska}.

\begin{theorem} {\rm cf.\ \cite{Ska}} \
Let $I,J,K,L\in{[n]\choose k}$ be subsets such that $\{I,J\}\ne\{K,L\}$.
The products of the Pl\"ucker coordinates satisfy the strict inequality
$$
\Delta_I \, \Delta_J  > \Delta_K\,\Delta_L
$$
for all points of the positive Grassmannian $Gr^{+}(k,n)$,
if and only if the multiset union of $I$ and $J$ equals to the multiset union of $K$ and $L$; and,
for any interval $[a,b]\subset [n]$, we have
$$
r(I,J;a,b)\leq r(K,L;a,b).
$$
%Assume that the multiset union of $I$ and $J$ equals to the multiset union of $K$ and $L$; and,
%for any interval $[a,b]\subset [n]$, we have
%$$
%r(I,J;a,b)\leq r(K,L;a,b).
%$$
%Then we have the strict inequality
%$$
%\Delta_I \, \Delta_J  > \Delta_K\,\Delta_L
%$$
%for all points of the positive Grassmannian $Gr^+(k,n)$.
\label{thm:Skandera_strict}
\end{theorem}

\begin{proof} In one direction ($\Rightarrow$), the result directly follows from
Theorem~\ref{thm:skandera}.  Indeed, the nonnegative Grassmannian $Gr^{\geq} (k,n)$ is the
closure of the positive Grassmannian $Gr^+(k,n)$.  This implies that, if $\Delta_I \, \Delta_J  > \Delta_K\,\Delta_L$
on $Gr^{+}(k,n)$, then $\Delta_I \, \Delta_J  \geq \Delta_K\,\Delta_L$ on $Gr^{\geq}(k,n)$.

Let us show how to prove the other direction ($\Leftarrow$) using results of \cite{Ska}.
Every totally positive (nonnegative) matrix $A=(a_{ij})$ can be obtained from an acyclic directed planar
network with positive (nonnegative) edge weights, cf.\ \cite{Ska}.   The matrix entries $a_{ij}$ are sums
of products of edge weights over directed paths from the $i$th source to the $j$th sink of a network.

Theorem~\ref{thm:skandera} implies the weak inequality $\Delta_I \, \Delta_J  \geq  \Delta_K\,\Delta_L$.
Moreover, \cite[Corollary 3.3]{Ska} gives a combinatorial interpretation of the difference
$\Delta_I \, \Delta_J  - \Delta_K\,\Delta_L$ as a weighted sum over {\it certain} families of
directed paths in a network.

In case of totally {\it positive} matrices, all edge weights, as well as weights of all families of paths,
are strictly positive.   It follows that, if $\Delta_I \, \Delta_J  - \Delta_K\,\Delta_L = 0$, for
{\it some\/} point of $Gr^+(k,n)$,  then there are no families of paths satisfying the condition of
\cite[Corollary 3.3]{Ska}, and thus $\Delta_I \, \Delta_J  - \Delta_K\,\Delta_L = 0$, for
{\it all\/} points of $Gr^+(k,n)$.

However, the only case when we have the equality
$\Delta_I \, \Delta_J = \Delta_K\,\Delta_L$ for all points of $Gr^+(k,n)$
is when $\{I,J\} = \{K,L\}$.
\end{proof}

Theorem~\ref{thm:Skandera_strict} implies the following corollary.

\begin{definition}
For $I,J\in{[n]\choose k}$, define their {\it sorting\/}  $I', J'$ by taking
the multiset union $I\cup J = \{a_1 \leq a_2 \leq \cdots \leq a_{2k}\}$ and
setting $I' = \{a_1,a_3,\dots, a_{2k-1}\}$ and $J' = \{a_2, a_4,\dots,a_{2k}\}$.
\label{def:sorting}
\end{definition}

\begin{corollary}\label{cor:Skandera}
Let $I,J\in{[n]\choose k}$ be a pair which is not sorted,
and let $I',J'$ be the sorting of the pair $I, J$.   Then we have the
strict inequality $ \Delta_{I'} \, \Delta_{J'} > \Delta_I \Delta_J$
for points of the positive Grassmannian $Gr^+(k,n)$.
\end{corollary}

\begin{proof}
We have $r(I',J';a,b)\leq r(I,J;a,b)$.
\end{proof}

This result easily implies one direction of Theorem~\ref{thm:sorted}.

\begin{proof}[Proof of Theorem~\ref{thm:sorted} in the $\Rightarrow$ direction.]
We need to show that a pair $I, J$, which is not sorted, cannot belong to an arrangement of largest minors.
If a pair $I,J$, which is not sorted, belongs to an arrangement of largest minors, we have
$\Delta_I = \Delta_J=a$, and the inequality $\Delta_{I'} \, \Delta_{J'} > \Delta_I \Delta_J$ implies that
$\Delta_{I'}$ or $\Delta_{J'}$ is greater than $a$.
\end{proof}

Using a similar argument, we can also prove the same direction of Theorem~\ref{thm:weakly_separated_123}.

\begin{proof}[Proof of Theorem~\ref{thm:weakly_separated_123} in the $\Rightarrow$ direction.]
The Grassmannian $Gr(k,n)$ can be identified with $Gr(n-k,n)$ so that the Pl\"ucker coordinates $\Delta_I$
in $Gr(k,n)$ map to the Pl\"ucker coordinates $\Delta_{[n]\setminus I}$ in $Gr(n-k,n)$.
This duality reduces the cases $k=n-1,n-2,n-3$ to the cases $k=1,2,3$.

The case $k=1$ is trivial.  The case $k=2$ is covered by
Theorem~\ref{thm:non_crossing}.
It remains to prove the claim in the case
$k=3$.

We need to show that a pair $I,J\in {[n]\choose 3}$, which is not weakly separated, cannot belong to an
arrangement of smallest minors in the positive Grassmannian $Gr^+(3,n)$.

If $|I\cap J| \geq 2$, then $I$ and $J$ are weakly separated.   If $|I\cap J|=1$, say $I\cap J=\{e\}$,
then the result follows from the 3-terms Pl\"ucker relation
$$
\Delta_{\{a,c,e\}} \, \Delta_{\{b,d,e\}} = \Delta_{\{a,b,e\}}\,\Delta_{\{c,d,e\}}+
\Delta_{\{a,d,e\}}\,\Delta_{\{b,c,e\}},\quad
\textrm{for } a<b<c<d,
$$
as in the $k=2$ case (Theorem~\ref{thm:non_crossing}).

Thus we can assume that $I\cap J=\emptyset$.  Without loss of generality, we can assume that
$I\cup J=\{1,2,3,4,5,6\}$.  Up to the cyclic symmetry, and up to switching $I$ with $J$, there are only 2 types of
pairs $I, J$ which are not weakly separated:
$$
I=\{1,3,5\}, \ J=\{2,4,6\}\qquad \textrm{and} \qquad I=\{1,2,4\}, \ J=\{3,5,6\}.
$$

In both cases, we have strict Skandera's inequalities (Theorem~\ref{thm:Skandera_strict}):
$$
\begin{array}{l}
\Delta_{\{1,3,5\}}\, \Delta_{\{2,4,6\}} > \Delta_{\{1,2,3\}}\, \Delta_{\{4,5,6\}}\\[.1in]
\Delta_{\{1,2,4\}}\, \Delta_{\{3,5,6\}} > \Delta_{\{1,2,3\}}\, \Delta_{\{4,5,6\}}.
\end{array}
$$

This shows that, if $\Delta_I = \Delta_J=a$, then there exists $\Delta_K<a$.
Thus a pair $I,J$, which is not weakly separated, cannot belong to an arrangement of smallest minors.
\end{proof}

\section{Cluster algebra on the Grassmannian}
\label{sec:cluster}

In this section we prove Theorem~\ref{thm:WSeparated}
using cluster algebras.

\medskip

%The construction of points in $Gr^+(k,n)$ with maximal arrangements of smallest minors
%is based on the results of \cite{Pos2, OPS}.
The following statement follows from results of \cite{OPS, Pos2}.

\begin{theorem}
Any maximal weakly separated subset $S\subset {[n]\choose k}$ corresponds to $k(n-k)+1$
algebraically independent Pl\"ucker coordinates $\Delta_I$, $I\in S$.
Any other Pl\"ucker coordinate $\Delta_J$ can be uniquely expressed in terms of
the $\Delta_I$, $I\in S$, by a subtraction-free rational expression.
\label{thm:cluster}
\end{theorem}

In the following proof we use plabic graphs from \cite{Pos2}.  See
Section~\ref{sec:not_weakly_separated} below for more details on plabic graphs.

\begin{proof}
In \cite{OPS}, maximal weakly subsets of ${[n]\choose k}$ were identified with labels
of faces of reduced plabic graphs for the top cell of $Gr^+(k,n)$.
(This labelling of faces is described in Section~\ref{sec:not_weakly_separated} of the current
paper in the paragraph after Definition~\ref{def:decorated_strand_perm}.)

According to \cite{Pos2}, all reduced plabic graphs for top cell can be obtained from each other
by a sequence of {\it square moves,} that correspond to {\it mutations\/} of weakly separated sets.

A mutation has the following form.  For $1\leq a<b<c<d\leq n$ and $R\in {[n]\choose k-2}$ such
that $\{a,b,c,d\}\cap R=\emptyset$, if a maximal weakly separated set $S$ contains
$\{a,b\}\cup R$,
$\{b,c\}\cup R$,
$\{c,d\}\cup R$,
$\{a,d\}\cup R$, and
$\{a,c\}\cup R$, then we can replace $\{a,c\}\cup R$ in $S$ by $\{b,d\}\cup R$.
In terms of the Pl\"ucker coordinates $\Delta_I$, $I\in S$, a mutation means that we replace
$\Delta_{\{a,c\}\cup R}$ by
$$
\Delta_{\{b,d\}\cup R} = {
\Delta_{\{a,b\}\cup R}\, \Delta_{\{c,d\}\cup R} +
\Delta_{\{a,d\}\cup R}\, \Delta_{\{b,c\}\cup R}
\over \Delta_{\{a,c\}\cup R}}.
$$

Since any $J\in{[n]\choose k}$ appears as a face label of some plabic graph for the top cell,
it follows that any Pl\"ucker coordinate $\Delta_J$  can be expressed in terms
the $\Delta_I$, $I\in S$, by a sequence of rational subtraction-free transformation of this form.

The fact that the $\Delta_I$, $I\in S$, are algebraically independent follows from dimension consideration.
Indeed, we have $|S| = k(n-k)+1$, and all Pl\"ucker coordinates (which are projective coordinates on the Grassmannian
$Gr(k,n)$)
can be expressed in terms of the $\Delta_I$, $I\in S$.   If there was an algebraic relation among the $\Delta_I$, $I\in S$,
it would imply that $\dim Gr(k,n)< k(n-k)$.   However, $\dim Gr(k,n) = k(n-k)$.
\end{proof}

This construction fits in the general framework of Fomin-Zelevinsky's {\it
cluster algebras} \cite{FZ1}.
For a maximal weakly separated set $S\subset {[n]\choose k}$,
the Pl\"ucker coordinates $\Delta_I$, $I\in S$,
form an initial seed of the cluster algebra associated with the Grassmannian.
It is the cluster algebra whose quiver is the dual graph of the
plabic graph associated with $S$.
This cluster algebra was studied by Scott \cite{Sc}.

%This theorem is related to Fomin-Zelevinsky's theory of {\it cluster algebras} \cite{FZ1}.
%The Pl\"ucker coordinates $\Delta_I$, $I\in S$ form an initial cluster
%of the cluster algebra associated with the Grassmannian $Gr(k,n)$.
%Other maximal weakly separated subsets in $[n]\choose k$ are obtained from the initial one
%by a sequence of {\it mutations.}

According to the general theory of cluster algebras, the subtraction-free
expressions mentioned in Theorem~\ref{thm:cluster} are actually Laurent
polynomials, see \cite{FZ1}.  This property is called the {\it Laurent phenomenon}.
In \cite{FZ1}, Fomin and Zelevinsky
conjectured that these Laurent polynomials have positive integer coefficients.
This conjecture was recently proven by Lee and Schiffler in \cite{LS}, for
skew-symmetric cluster algebras.  Note that the cluster algebra associated with the Grassmannian
$Gr(k,n)$ is skew-symmetric.

The Laurent phenomenon and the result of Lee-Schiffler \cite{LS} imply the following claim.

\begin{theorem}
The rational expressions from Theorem~\ref{thm:cluster} that express the $\Delta_J$ in terms
of the $\Delta_I$, $I\in S$, are Laurent polynomials with nonnegative integer coefficients
that contain at least 2 terms.
\label{thm:cluster_Laurent}
\end{theorem}

Theorem~\ref{thm:cluster} implies that any maximal weakly separated subset $S$ uniquely defines a point $A_S$
in the positive Grassmannian $Gr^+(k,n)$ such that the Pl\"ucker coordinates $\Delta_I$, for all $I\in S$,
are equal to each other. Moreover, Theorem~\ref{thm:cluster_Laurent} implies that
all other Pl\"ucker coordinates $\Delta_J$ are strictly greater than the $\Delta_I$, for
$I\in S$.    This proves Theorem~\ref{thm:WSeparated}, (and hence the other
direction ($\Leftarrow$) of Theorem~\ref{thm:weakly_separated_123}).

We can now reformulate Conjecture~\ref{conj:max_smallest} as follows.

\begin{conjecture}
Any point in $Gr^+(k,n)$ with a maximal (by size) arrangement of smallest equal minors has the form $A_S$,
for some maximal weakly separated subset $S\subset {[n]\choose k}$.
\end{conjecture}

\section{Constructions of matrices for arrangements of largest minors}
\label{sec:construction_largest}

In this section, we prove the other direction ($\Leftarrow$) of Theorem~\ref{thm:sorted}.
In the previous sections, we saw that the points in $Gr^+(k,n)$
with a maximal arrangement of smallest equal minors have a very rigid structure.
On the other hand, the cardinality of a maximal arrangement of {\it largest\/} minors is $n$, which is much smaller
than the conjectured cardinality $k(n-k)+1$ of a maximal arrangement of smallest minors.
Maximal arrangements of largest minors impose fewer conditions on points of $Gr^+(k,n)$ and
have much more flexible structure.
Actually, one can get {\it any\/} maximal arrangement of largest minors from {\it any} point of $Gr^+(k,n)$ by the torus action.

The {\it ``positive torus''\/}
$\R_{>0}^n$ acts on the positive Grassmannian $Gr^+(k,n)$ by rescaling the coordinates in $\R^n$.
(The group $\R_{>0}^n$ is the {\it positive part\/} of the complex torus $(\C\setminus\{0\})^n$.)
In terms of $k\times n$ matrices this action is given by rescaling the columns of the matrix.

%}\footnote{It might be a little strange to call $\R_{>0}^n$ the ``positive torus'', since
%topologically this contractible set is quite far from been a torus.
%However, $\R_{>0}^n$ is the {\it positive part\/} of the complex torus
%$(\C\setminus\{0\})^n$.}

\begin{theorem}
{\rm (1)}
For any point $A$ in $Gr^+(k,n)$ and any maximal sorted subset $S\subset {[n]\choose k}$, there is a unique
point $A'$ of $Gr^+(k,n)$ obtained from $A$ by the torus action (that is, by rescaling the columns of the $k\times n$
matrix $A$) such that the Pl\"ucker coordinates $\Delta_I$, for all $I\in S$, are equal to each other.

\medskip
\noindent
{\rm (2)}
All other Pl\"ucker coordinates $\Delta_{J}$, $J\not\in S$,
for the point $A'$ are strictly less than the $\Delta_I$, for $I\in S$.
\label{thm:torus_action}
\end{theorem}

The proof of this result is based on geometric techniques of alcoved polytopes and affine
Coxeter arrangements developed in \cite{LP}.

\medskip

Before presenting the proof, let us give some examples of $3\times n$ matrices $A=[v_1,v_2,\dots,v_n]$ with maximal arrangements of largest equal minors.
Here $v_1,\dots,v_n$ are 3-vectors.  Projectively, we can think about the 3-vectors $v_i$ as points in the (projective) plane.
More precisely, let $P\simeq\R^2$
be an affine plane in $\R^3$ that does not pass through the origin $0$.  A point $p$ in the plane $P$ represents
the 3-vector $v$ from the origin $0$ to $p$.  A collection of points $p_1,\dots,p_n\in P$ corresponds to an element
$A=[v_1,\dots,v_n]$ of the {\it positive\/} Grassmannian $Gr^+(3,n)$ if and only if the points $p_1,\dots,p_n$ form vertices
of a {\it convex\/} $n$-gon with vertices labelled in the clockwise order.

Let us now assume that the $n$-gon formed by the points $p_1,\dots,p_n$ is a regular $n$-gon.
Theorem~\ref{thm:torus_action} implies that it is always possible to uniquely rescale
(up to a common factor) the corresponding $3$-vectors by some positive scalars
$\lambda_i$ in order to get any sorted subset in $[n]\choose 3$.
Geometrically, for a triple $I=\{i,j,r\}$, the minor $\Delta_I$ equals the area of the triangle
with the vertices $p_i$, $p_j$, $p_r$ times the product of the scalar factors $\lambda_i$, $\lambda_j$, $\lambda_r$
(times a common factor which can be ignored).
We want to make the largest area of such rescaled triangles to repeat as many times as possible.

\begin{example}
For the regular pentagon, there are the Eulerian number $A(4,2)=11$ rescalings
of vertices that give maximal sorted subsets in $[5]\choose 3$.  For the regular hexagon there are
$A(5,2)=66$ rescalings.  Figures~\ref{fig:pentagon} and~\ref{fig:hexagon} show all these rescalings
up to rotations and reflections.
\begin{figure}[h]
\qquad \qquad
\includegraphics[height=0.8in,width=.8in]{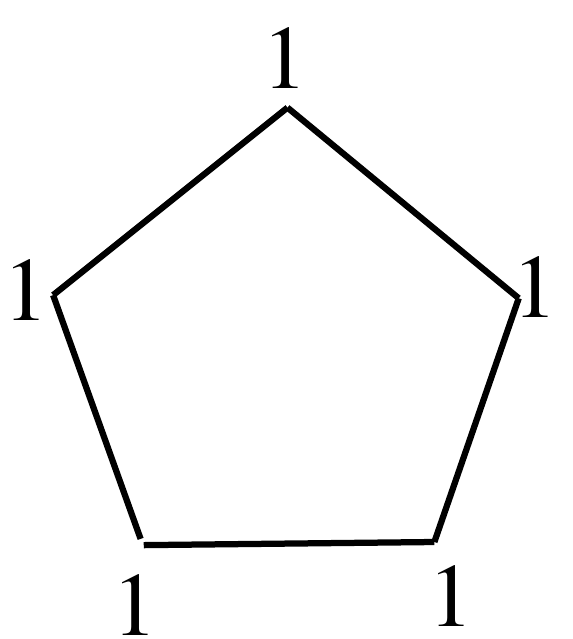}
\qquad \qquad
\includegraphics[height=0.8in,width=.8in]{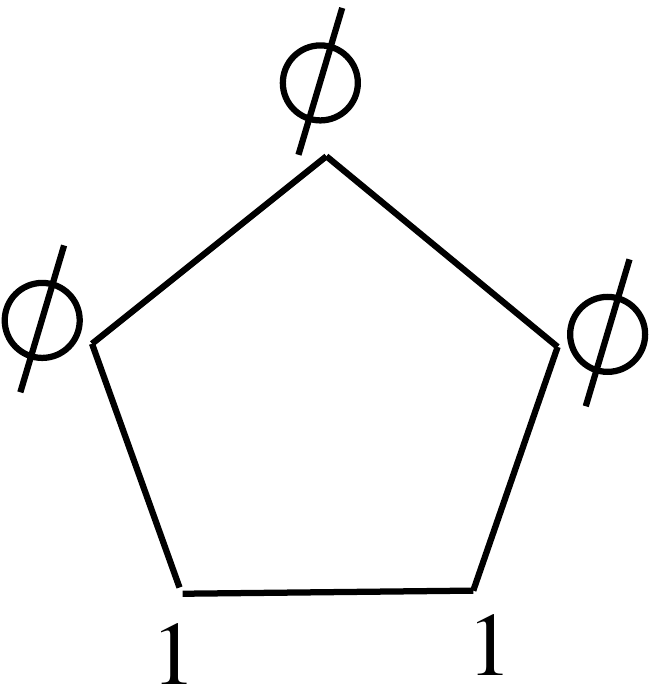}
\qquad \qquad
\includegraphics[height=0.8in,width=.8in]{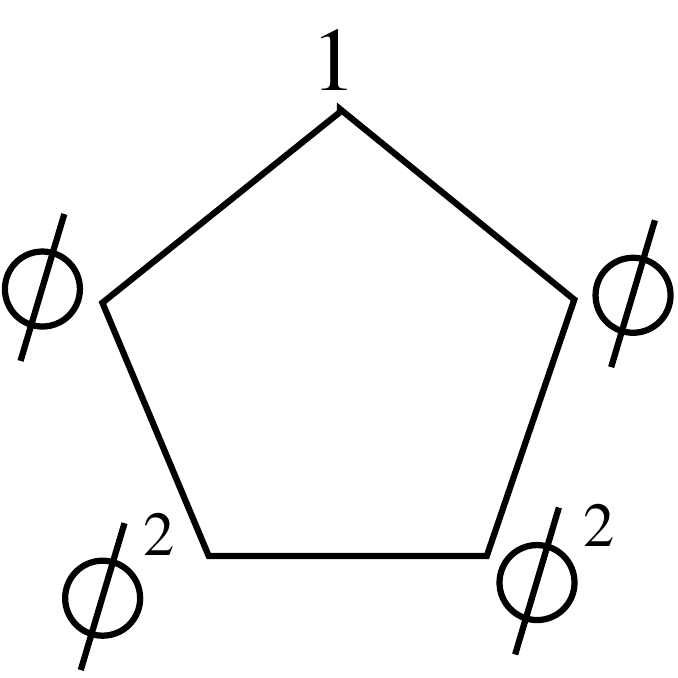}
\qquad \qquad
\caption{For the regular pentagon, there are the Eulerian number $A(4,2)=11$ rescalings that
give maximal sorted subsets in $[5]\choose 3$.
In the first case, all the scalars $\lambda_i$ are $1$.  In the second case,
the $\lambda_i$ are $1,1,\phi,\phi,\phi$.
Here $\phi = (1+\sqrt{5})/2$ is the golden ratio.  (There are 5 rotations of this case.)
In the last case, the $\lambda_i$ are $1,\phi,\phi^2,\phi^2,\phi$. (Again, there are 5 rotations.)
In total, we get $1+ 5 + 5 = 11$ rescalings.
}
\label{fig:pentagon}
\end{figure}
\begin{figure}[ht]
\qquad \qquad
\includegraphics[height=0.8in,width=.8in]{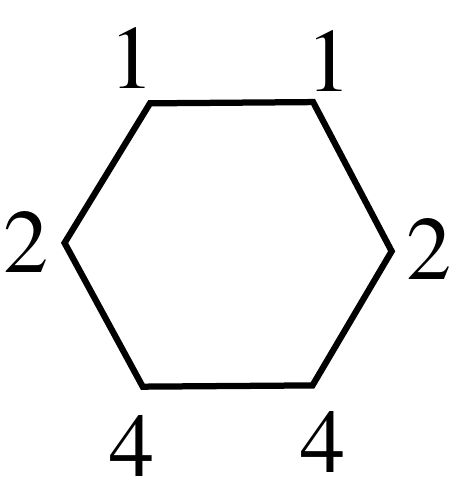}
\qquad \qquad
\qquad \qquad
\includegraphics[height=0.8in,width=1in]{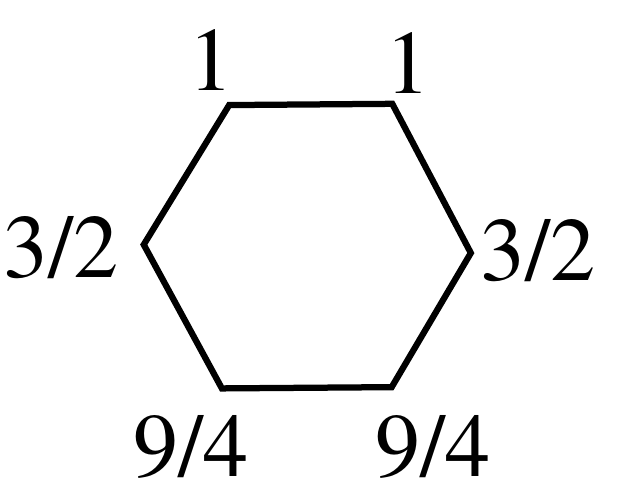}
\qquad \qquad
\qquad \qquad
\includegraphics[height=0.8in,width=.8in]{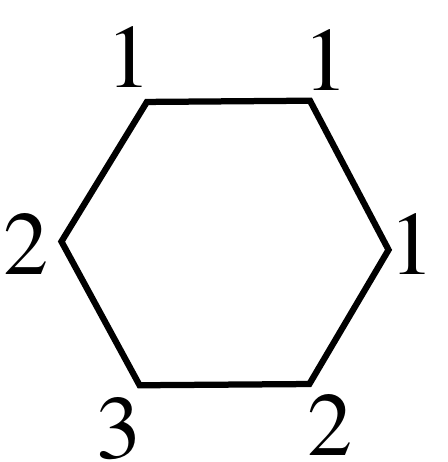}
\qquad \qquad
\qquad \qquad
\includegraphics[height=0.8in,width=0.9in]{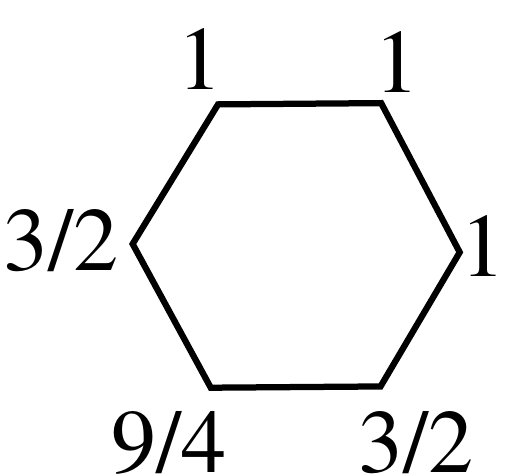}
\qquad \qquad
\qquad \qquad
\includegraphics[height=0.8in,width=.85in]{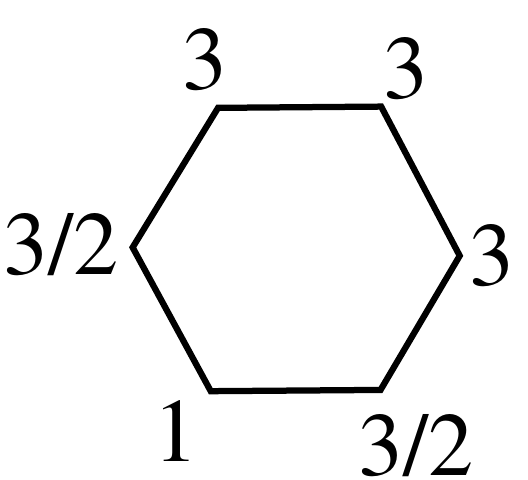}
\qquad \qquad
\qquad \qquad
\includegraphics[height=0.8in,width=1in]{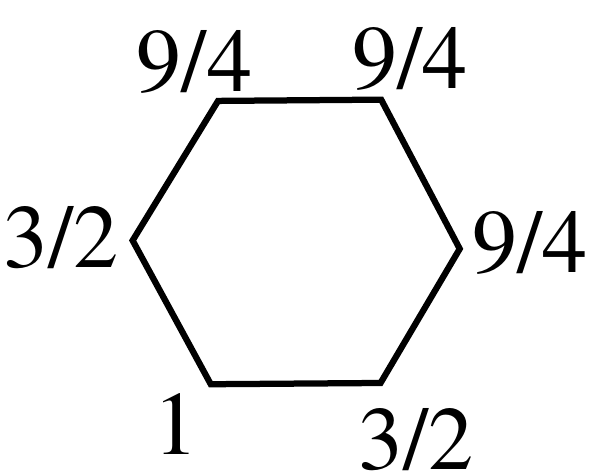}
\qquad \qquad
\qquad \qquad
\includegraphics[height=0.8in,width=1in]{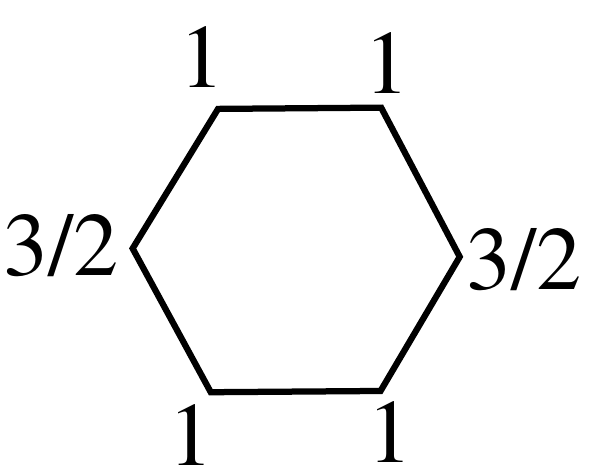}
\qquad \qquad
\qquad \qquad
\includegraphics[height=0.8in,width=.7in]{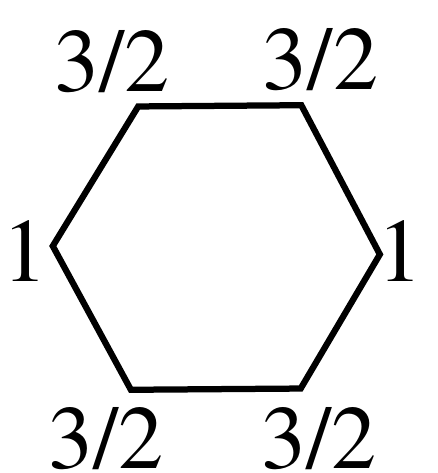}
\qquad \qquad
\qquad \qquad
\includegraphics[height=0.8in,width=0.9in]{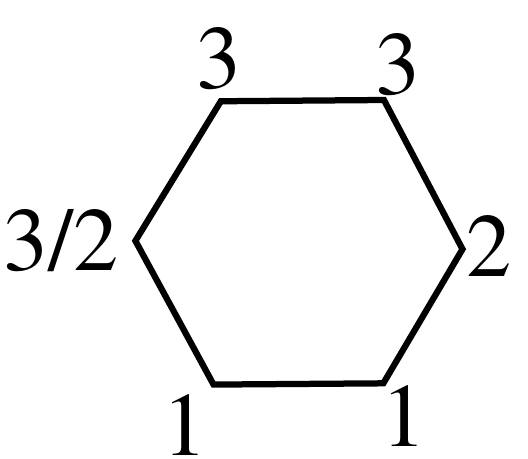}
\qquad \qquad
\qquad \qquad
\includegraphics[height=0.8in,width=0.9in]{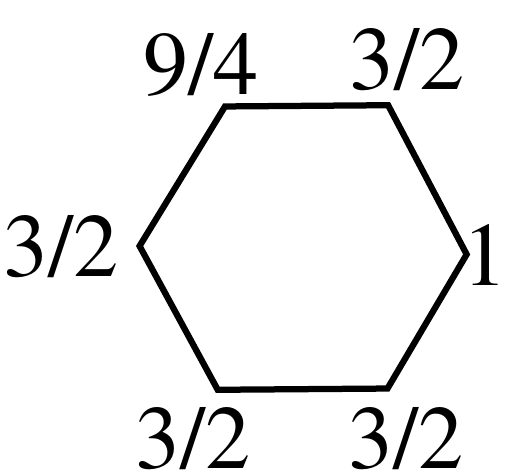}
\qquad \qquad
\qquad \qquad
\caption{For the regular hexagon, there are 10 types of allowed rescalings (up to rotations and reflections) shown
in this figure.  In total, we get the Eulerian number $A(5,2) = 6+6+6+6+6+6+3+3+12+12 = 66$ rescalings.}
\label{fig:hexagon}
\end{figure}
\end{example}

\medskip
Our proof of Theorem~\ref{thm:torus_action} relies on results from \cite{LP}
about hypersimplices and their alcoved triangulations.  Let us first summarize these results.

The hypersimplex $\Delta_{k,n}$ is the $(n-1)$-dimensional polytope
$$
\Delta_{k,n}:=\{(x_1,\ldots,x_n) \mid 0 \leq
x_1 ,\ldots,x_n \leq 1;\, x_1+x_2+\ldots+x_n = k\}.
$$

Let $e_1,\dots,e_n$ be the coordinate vectors in $\R^n$.
For $I\in{[n]\choose k}$, let $e_I = \sum_{i\in I} e_i$ denote the $01$-vector with $k$ ones in positions $I$.
For a subset $S\subset {[n]\choose k}$, let $P_S$ be the polytope defined as the convex hull of $e_I$, for $I\in S$.
Equivalently, $P_S$ has the vertices $e_I$, $I\in S$.
The polytope $P_S$ lies in the affine hyperplane $H=\{x_1+\cdots + x_n =k\}\subset\R^n$.

For $1 \leq i \leq j \leq n$ and an integer $r$, let $H_{i,j,r}$ be the affine
hyperplane $\{x_i+x_{i+1}\cdots + x_j =r\}\subset\R^n$.

\begin{theorem} {\rm \cite{LP}, cf.\ \cite{Sta, Stu}} \
{\rm (1)}  The hyperplanes $H_{i,j,r}$ subdivide the hypersimplex $\Delta_{k,n}$ into
simplices.  This forms a triangulation of the hypersimplex.

\medskip
\noindent
{\rm (2)} Simplices (of all dimensions) in this triangulation of $\Delta_{k,n}$ are in bijection with
sorted sets in $[n]\choose k$.  For a sorted set $S$, the corresponding simplex is $P_S$.
%has the vertices $e_I$, for $I\in S$.

\medskip
\noindent
{\rm (3)}  There are the Eulerian number $A(n-1,k-1)$ of $(n-1)$-dimensional simplices $P_S$ in this triangulation.  They correspond
to $A(n-1,k-1)$ maximal sorted sets $S$ in $[n]\choose k$.
In particular, maximal sorted sets in $[n]\choose k$ have exactly $n$ elements.
\label{thm:LP}
\end{theorem}

The following lemma proves the first part of Theorem~\ref{thm:torus_action}.

\begin{lemma}\label{lemma:torus_action_first_part}
Let $A$ be a point in $Gr^+(k,n)$, and let $S\subset {[n]\choose k}$ be a maximal sorted subset. There is a unique
point $A'$ of $Gr^+(k,n)$ obtained from $A$ by the torus action, such that the Pl\"ucker coordinates $\Delta_I$, for all $I\in S$, are equal to each other.
\end{lemma}
\begin{proof}
Let $t_1,t_2,\ldots,t_n>0$ be a collection of $n$ positive real variables, and let $A'$ be a
matrix that is obtained from $A$ by multiplying the $i$-th column of $A$ by
$t_i$, for each $1 \leq i \leq n$. We will show that we can choose positive
variables $\{t_i\}_{i=1}^{n}$ in such a way that $\Delta_I(A')=1$ for all $I
\in S$. For each $I \in S$, we have
$\Delta_I(A')=\Delta_I(A)\prod_{i \in I}t_i$.

In order to find $A'$ with $\Delta_I(A')=1$ for $I\in S$,
we need to solve the following system of $n$ linear equations
(obtained by taking the logarithm of $\Delta_I(A')=\Delta_I(A)\prod_{i \in I}t_i$):
$$
\sum_{i \in I} z_i = -b_I, \textrm{ for every }I \in S,
$$
where $z_i=\log(t_i)$ and $b_I=\log(\Delta_I(A))$.

This $n\times n$ system has a unique solution $(z_1,\dots,z_n)$ because,
according to Theorem~\ref{thm:LP},
the rows of its matrix are exactly the vertices
of the symplex $P_S$, so the matrix of the system is invertible.

The positive numbers $t_i = e^{z_i}$, $i=1,\dots,n$, give us the needed rescaling constants.
%Since
%$t_i=e^{z_i}$ then $t_i>0$ for all $1\leq i \leq n$, and hence $A'$ is a point
%in $Gr^+(k,n)$ for which $\Delta_I(A')=1$ for all $I \in S$.
%
%Denote the
%elements of $S$ by $I_1,I_2,\ldots,I_n$, and let $M$ be an $n\times n$ matrix in
%which the $r$-th row equals $e_{I_r}$.  The system above has a unique solution
%iff the matrix $M$ is invertible. The rows of $M$ are just the coordinate
%vectors corresponding to the vertices of $P_S$, and hence they form vertices of
%a simplex (which is a unit simplex that appears in the alcoved triangulation of
%$\Delta_{k,n}$, \cite[section 2]{LP}). Therefore, they are linearly independent
%and hence form a basis for $\R^n$. Thus $M$ is invertible, and hence there
%exists a unique solution $z=\{z_i\}_{i=1}^n$ for the linear system above. Since
%$t_i=e^{z_i}$ then $t_i>0$ for all $1\leq i \leq n$, and hence $A'$ is a point
%in $Gr^+(k,n)$ for which $\Delta_I(A')=1$ for all $I \in S$.
\end{proof}

In order prove the second part of Theorem~\ref{thm:torus_action}, let us
define a distance $d(S,J)$  between a maximal sorted set $S$ and
some $J \in {[n]\choose k}$.  Such a function will enable us to provide an
inductive proof.

Let us say that a hyperplane  $H_{i,j,r}=\{x_i+x_{i+1}\cdots + x_j =r\}$
{\it separates\/} a simplex $P_S$ and a point $e_J$ if $P_S$ and $e_J$ are in the two disjoint halfspaces
formed by $H_{i,j,r}$.  Here we allow $H_{i,j,r}$ to
touch the simplex $P_S$ along the boundary, but the point $e_J$ should not lie on
the hyperplane.

For $J \in {[n]\choose k}$, $1 \leq i \leq j \leq n$, let
$$
d_{ij}(S,J):=\#\{r \mid \textrm{ the hyperplane } H_{i,j,r} \textrm{ separates } P_S
\textrm{ and } e_J\}.
$$
Define the {\it distance\/} between $J$ and $S$ as
$$
d(S,J):=\sum\limits_{1 \leq i \leq j \leq n}d_{ij}(S,J).
$$
In other words, $d(S,J)$ is the
total number of hyperplanes $H_{i,j,r}$ that separate $P_S$ and $e_J$.

%We will now discuss three useful properties of $d$:

\begin{lemma}\label{lemma:distance_property}
Let $J \in {[n]\choose k}$ and let $S\subset {[n]\choose k}$ be a maximal sorted subset. Then $d(S,J)=0$
if and only if $J \in S$.
\end{lemma}

\begin{proof}
If $J\in S$, that is, $e_J$ is a vertex of the simplex $P_S$, then $d(S,J)=0$.

Now assume that $e_J$ is not a vertex of $P_S$, so it lies strictly outside of $P_S$.
Consider the $n$ hyperplanes $H_{i,j,r}$ that contain the $n$ facets
of the $(n-1)$-simplex $P_S$.  At least one of these hyperplanes separate $P_S$ and $e_J$,
so $d(S,J)\geq 1$.
\end{proof}

Recall (Definition~\ref{def:sorting}) that the sorting $I', J'$ of a pair $I,J\in{[n]\choose k}$
with the multiset union $I\cup J = \{a_1 \leq a_2 \leq \cdots \leq a_{2k}\}$ is
given by $I' = \{a_1,a_3,\dots, a_{2k-1}\}$ and $J' = \{a_2, a_4,\dots,a_{2k}\}$.

\begin{lemma}\label{lemma:seconddistance_property}
Let $S\subset {[n]\choose k}$ be a maximal sorted subset, let $I \in S$ and $J \in
{[n]\choose k}$, let $I',J'$ be the sorting of $I,J$, and let
$1 \leq i \leq j \leq n$. Then $d_{ij}(S,I')\leq
d_{ij}(S,J)$ and $d_{ij}(S,J')\leq d_{ij}(S,J)$.
\end{lemma}

\begin{proof} In order to show that $d_{ij}(S,I'), d_{ij}(S,J')\leq
d_{ij}(S,J)$, it is enough to show that any hyperplane $H_{i,j,r}$
(for some positive integer $r$) that separates $P_S$ and $e_{I'}$ also separates
$P_S$ and $e_J$ (and similarly for $P_S$ and $e_{J'}$).

Let $\alpha=|I \cap [i,j]|$ and $\beta=|J \cap [i,j]|$,
where $[i,j]=\{i,i+1,\dots,j\}$.
So $e_I$ lies on $H_{i,j,\alpha}$ and $e_J$ lies on $H_{i,j,\beta}$.

%and
%assume without loss of generality that $\alpha\leq \beta$.

By the definition of sorting, the numbers $|I' \cap [i,j]|$ and $|J' \cap
[i,j]|$ are equal to $\lfloor \frac{\alpha+\beta}{2} \rfloor$ and $\lceil
\frac{\alpha+\beta}{2} \rceil$ (not necessarily respectively).
So $e_{I'}$ lies on $H_{i,j, \lfloor \frac{\alpha+\beta}{2} \rfloor}$ or
$H_{i,j, \lceil \frac{\alpha+\beta}{2} \rceil}$; and similarly for $e_{J'}$.

Since both $\lfloor \frac{\alpha+\beta}{2} \rfloor$ and
$\lceil\frac{\alpha+\beta}{2} \rceil$ are weakly between $\alpha$ and $\beta$,
we get the needed claim.
\end{proof}

\begin{lemma}\label{lemma:thirddistance_property}
Let $S\subset {[n]\choose k}$
be a maximal sorted subset, and let $J \in {[n]\choose k}$ such that
$d(S,J)>0$. Then there exists $I \in S$ such that, for the sorting $I',J'$ of the pair $I,J$,
we have the strict inequalities
$d(S,I')<d(S,J)$ and $d(S,J')<d(S,J)$.  \end{lemma}

\begin{proof}
According to Lemma~\ref{lemma:distance_property}, there exists $I\in S$ such that $I$ and $J$ are not
sorted.
This means that there are $1\leq i\leq j \leq n$ such that the numbers
$\alpha=|I \cap [i,j]|$ and $\beta=|J \cap [i,j]|$ differ by at least two.
(We leave it as exercise for the reader to check that $I$ and $J$ are sorted if and only if
$|\alpha-\beta|\leq 1$ for any $1\leq i\leq j\leq n$.)
Therefore, both $\lfloor \frac{\alpha+\beta}{2} \rfloor$ and $\lceil
\frac{\alpha+\beta}{2} \rceil$ are strictly between $\alpha$ and $\beta$.

The point $e_{I'}$ lies on the hyperplane $H_{i,j, \lfloor \frac{\alpha+\beta}{2} \rfloor}$ or on
$H_{i,j, \lceil \frac{\alpha+\beta}{2} \rceil}$. In both cases this hyperplane separates $P_S$ and $e_J$, but does not
separate $P_S$ and $e_{I'}$.  Similarly for $e_{J'}$.
This means that we have the strict inequalities
$d_{ij}(S,I')<d_{ij}(S,J)$ and $d_{ij}(S,J')<d_{ij}(S,J)$.
Also, according to Lemma~\ref{lemma:seconddistance_property}, we have the weak inequalities
$d_{uv}(S,I')\leq d_{uv}(S,J)$ and $d_{uv}(S,J')\leq d_{uv}(S,J)$, for any $1\leq u\leq v\leq n$.
This implies the needed claim.
\end{proof}

We are now ready to prove the second part of Theorem~\ref{thm:torus_action}.

\begin{proof}
Let $A$, $A'$ and $S$ be as in Lemma~\ref{lemma:torus_action_first_part}.
Rescale $A'$ so that $\Delta_I(A')=1$, for $I\in S$.
We want to show that, for any $J \in {[n]\choose k}$ such that $J \notin S$, we
have $\Delta_J(A')<1$.

The proof is by induction.  Start with the base case, that is, with $J$ for which
$d(S,J)=1$. By Lemma~\ref{lemma:thirddistance_property}, there exists $I \in S$
such that $d(S,J'),\ d(S,I')<d(S,J)=1$, and hence $d(S,J')=d(S,I')=0$.
Therefore, by Lemma~\ref{lemma:distance_property}, we have $I',J' \in S$, and
thus $\Delta_{J'}(A')=\Delta_{I'}(A')=\Delta_I(A')=1$.  Applying
Corollary~\ref{cor:Skandera}, we get that
$\Delta_{I}(A')\Delta_{J}(A')<\Delta_{J'}(A')\Delta_{I'}(A')$, so
$1\cdot\Delta_{J}(A') <1 \cdot 1$, and hence $\Delta_{J}(A') <1$, which proves
the base case.

Now assume that the claim holds for any set whose distance
from $S$ is smaller than $d$, and let $J \in S$ such that $d(S,J)=d$. Using
again Lemma~\ref{lemma:thirddistance_property}, we pick $I \in S$ for which
$d(S,J'),\, d(S,I')<d$. By the inductive assumption,
$\Delta_{J'}(A'),\ \Delta_{I'}(A') \leq 1$. Therefore, applying
Corollary~\ref{cor:Skandera}, we get that
$\Delta_{I}(A')\Delta_{J}(A')<\Delta_{J'}(A')\Delta_{I'}(A') \leq 1$, and since
$\Delta_{I}(A')=1$, we get $\Delta_{J}(A')<1$.   We showed that,
for all $J \in {[n]\choose k}$ such that $J \notin S$, we have
$\Delta_{J}(A')<1$, so we are done.  \end{proof}

We can now finish the proof of Theorem~\ref{thm:sorted}.

\begin{proof}[Proof of Theorem~\ref{thm:sorted}]
The $\Rightarrow$ direction was already proven in Section~\ref{sec:inequalities_products_minors}.

For the case of maximal sorted sets, Theorem~\ref{thm:torus_action} implies
the $\Leftarrow$ direction of Theorem~\ref{thm:sorted}.

Suppose that the sorted set $S'$ (given in Theorem~\ref{thm:sorted}) is not
maximal.  Complete it to a maximal sorted set $S$ and rescale the columns of $A$ to get $A'$
as in Theorem~\ref{thm:torus_action} for the maximal sorted set $S$.

We now want to slightly modify $A'$ so that only the subset of minors $\Delta_{I}$,
for $I\in S'$, forms an arrangement of largest minors.

Apply the procedure in the proof Lemma~\ref{lemma:torus_action_first_part} to get
the matrix $A''_\epsilon$ such that
$$
\Delta_I(A''_\epsilon)=
\left\{
\begin{array}{cl}
1 & \textrm{for }I\in S' \\
1-\epsilon &\textrm{for } I\in S\setminus S'.
\end{array}
\right.
$$
Clearly, in the limit $\epsilon \to 0$, we have $A''_\epsilon \to A'$.

Since all minors $\Delta_J(A''_\epsilon)$ are continuous functions of $\epsilon$,
we can take $\epsilon>0$ to be small enough, so that all the minors
$\Delta_J(A''_\epsilon)$, $J\not\in S'$, are strictly less than 1.
This completes the proof of Theorem~\ref{thm:sorted}.
\end{proof}

\section{Sort-closed sets and alcoved polytopes}
\label{sec:sort_closed}

In this section, we extend Theorem~\ref{thm:sorted} about arrangements of largest minors
in a more general context of sort-closed sets.

%to sort-closed subsets $\S$ in $[n]\choose k$.
\medskip

\begin{definition}
For a set $\S\subset{[n]\choose k}$,  a subset $\A\subset \S$ is called
an
{\it arrangement of largest minors} in $\S$ if and only if there exists $A\in Gr^+(k,n)$ such that
all minors $\Delta_I(A)$, for $I\in\A$, are equal to each other; and the minors $\Delta_J$, for
$J\in \S\setminus \A$, are strictly less than the $\Delta_I$, for $I\in \A$.
\end{definition}

As before, the pair $I' = \{a_1,a_3,\dots, a_{2k-1}\}$, $J' = \{a_2, a_4,\dots,a_{2k}\}$
is the sorting of a pair $I,J$ with the multiset union $I\cup J = \{a_1 \leq a_2 \leq \cdots \leq a_{2k}\}$
(Definition~\ref{def:sorting}).

%For a pair $I,J\in{[n]\choose k}$, let
%$I'=\sort_1(I,J)$ and $J'=\sort_2(I,J)$ be their sortings defined by taking the multiset union
%$I\cup J= \{a_1\leq \cdots \leq a_{2k}\}$ and setting
%$\sort_1(I,J) =\{a_1,a_3,\dots,a_{2k-1}\}$
%and
%$\sort_2(I,J) =\{a_2,a_4,\dots,a_{2k}\}$.

\begin{definition}
A set $\S\subset {[n]\choose k}$ is called {\it sort-closed} if, for any pair $I,J\in \S$,
the elements of the sorted pair $I',J'$ are also in $\S$.
\end{definition}

For $\S\subset{[n]\choose k}$, let $P_\S\in \R^n$ be the polytope with vertices
$e_I = \sum_{i\in I } e_i$ for all $I\in \S$.

\begin{definition} \cite{LP} \ A polytope $P$ that belongs to a hyperplane $x_1+\cdots+x_n=k$ in $\R^n$ is
called {\it alcoved} if it is given by some inequalities of the form $x_i + x_{i+1} + \cdots + x_j \leq l$,
where $i,j\in [n]$ and $l\in\Z$.  (We assume that the indices $i$ of $x_i$ are taken modulo $n$.)
\end{definition}

The hyperplanes $x_i+\cdots +x_j=l$, $l\in \Z$, form the {\it affine Coxeter arrangement\/} of type A.
According to \cite{LP}, these hyperplanes subdivide an alcoved polytope
into unit simplices called {\it alcoves.}  This forms a triangulation of $P$.

Theorem~3.1 from \cite{LP} includes the following claim.
%The following proposition follows from results of \cite{LP}.

\begin{proposition} \cite{LP} \  A set $\S\subset{[n]\choose k}$ is sort-closed if and only if
the polytope $P_\S$ is alcoved.
\end{proposition}

For $\S\in{[n]\choose k}$, let $d=d(\S)$ denote the dimension of the polytope $P_\S$.

Let us first consider the case when $d=n-1$, that is, $P_\S$ is a full-dimensional polytope inside
the hyperplane $H=\{x_1+\cdots+x_n=k\}$.

Define the {\it normalized volume\/} $\Vol(P_\S)$ of this polytope
as $(n-1)!$ times the usual $(n-1)$-dimensional Eucledian volume of the projection $p(P_\S)\subset\R^{n-1}$,
where the projection $p$ is given by $p:(x_1,\dots,x_n)\mapsto(x_1,\dots,x_{n-1})$.

%Let $V_\S := \Vol\, P_\S$.

\begin{theorem}
\label{thmvolume}
Suppose that $\S\in{[n]\choose k}$ is a sort-closed set.  Assume that $d(\S)=n-1$.
A subset $\A\subset\S$ is an arrangement of largest minors in $\S$ if and only if $\A$ is sorted.

All maximal (by inclusion) arrangements of largest minors in $\S$ contain exactly $n$ elements.

The number of maximal arrangements of largest minors in $\S$ equals $\Vol(P_\S)$.
\end{theorem}

\begin{proof}
We can apply the same proof as for Theorem~\ref{thm:sorted}

%\ref{thm:torus_action}.
The proof of the $\Rightarrow$ direction of the first claim is exactly the same,
see Section~\ref{sec:inequalities_products_minors}.
For the $\Leftarrow$ direction of the first claim, we can apply the same inductive
argument as in the proof of Theorem~\ref{thm:torus_action}.
Indeed, in Section~\ref{sec:construction_largest},
we only used inequalities of the form $\Delta_I\, \Delta_J < \Delta_{I'}\, \Delta_{J'}$.
If $I$ and $J$ are in a sort-closed set $\S$, then so do their sortings $I'$ and $J'$.  So
the same argument works for sort-closed sets.

The maximal arrangements of largest minors correspond to top-dimensional simplices of the triangulation
of the polytope $P_\S$ into alcoves.  Thus each of these arrangements contains $n$ elements, and the
number of such arrangements is the number of alcoves in $P_\S$, which is equal to $\Vol(P_\S)$.
\end{proof}

This result can be easily extended to the case when $P_\S$ is not full
dimensional, that is, $d(\S)<n-1$.  Let $U$ be the affine subspace spanned by the polytope $P_\S$.
Then intersections of the hyperplanes $x_i+\cdots+x_j=l$, $l\in \Z$, with $U$ form a
$d(\S)$-dimensional affine Coxeter arrangement in $U$.  Let us define the normalized volume form $\Vol_U$ on
$U$ so that the smallest volume of an integer simplex in $U$ is 1.  See \cite{LP} for more details.

Theorem~\ref{thmvolume} holds without assuming that $d(\S)=n-1$.

\begin{corollary}
Let $\S\subset{[n]\choose k}$ be any sort-closed set.
A subset $\A\subset\S$ is an arrangement of largest minors in $\S$ if and only if $\A$ is sorted.
All maximal (by inclusion) arrangements of largest minors in $\S$ contain exactly $d(\S)+1$ elements.
The number of maximal arrangements of largest minors in $\S$ equals $\Vol_U(P_\S)$.
\label{cor:sort_closed_all_dim}
\end{corollary}

\section{Example: Equalities of matrix entries}
\label{sec:matrix_entries}

Under the map $\phi:\Mat(k,n-k)\to Gr(k,n)$ defined in Section~\ref{sec:from_Mat_to_Gr},
the matrix entries $a_{ij}$ of a $k\times (n-k)$ matrix $A$ correspond to a special subset
of the Pl\"ucker coordinates of $\phi(A)$.
%, defined as follows.
For $i\in[k]$ and $j\in[n-k]$, let $I(i,j) := ([k]\setminus\{k+1-i\})\cup\{j+k\}$.
Then $a_{ij} = \Delta_{I(i,j)}(\phi(A))$.  Let
$$
\S_{k,n} = \{I(i,j) \mid  i\in[k] \textrm{ and } j\in[n-k]\}\subset {[n]\choose k}.
$$

\begin{lemma}\label{sortedcloselemmaone}
Let $i_1, i_2\in[k]$, and $j_1,j_2\in[n-k]$ such that $i_1\leq i_2$.
Then the pair $I(i_1,j_1)$, $I(i_2,j_2)$ is sorted if and only if $j_1\leq j_2$.

If the pair $I(i_1,j_1)$, $I(i_2,j_2)$ is not sorted then its sorting is the pair
$I(i_1,j_2)$, $I(i_2,j_1)$.
\end{lemma}

\begin{proof}
If $i_1=i_2$, then the statement holds trivially.
Assume that $i_1 < i_2$. By definition, $I=I(i_1,j_1)$, $J=I(i_2,j_2)$ are sorted if and only if
their sorting, $I', J'$ satisfy $I=I'$ and $J=J'$ or $I=J'$ and $J=I'$.
We have $I'=([k]\setminus\{k+1-i_1\})\cup\{k+\min\{j_1,j_2\}\}$, and hence
the pair $I(i_1,j_1)$, $I(i_2,j_2)$ is sorted if and only if
$\min\{j_1,j_2\}=j_1$, or equivalently $j_1\leq j_2$. The second part of the
statement follows directly from the description of $I'$.
\end{proof}

\begin{remark}
Assume that $i_1<i_2$ and $j_1<j_2$.
The positivity of $2\times 2$ minors
of a totally positive $k\times (n-k)$ matrix $A$
means that
$a_{i_1,j_1} a_{i_2,j_2}> a_{i_1,j_2} a_{i_2,j_1}$.
Equivalently, for the positive Grassmannian $Gr^+(k,n)$, we have
$$
\Delta_{I(i_1,j_1)}
\Delta_{I(i_2,j_2)}
>
\Delta_{I(i_1,j_2)}
\Delta_{I(i_2,j_1)}.
$$
\end{remark}

The next lemma follows immediately from Lemma~\ref{sortedcloselemmaone}
\begin{lemma}
The set $\S_{k,n}$ is sort-closed.
\end{lemma}

Note that polytope $P_{\S_{k,n}}$ is the product of two simplices
$\Delta^{k-1}\times \Delta^{n-k-1}$.  Its normalized volume is
${n-2\choose k-1}$.

The $k\times (n-k)$ {\it grid graph\/} $G_{k,n-k}= P_k \,\Box\, P_{n-k}$ is
the Cartesian product of two path graphs $P_k$ and $P_{n-k}$ (with $k$ and $n-k$
vertices respectively).
We can draw $G_{k,n-k}$ as the lattice that
has $k$ rows and $n-k$ vertices in each row.  Denote the vertices in
$G_{k,n-k}$ by $v_{ij}$, $1 \leq i \leq k$, $1 \leq j \leq n-k$. A {\it lattice
path\/} is a shortest path in $G_{k,n-k}$ that starts at $v_{11}$ and ends at $v_{k,n-k}$.
% such
%that in each step, one can move either one step to the right or one step down.
Clearly, any lattice path in $G_{k,n-k}$ contains exactly $n-1$ vertices.

Let us associate the element $I(i,j)\in \S_{k,n}$ to a vertex $v_{ij}$ of the grid graph $G_{k,n-k}$.
Then a lattice path corresponds to a subset of $\S_{k,n}$ formed by the elements $I(i,j)$ for
the vertices $v_{ij}$ of the lattice path.

The following result follows Theorem~\ref{thmvolume}.

\begin{corollary}\label{cor:maximalentries}
Every maximal arrangement of largest minors in $\S_{k,n}$ (that is, a maximal
arrangement of largest entries of a totally positive $k\times(n-k)$ matrix $A$)
contains exactly $n-1$ elements.  There are $n-2 \choose k-1$ such maximal arrangements.
They correspond to the lattice paths in the grid $G_{k,n-k}$.
\end{corollary}

Equivalently, maximal arrangements of largest minors in $\S_{k,n}$
are in bijection with non-crossing spanning trees in the complete bipartite graph
$K_{k,n-k}$.

More generally, let us say  that a bipartite graph $G\subset K_{k,n-k}$ is {\it sort-closed\/} if whenever $G$ contain
a pair of edges $(i_1,j_1)$ and $(i_2,j_2)$ with $i_1<i_2$ and $j_1>j_2$, it also contains two edges $(i_1,j_2)$ and
$(i_2,j_1)$.

The previous result can be generalized to sort-closed graphs.  Let $\S_G=\{I(i,j)\mid (i,j)\in E(G)\}$.
The polytope $P_{\S_G}$ is called the {\it root polytope\/} of the graph $G$.
The following claim follows from Corollary~\ref{cor:sort_closed_all_dim}.

\begin{corollary}
Let $G\subset K_{k,n-k}$ be a sort-closed graph.
Any  maximal arrangement of largest minors in $\S_G$ contains exactly $n-c$ elements, where $c$ is the number of
connected components in $G$.   The number of maximal arrangements of largest minors in $\S_G$ equals the normalized
volume of the root polytope of the graph $G$.
\end{corollary}

We conclude this section with a statement regarding arrangement of smallest
minors in $\S_{k,n}$.  We say that a {\it transposed lattice path\/} is a shortest path in $G_{k,n-k}$
that starts at $v_{1,n-k}$ and ends at $v_{k1}$.

\begin{theorem}
Every maximal arrangement of smallest minors in $\S_{k,n}$ (that is, a maximal
arrangement of smallest matrix entries of totally positive $k\times(n-k)$ matrix $A$)
contains exactly $n-1$ elements.  There are exactly $n-2 \choose k-1$ such maximal arrangements.
They correspond to transposed lattice paths in $G_{k,n-k}$.
 \end{theorem}

%Equivalently, maximal arrangements of smallest matrix entries correspond to
%the spanning trees  in the complete bipartite graph $K_{k,n-k}$ such that
%any pair edges cross.

\begin{proof}
We will describe a bijection between arrangements of largest minors in
$\S_{n-k,n}$ and arrangements of smallest minors in $\S_{k,n}$. Let $A$ be an
$(n-k)\times k$ totally positive matrix in which the maximal entries form an
arrangement of largest minors in $\S_{n-k,n}$, and assume without loss of
generality that the maximal entry is 1.  Consider the matrix $B$ obtained from $A$
inverting its entries and rotating it by 90 degrees.
That is, the entries of $B$ are $b_{ij}=\frac{1}{a_{j,k-i+1}}$. Since $A$ is totally positive, its entries
and $2\times 2$ minors are also positive. Therefore, the entries of $B$ are positive
as well. Moreover, the $2\times 2$ minors of $B$ are also positive, since
$$
\left| \begin{array}{cc}
        b_{ij} & b_{iy} \\
        b_{xj} & b_{xy} \\
      \end{array}
    \right|
=
 \left| \begin{array}{cc}
                    \frac{1}{a_{j,k-i+1}} & \frac{1}{a_{y,k-i+1}} \\
                    \frac{1}{a_{j,k-x+1}} & \frac{1}{a_{y,k-x+1}} \\
                  \end{array}
                \right|=\frac{1}{a_{j,k-i+1}a_{y,k-x+1}}-\frac{1}{a_{y,k-i+1}a_{j,k-x+1}}>0.
$$
The last inequality follows from the fact that
$
\left|
      \begin{array}{cc}
        a_{j,k-x+1} & a_{j,k-i+1} \\
        a_{y,k-x+1} & a_{y,k-i+1} \\
      \end{array}
    \right|>0$.

From \cite[Theorem 7]{Fallat}, it follows that since the entries and the $2\times 2$ minors
of $B$ are positive, there exists some positive integer $m$ such that the
$m$-th Hadamard power of $B$ is totally positive. That is, the matrix $C$ with
entries $c_{ij}=b_{ij}^m$ is totally positive. Note that the largest
entries in $A$ correspond to the smallest entries in $C$. Similarly, we could
start from a matrix $A$ that form maximal arrangement of smallest minors in
$\S_{k,n}$, and by the same procedure obtain a matrix $C$ that form maximal
arrangement of largest minors in $\S_{n-k,n}$. Hence, we obtained a bijection
between arrangements of largest minors in $\S_{n-k,n}$ and arrangements of
smallest minors in $\S_{k,n}$. The proof now follows from
Corollary~\ref{cor:maximalentries}.
\end{proof}

\section{The case of the nonnegative Grassmannian}
\label{sce:nonnegative_Grass}

The next natural step is to extend the structures discussed above to the case of the
nonnegative Grassmannian $Gr^{\geq }(k,n)$.
In other words, let us now allow some subset of Pl\"ucker coordinates to be zero, and try to describe
possible arrangements of smallest (largest) positive Pl\"ucker coordinates.

Many arguments that we used for the positive Grassmannian, will not work for the nonnegative
Grassmannian.  For example, if some Pl\"ucker coordinates are allowed to be zero, then we can no longer conclude
from the 3-term Pl\"ucker relation that $\Delta_{13}\Delta_{24} > \Delta_{12}\Delta_{34}$.

Let us describe these structures in the case $k=2$.
The combinatorial structure of the nonnegative Grassmannian $Gr^{\geq}(2,n)$ is relatively easy.
Its positroid cells \cite{Pos2} are represented by $2\times n$ matrices $A=[v_1,\dots,v_n]$, $v_i\in\R^2$,
with some (possibly empty) subset of zero columns $v_i=0$, and some (cyclically) consecutive columns
$v_r,v_{r+1},\dots,v_s$ parallel to each other.  One can easily remove the zero columns; and assume that $A$ has
no zero columns.  Then this combinatorial structure is given by a decomposition of the set $[n]$
into a disjoint union of cyclically consecutive intervals $[n]=B_1\cup \dots\cup B_r$.
The Pl\"ucker coordinate $\Delta_{ij}$ is strictly positive if $i$ and $j$ belong to two different intervals $B_l$'s;
and $\Delta_{ij}=0$ if $i$ and $j$ are in the same interval.

The following result can be deduced from the results of Section~\ref{sec:k=2}.

\begin{theorem}\label{thmcolumnduplication}
Maximal arrangements of smallest (largest) positive minors correspond to triangulations (thrackles) on the $r$ vertices $1,\dots,r$.
Whenever a triangulation (thrackle) contains an edge $(a,b)$, the corresponding arrangement contains all Pl\"ucker coordinates
$\Delta_{ij}$\,, for $i\in B_a$ and $j\in B_b$.
\end{theorem}

We can think that vertices $1,\dots,r$ of a triangulation (thrackle) $G$ have the multiplicities $n_a = |B_a|$.
The total sum of the multiplicities should be $\sum n_a = n$.
The number of minors in the corresponding arrangement of smallest (largest) minors equals the sum
$$
\sum_{(ab)\in E(G)} n_a n_b
$$
over all edges $(a,b)$ of $G$.

Remark that it is no longer true that all maximal (by containment)
arrangements of smallest (or largest) equal minors contain the same number of minors.

\begin{theorem}
A maximal (by size)
arrangement of smallest minors or largest minors in $Gr^{\geq }(2,n)$ contains the following number of elements:
$$
\left\{
\begin{array}{cl}
3m^2 & \textrm{if } n = 3m\\
m(3m+2) & \textrm{if } n = 3m+1 \\
(m+1)(3m+1) & \textrm{if } n = 3m+2
\end{array}
\right.
$$
\end{theorem}

\begin{proof}
We start with smallest minors. By Theorem~\ref{thmcolumnduplication}, we can
assume that the graph $G$ described above corresponds to a triangulation (since
adding an edge to $G$ cannot decrease the expression $ \sum_{(ab)\in E(G)} n_a
n_b $), and we would like to maximize $\sum_{(ab)\in E(G)} n_a n_b$, subject to
the constraint $\sum n_a = n$ (keeping in mind that all the variables are
nonnegative integers). We will use Lagrange multipliers.  Define
$$
f(n_1,n_2,\ldots,n_r)=\sum_{(ab)\in E(G)} n_a n_b-\lambda\left(\sum_{a=1}^r n_a - n\right).
$$
Taking partial derivatives with respect to the variables
$n_1,n_2,\ldots,n_r, \lambda$,
we
get, for every $v \in V(G)$, an equality of the form $\sum_{(vb)\in E(G)}
n_b=\lambda$. We also get $\sum n_a = n$. Now consider several
cases.

(1) $r=3$. In this case, $G$ is a triangle, and the equalities are
$$
n_1+n_2=n_1+n_3=n_2+n_3, \ n_1+n_2+n_3=n.
$$
Thus, if $n=0\pmod{3}$, the solution is $n_1=n_2=n_3=n/3$, and
$n_1n_2+n_1n_3+n_2n_3={n^2}/3$. If $n=1\pmod{3}$ then let $n=3m+1$. Since
$n_1,n_2,n_3$ are integers, the maximal possible value for
$n_1n_2+n_1n_3+n_2n_3$ is $\lfloor n^2/3 \rfloor$, which we attain by choosing
$n_1=n_2=m$, $n_3=m+1$. Finally, if $n=2\pmod{3}$, let $n=3m+2$. Then by choosing
$n_1=n_2=m+1$, $n_3=m$ we obtain again $n_1n_2+n_1n_3+n_2n_3=\lfloor n^2/3 \rfloor$.

(2) $r=4$. In this case, $G$ is $K_4 \backslash e$, and the equalities are
$$
n_1+n_2+n_4=n_2+n_3+n_4=n_1+n_3, \ n_1+n_2+n_3+n_4=n.
$$
Hence $n_1=n_3=n_2+n_4$ and thus, if $n=0\pmod{3}$, the maximal value achieved
at $n_1=n_3=n/3, n_2+n_4=n/3$.  We have
$$
\begin{array}{l}
n_1n_3+n_1n_2+n_2n_3+n_3n_4+n_1n_4=\\[.1in]
\qquad =n^2/9+(n_2+n_4)(n_1+n_3)= n^2/9+(2n/3)(n/3)=n^2/3.
\end{array}
$$
Note that for $n=1,2\pmod{3}$, the maximal value of
$n_1n_3+n_1n_2+n_2n_3+n_3n_4+n_1n_4$ (subject to the constraints) cannot exceed
$\lfloor n^2/3 \rfloor$, and thus for $r=4$ we obtain at most the same maximal
value as in the case $r=3$.

(3) $r\geq 5$. In this case, let $v$ be a vertex of degree 2 in the
triangulation, and let $a$ and $b$ be its neighbors, so $a$ and $b$ are
connected. Note that the edge $(a,b)$ is an ``inner edge'' in the triangulation
(since $r\geq 5$), and hence it is part of another triangle. Let $p \neq v$ be
the vertex that forms, together with $a$ and $b$, this additional triangle, and
hence $p$ is connected to both $a$ and $b$. Since $r\geq 5$, the degree of $p$
is at least 3, so there exists a vertex $x \notin \{a,b,p,v\}$ that is
connected to $p$. Therefore we get in particular that $n_a+n_b \geq
n_a+n_b+n_x$, and since all the $n_i$'s are nonnegative (since those are the
only cases that we consider) we get $n_x=0$. Thus we could equivalently
consider a triangulation on $r-1$ vertices instead of $r$ vertices (having even
less constraints, so the maximal value can only increase). Since this process
holds for any triangulation on at least 5 vertices, we obtain a reduction to
the case $r=4$.

After considering all the possible cases, we conclude that the maximal
arrangement of smallest minors in $Gr^{\geq }(2,n)$ contains the number of
elements that stated in the theorem.

Now consider an arrangement of largest minors. In this case, $G$ is a maximal
thrackle. It is easy to check that if $G$ contains leaves then there exists a
vertex $v$ for which $n_v=0$, and hence we get reduction to smaller number of
vertices. Thus we can assume that $G$ does not contain leaves, and hence $G$ is
an odd cycle. In this case, applying Lagrange multipliers we get that
$n_1=n_2=\ldots=n_r=n/r$, and hence the maximal value of the expression
$\sum_{(ab)\in E(G)} n_a n_b$ is $n^2/r$. Thus we get that the maximal value
achieved in the case $r=3$ (where $G$ is a triangle). We can analyze the cases
$n=0,1,2\pmod{3}$ in the same way as above, and hence we are done.
\end{proof}

\section{Construction of arrangements of smallest minors which are not weakly separated}
\label{sec:not_weakly_separated}

In this section, we discuss properties of pairs of minors which are not weakly separated
but still can be equal and smallest. In order to construct such pairs,
we will use plabic graphs from \cite{Pos2}.
A bijection between plabic graphs and weakly separated sets was constructed in \cite{OPS}.

\subsection{Plabic graphs}
Let us give some definitions and theorems from \cite{Pos2, OPS}.
See these papers for more details.

\begin{definition}
A {\it plabic graph\/} (planar bicolored graph) is a planar undirected graph $G$ drawn inside a disk
with vertices colored in black or white colors.
The vertices on the boundary of the disk, called the {\it boundary vertices,}
are labeled in clockwise order by $[n]$.
\end{definition}

\begin{definition}
Let $G$ be a plabic graph.  A {\it strand} in $G$ is a directed path $T$ such
that $T$
%joins two boundary vertices and
satisfies the following rules of the road: At every black
vertex turn right, and at a white vertex turn left.
\end{definition}

%We now introduce the notion of reduced plabic graphs.
\begin{definition}
\label{definition:reduced}
A plabic graph is called {\it reduced} if the following holds:
\begin{enumerate}
  \item (No closed strands)
The strands cannot be closed loops in the interior of the graph.
  \item
(No self-intersecting strands)
No strand passes through itself. The only exception is that we allow simple
loops that start and end at a boundary vertex $i$.
  \item (No bad double crossings) For any two strands $\alpha$ and $\beta$, if $\alpha$ and $\beta$ have two common vertices $A$ and
$B$, then one strand, say $\alpha$, is directed from $A$ to $B$, and the other strand $\beta$
is directed from $B$ to $A$. (That is, the crossings of $\alpha$ and $\beta$ occur in opposite
orders in the two strands.)
\end{enumerate}
\end{definition}

Any strand in a reduced plabic graph $G$ connects two boundary vertices.

\begin{definition}
\label{def:decorated_strand_perm}
We associate the {\it decorated strand permutation\/} $\pi_G \in S_n$ with a reduced plabic graph $G$,
such that $\pi_G(i)=j$ if the strand that starts at the boundary vertex $i$ ends at the boundary
vertex $j$.  A strand is labelled by $i\in[n]$ if it ends at the boundary vertex $i$
(and starts at the boundary vertex $\pi_G^{-1}(i)$).

The fixed points of $\pi_G$ are colored in two colors, as follows.
If $i$ is a fixed point of $\pi_G$, that is $\pi_G(i)=i$, then the boundary
vertex $i$ is attached to a vertex $v$ of degree 1.   The color of $i$ is the color
of the vertex $v$.
\end{definition}

Let us describe a certain labeling of faces of a reduced plabic graph $G$ with
subsets of $[n]$.
% This labeling was originally introduces in \cite{Scott}.
Let $ i \in [n]$ and consider the strand labelled by $i$. By
definition~\ref{definition:reduced}(2), this strand divides the disk into two
parts. Place $i$ in every face $F$ that lies to the left of strand $i$. Apply
the same process for every $i$ in $[n]$. We then say that the label of $F$ is
the collection of all $i$'s that placed inside $F$. Finally, let $F(G)$ be the
set of labels that occur on each face of the graph $G$. In \cite{Pos2} it was
shown that all the faces in $G$ are labeled by the same number of strands,
which we denote by $k$. The following theorem is from~\cite{OPS}.

\begin{theorem} \cite{OPS}
Each maximal weakly separated collection $C \subset {[n]\choose k}$ has the form
$C=F(G)$ for some reduced plabic graph $G$ with decorated strand permutation
$\pi(i) = i + k \pmod n$, $i=1,\dots,n$.  \end{theorem}

Let us describe 3 types of moves on a plabic graph:

\begin{itemize}
\item[(M1)]
Pick a square with vertices alternating in colors, such that all vertices have
degree 3. We can switch the colors of all the vertices as described in Figure~\ref{move1}.

\begin{figure}[h!]
\includegraphics[height=0.5in]{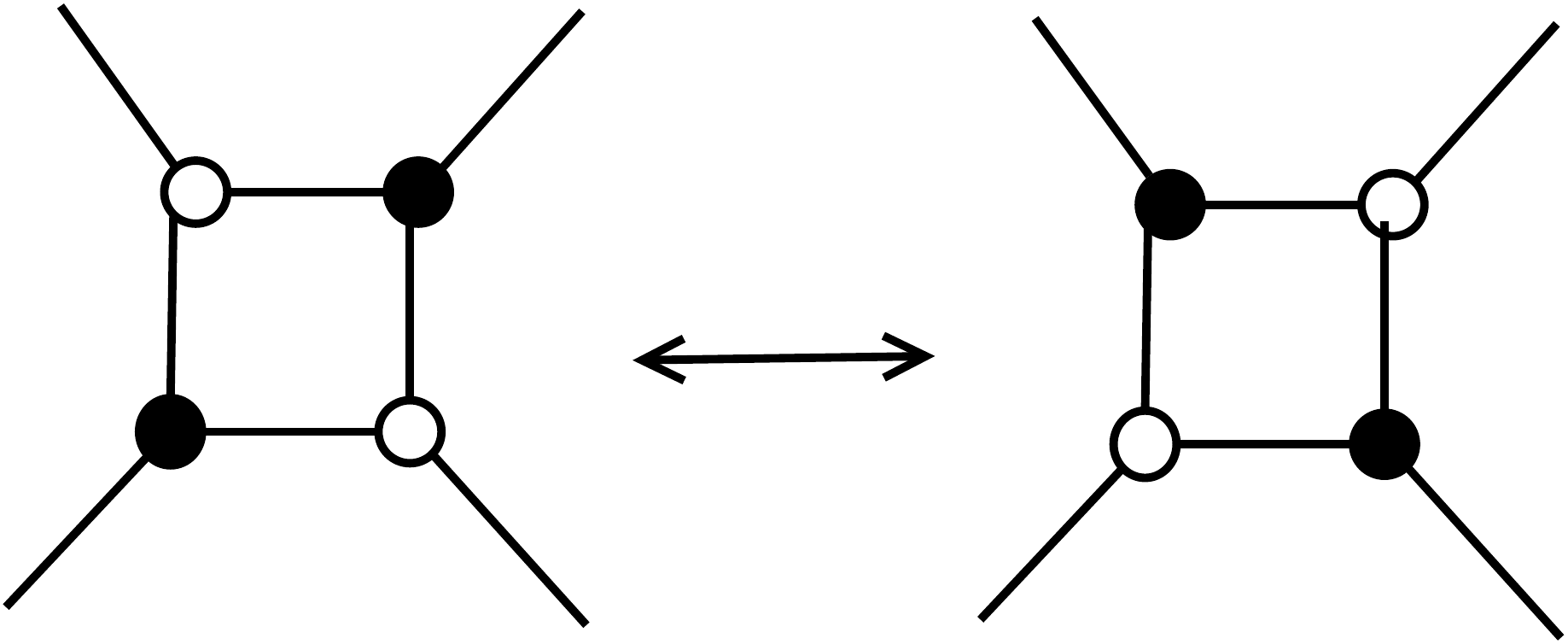}
\caption{(M1) square move}
\label{move1}
\end{figure}

\item[(M2)]
For two adjoint vertices of the same color, we can contract them into one
vertex. See Figure~\ref{move2}.

\begin{figure}[h!]
\includegraphics[height=0.5in]{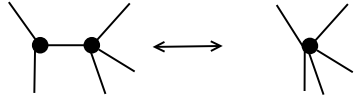}
\caption{(M2) unicolored edge contraction}
\label{move2}
\end{figure}

\item[(M3)]
We can insert or remove a vertex inside any edge. See Figure~\ref{move3}.

\begin{figure}[h!]
\includegraphics[height=0.1in]{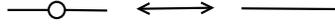}
\caption{(M3) vertex removal}
\label{move3}
\end{figure}
\end{itemize}

The moves do not change reducedness of plabic graphs.
%the associated strand permutation $\pi_G$ of the plabic
%graph $G$, as well as its
%Moreover, the following result holds.

\begin{theorem} {\rm \cite{Pos2}} \
Let $G$ and $G'$ be two reduced plabic graphs with the same number of boundary
vertices. Then $G$ and $G'$ have the same decorated strand permutation
$\pi_G= \pi_{G'}$ if and only if
$G'$ can be obtained from $G$ by a sequence of moves {\rm (M1)--(M3).}
%\begin{itemize}
%  \item $G$ can be obtained from $G'$ by
%  \item $G$ and $G'$ have the same strand permutation.
%\end{itemize}
\end{theorem}

\subsection{$p$-Interlaced sets}
Let us associate to each pair $I,J$ of $k$-element subset in $[n]$ a certain
lattice path.
%So far we introduced some properties of weakly separated sets and their
%relation with plabic graphs.
%Before we state our conjecture regarding
%arrangements of smallest minors, we would like to associate a certain path to
%pairs of $k$-element subsets of $[n]$.

\begin{definition}
Let $I,J\in {[n]\choose k}$ be two $k$-element sets, and let
$r= |I\setminus J| = |J\setminus I|$.
%Let $I\setminus J = \{a_1,\dots,a_r\}$ and $J\setminus I=\{b_1,\dots,b_r\}$,
%
Let $(I\setminus J) \cup (J\setminus I)= \{c_1<c_2<\ldots<c_{2r-1}<c_{2r}\}$.
Define $P=P(I,J)$ to be the lattice path in $\Z^2$ that
starts at $P_0=(0,0)$, ends at $P_{2r}=(2r,0)$, and contains up steps $(1,1)$ and
down steps $(1,-1)$, such that if $c_i\in I\setminus J$ (resp.,
$c_i\in J\setminus I$) then the $i$th step of $P$ is an up step (resp., down step).
%and obtained by the following
%iteration on $i$: For every $1 \leq i \leq 2r$, \begin{enumerate} \item If $c_i
%\in I$ then the path contains the straight line segment that connects $P_{i-1}$
%and $P_{i-1}+(1,1)$.  \item If $c_i \in J$ then the path contains the straight
%line segment that connects $P_{i-1}$ and $P_{i-1}+(1,-1)$.
%
%\end{enumerate}
\end{definition}

For example, the paths
$P(\{1,4,7,8\},\{2,3,5,6\})$ and $P(\{1,2,3,6\},\{4,5,7,8\})$ are
shown in Figure~\ref{path}.

\begin{figure}[h!]
\includegraphics[height=0.4in]{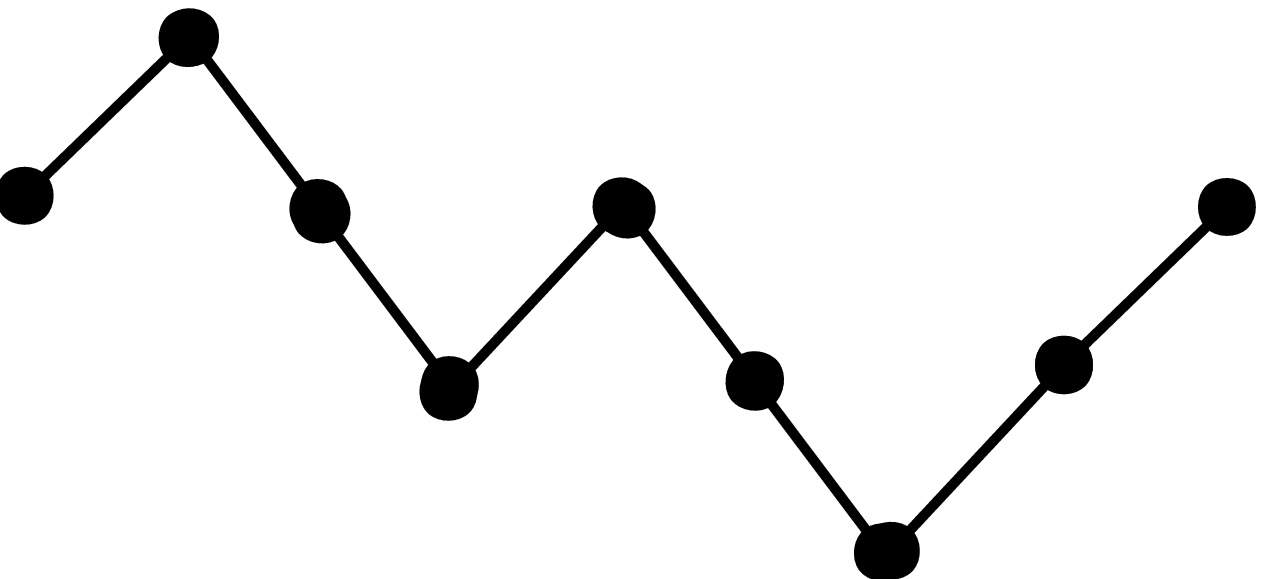}
\qquad\qquad
\includegraphics[height=0.4in]{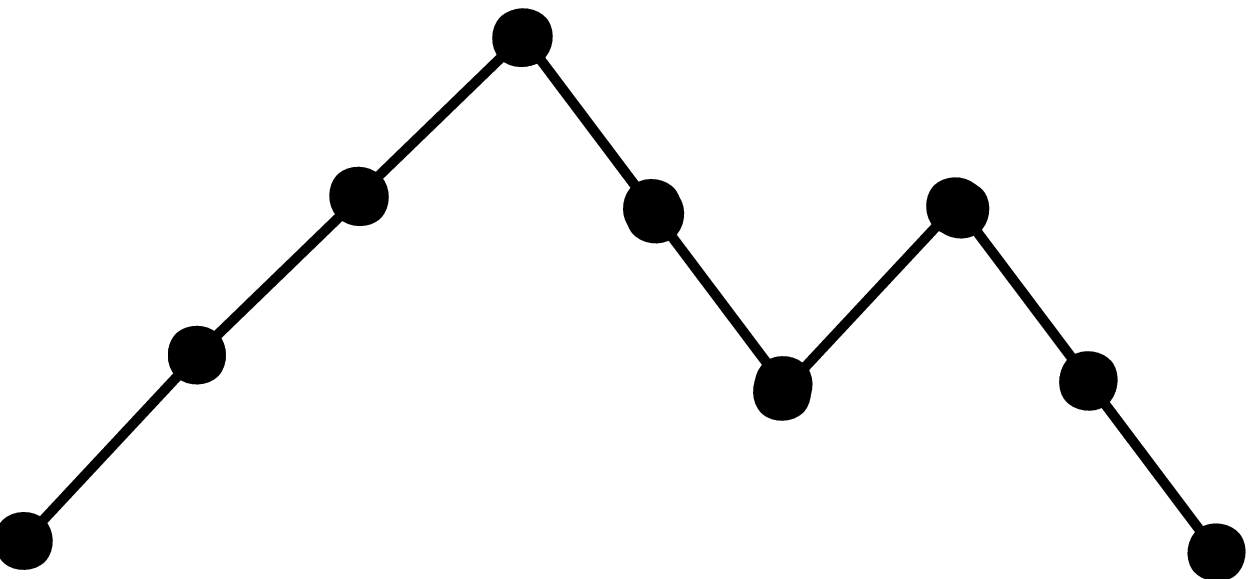}
\caption{The path $P(\{1,4,7,8\},\{2,3,5,6\})$ and its cyclic rotation
$P(\{1,2,3,6\},\{4,5,7,8\})$}
\label{path}
\end{figure}

Clearly, for any pair $I,J\in{[n]\choose k}$, there is a cyclic shift $I',J'$
such that the path $P(I',J')$
is a Dyck path, that is, it never goes below $y=0$.
In the following discussion we will assume, without loss of generality,
that $P(I,J)$ is a Dyck path.

%It is easy to see that by possibly switching the role of $I$ and $J$, as well
%as rotating the path modulo $2r$, we can transform $P(I,J)$ into a path that
%never goes bellow $y=0$. Let us denote the path after transformation by
%$P_T{I,J}$ (there might be several options for $P_T(I,J)$, in such a case we
%just pick one of them). For example, the path
%$P_T(\{1,2,6,7\},\{3,4,5,8\})=P(\{1,2,4,5\},\{3,6,7,8\})$ is depicted in
%Figure~\ref{rotatedpath}.

%\begin{figure}[h!]
%\includegraphics[height=0.4in]{pictures/IJpathrotated.eps}
%\caption{The path $P_T(\{1,2,6,7\},\{3,4,5,8\})$}
%\label{rotatedpath}
%\end{figure}

\begin{definition}
A {\it pick} in the path $P=P(I,J)$ is an index $i\in [2r-1]$ such that the $i$th step in $P$
is an up step and the $(i+1)$st step of $P$ is a down step.
%a point $P_i$ for some $1 \leq i \leq
%2r-1$ such that the $y$-coordinate of the points $P_{i-1}, P_{i+1}$ are bellow
%the $y$-coordinate of $P_i$.

We say that the pair $I,J$
%\in {[n]\choose k}$
is
{\it $p$-interlaced\/} if the number of picks in $P(I,J)$ is $p$.
%Note that even if there is more than one option for $P_T(I,J)$, the number of picks stays the same.
\end{definition}

For example, the pair $\{1,2,3,6\},\{4,5,7,8\}$ is 2-interlaced.

\begin{remark}
The pair $I,J\in {[n]\choose k}$ is weakly separated if and only if it is
1-interlaced.  The pair $I,J\in {[n]\choose k}$ for which $|I\setminus
J|=|J\setminus I|=r$ is sorted if and only if it is $r$-interlaced.
\end{remark}

For a $p$-interlaced pair $I,J$, the
{\it length parameters\/} $(\alpha_1,\beta_1,\alpha_2,\beta_2,\ldots,\alpha_p,\beta_p)$
are defined as the lengths the $2p$ straight line segments of $P(I,J)$.
(The $\alpha_i$ are the lengths of chains of up steps, and the $\beta_j$ are the
length of chains of down steps.)
For example the length parameters for the pair $\{1,2,3,6\},\{4,5,7,8\}$ are
$(3,2,1,2)$.

%For example the length parameters that
%correspond to the path $P_T(\{1,2,6,7\},\{3,4,5,8\})$ in
%Figure~\ref{rotatedpath} are (2,1,2,3). More generally, let us denote those
%parameters by $(\alpha_1,\beta_1,\alpha_2,\beta_2,\ldots,\alpha_t,\beta_t)$.

%Note that (by possible cyclic shifts and switching the rule of $I$ and $J$) we
%can always chose the path $P(I,J)$ in such a way that $\alpha_1 \geq
%\alpha_i$ and $\alpha_1 \geq \beta_i$ for all $i\in[p]$.
%For example, the path $P(\{1,2,3,6\},\{4,5,7,8\})$ has this property.
%
%From now on, we will assume, without loss of generality, that $\alpha_1$ is the maximal length
%parameter for $I,J$.

%that this property holds for all $P_T(I,J)$. For example,
%instead of the path appears in Figure~\ref{rotatedpath}, we could have chosen
%$P_T(\{1,2,6,7\},\{3,4,5,8\})$ to be the path that appears in
%Figure~\ref{newrotatedpath}. Its length parameters are (3,2,1,2).

%\begin{figure}[h!]
%\includegraphics[height=0.4in]{pictures/newrotatedpath.eps}
%\caption{The path $P_T(\{1,2,6,7\},\{3,4,5,8\})$}
%\label{newrotatedpath}
%\end{figure}

\subsection{Conjecture and results on pairs of smallest minors}
We are now ready to state a conjecture regarding the structure of pairs of
minors that can be equal and minimal.

\begin{conjecture}
\label{conjecture:minimal_minors}
Let $I,J\in {[n]\choose k}$ such that $P(I,J)$ is a Dyck path.
Then there exists an arrangement of smallest minors $S \subset {[n]\choose k}$
such that $I,J \in S$ if and only if one of the following holds:
\begin{enumerate}
  \item the pair $I,J$ is 1-interlaced (equivalently, it is weakly separated), or
  \item the pair $I,J$ is 2-interlaced and its length parameters
  $(\alpha_1,\beta_1,\alpha_2,\beta_2)$ satisfy $\alpha_i\ne \beta_j$, for any $i$ and $j$.
\end{enumerate}
\end{conjecture}

%Let us explain the intuition behind the criteria of part (2) above.

According to strict Skandera's inequalities (Theorem~\ref{thm:Skandera_strict}),
for a 2-interlaced pair $I,J$, there exists a point
in $Gr^+(k,n)$ for which the product $\Delta_I\Delta_J$ is smaller than any
other product of complimentary minors if and only if $\alpha_i\ne \beta_j$, for any $i$ and $j$.
This shows that, for $2$-interlaced pairs, condition (2) is necessary.
%Therefore, Skandera's inequalities do not contradict the conjecture.

%Also, note
%that for a pair $I,J$ that satisfies condition (2), there exists exactly one
%possible option for $P_T(I,J)$.

%Our purpose in this section is to provide
Let us provide some evidence for the validity of the conjecture.
From Theorem~\ref{thm:weakly_separated_123}, the conjecture holds for $1 \leq k \leq
3$.  We will show in this section that the conjecture holds for $k=4,5$ as well,
and then suggest a possible way to generalize the proof for general $k$. The
idea behind the construction is that pairs $I,J$ that appear in the conjecture
are related in a remarkable way via a certain chain of moves of plabic graphs.

\begin{theorem}\label{kequals4case}
Conjecture~\ref{conjecture:minimal_minors} holds for $k\leq 5$ (or $k\geq n-5$)
and any $n$.
\end{theorem}

In order to prove this theorem, we will present several examples of matrices
with needed equalities and inequalities between the minors.  It is not hard to
check directly that these matrices satisfy the needed conditions.  However, it
was a quite nontrivial problem to find these examples.   After the proof we
will explain a general method that allowed us to construct such matrices using
plabic graphs.

\begin{proof}
Because of the duality of $Gr(k,n)\simeq Gr(n-k,n)$, the cases $k\geq n-5$ are
equivalent to the cases $k\leq 5$.  The case $k\leq 3$ follows from
Theorem~\ref{thm:weakly_separated_123}.

Let us assume that $k=4$. If $I \cap J \neq \emptyset$, then the problem reduces to
a smaller $k$ and the result follows from Theorem~\ref{thm:weakly_separated_123}.
Therefore, assume that $I \cap J = \emptyset$.
Without loss of generality we can assume that $n=8$. Using
the cyclic symmetry of the Grassmannian and the results from previous sections,
there is only one case to consider: $I=\{1,2,3,6\}, J=\{4,5,7,8\}$ (all the
other cases follow either from Theorem~\ref{thm:WSeparated}, or from
Theorem~\ref{thm:Skandera_strict}). The matrix bellow satisfies
$\Delta_I=\Delta_J=1$, and $\Delta_K \geq 1$ for all $K\in {[8]\choose 4}$.
% This matrix was constructed using certain properties of plabic graphs.
% We will explain the idea behind the construction after the proof.

$$
  \begin{pmatrix}
    1 & 0 & 0 & 0 & -1 & -7 & -\frac{37}{2} & -13 \\[.05in]
    0 & 1 & 0 & 0 & \frac{3}{2} & \frac{19}{2} & \frac{95}{4} & \frac{33}{2} \\[.05in]
    0 & 0 & 1 & 0 & -\frac{5}{2} & -\frac{27}{2} & -\frac{125}{4} & -\frac{43}{2} \\[.05in]
    0 & 0 & 0 & 1 & 1 & 1 & \frac{3}{2} & 1
  \end{pmatrix}
$$
This proves the case $k=4$.

Let us now assume that $k=5$. As before, if $I \cap J \neq \emptyset$ then we are done.
So assume that $I \cap J= \emptyset$.  Up to cyclic shifts and exchanging the roles of $I$ and $J$,
there are 3 cases to consider:
\begin{enumerate}
  \item $I=\{1,2,3,4,7\}, J=\{5,6,8,9,10\}$
  \item $I=\{1,2,3,4,8\}, J=\{5,6,7,9,10\}$
  \item $I=\{1,2,3,6,8\}, J=\{4,5,7,9,10\}$
\end{enumerate}

We need to show that the pair $I,J$ that appear in cases (1) and (2) can be
equal and minimal, while the pair that appears in case (3) cannot be equal and
minimal. Let $Q=-2955617 + \sqrt{8665656785065}$. Then the following two
matrices provide the constructions for cases (1) and (2) respectively. In each
one of them, $\Delta_I=\Delta_J=1$, and $\Delta_U \geq 1$ for all $U\in
{[10]\choose 5}$.

$$
  \begin{pmatrix}
    1 & 0 & 0 & 0 & 0 & 1 & 6 & 53 & 98311 + \frac{Q}{124} & 237904 \\[.05in]
    0 & 1 & 0 & 0 & 0 & -1 & -5 & -36 & -32768 & -79343 \\[.05in]
    0 & 0 & 1 & 0 & 0 & 1 & 4 & 20 & -\frac{Q}{372} & -19+-\frac{Q}{186} \\[.05in]
    0 & 0 & 0 & 1 & 0 & -1 & -3 & -5 & -6 & -7 \\[.05in]
    0 & 0 & 0 & 0 & 1 & 1 & 1 & 1 & 1 & 1
  \end{pmatrix}
$$

$$
  \begin{pmatrix}
    1 & 0 & 0 & 0 & 0 & 1 & 5 & 25 & 265 & 318 \\[.05in]
    0 & 1 & 0 & 0 & 0 & -1 & -4 & -17 & -128 & -\frac{4869}{32} \\[.05in]
    0 & 0 & 1 & 0 & 0 & 1 & 3 & 10 & 43 & \frac{123761}{2480} \\[.05in]
    0 & 0 & 0 & 1 & 0 & -1 & -2 & -4 & -9 & -10 \\[.05in]
    0 & 0 & 0 & 0 & 1 & 1 & 1 & 1 & 1 & 1
  \end{pmatrix}
$$

We will now consider case (3). Assume in contradiction that there exists $M \in
Gr^{+}(5,10)$ for which $\Delta_I(M)=\Delta_J(M)=1$, and all the other
Pl\"ucker coordinates of $M$ are at least 1. Let $G$ be the plabic graph
appears in Figure~\ref{proofk5}.

\begin{figure}[h!]
\includegraphics[height=3in]{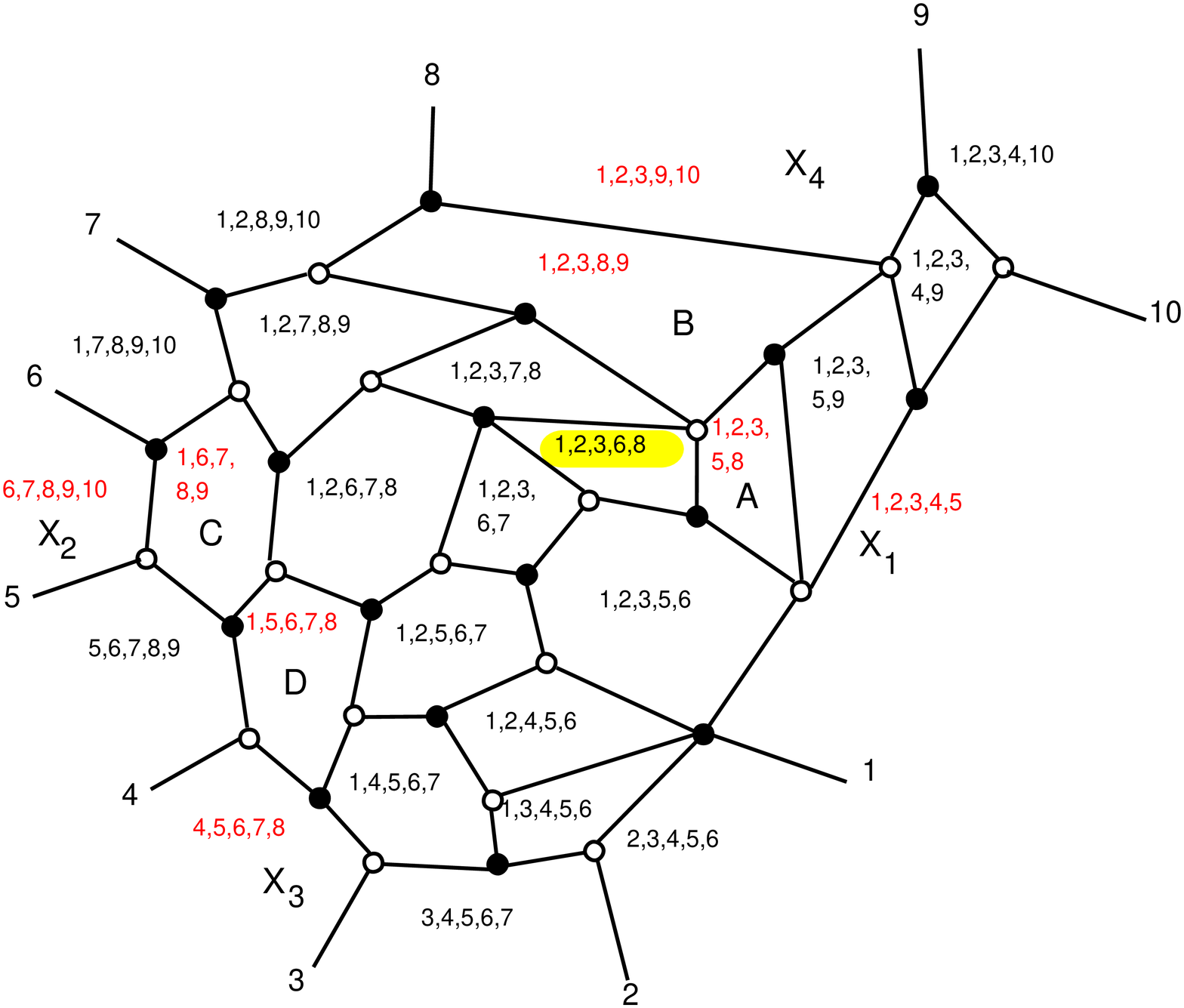}
\caption{The plabic graph $G$}
\label{proofk5}
\end{figure}

The faces of $G$ form a maximal weakly separated collection, and note that one
of the faces is labeled by $I$ (the face with the yellow background). We assume that $\Delta_I=1$, and assign 25
variables to the remaining 25 Pl\"ucker coordinates that correspond to the
faces of $G$ ($G$ has 26 faces). Among those 25 faces, 8 were of particular
importance for the proof, and we assign the following variables to the
corresponding 8 minors:
$\Delta_{\{1,6,7,8,9\}}=C$, $\Delta_{\{1,5,6,7,8\}}=D$, $\Delta_{\{1,2,3,8,9\}}=B$,
$\Delta_{\{1,2,3,5,8\}}=A$, $\Delta_{\{1,2,3,4,5\}}=X_1$, $\Delta_{\{6,7,8,9,10\}}=X_2$,
$\Delta_{\{4,5,6,7,8\}}=X_3$, $\Delta_{\{1,2,3,9,10\}}=X_4$ (those variables also appear
in the figure, where the relevant labels are written in red).

Recall that we
assume that all these variables are equal to or bigger than 1. By
Theorem~\ref{thm:cluster} and the discussion afterwards, any other Pl\"ucker
coordinate can be uniquely expressed through Laurent polynomials in those 25
variables with positive integer coefficients. Using the software {\tt Mathematica},
we expressed $\Delta_J$ in terms of these 25 variables.
The minor $\Delta_J$ is a sum of Laurent monomials\footnote{For the sake of brevity, we omitted here this
expression for $\Delta_J$.  The authors can provide it upon a request.}, and among
others, the following terms appear in the sum:
$X_1X_2\frac{DB}{AC}+X_3X_4\frac{AC}{DB}$. Note that since all the variables
are at least 1, we have

$$\Delta_J> X_1X_2\frac{DB}{AC}+X_3X_4\frac{AC}{DB} \geq
\frac{DB}{AC}+\frac{AC}{DB} > 1.
$$

Therefore, it is impossible to have $\Delta_I(M)=\Delta_J(M)=1$, and we are done.
\end{proof}

%In order to characterize the complimentary pairs $I,J\in {[8]\choose 4}$ that
%can belong to an arrangement of %smallest minors, we need to consider only
%pairs for which there exists a point in $Gr^+(4,8)$ such that the product
%%$\Delta_I\Delta_J$ is smaller than any other product of complimentary minors.
%Up to cyclic symmetry, the only such %pair is $I=\{1,2,3,6\}, J=\{4,5,7,8\}$.

\subsection{The $2\times 2$ honeycomb and an example of arrangement of smallest minors which is not weakly
separated}

We would like to explain how we constructed the matrices in the proof above,
using properties of plabic graphs.  We think that these properties may be
generalized and lead to the proof of Conjecture~\ref{conjecture:minimal_minors}.
In addition, these properties also reveal a quite remarkable structure of
plabic graphs that is interesting on its own.

\begin{figure}[h!]
\includegraphics[height=2.2in]{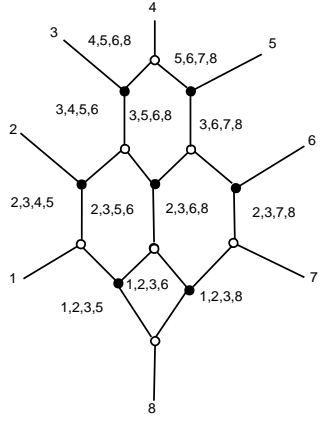}
\caption{The $2\times 2$ honeycomb}
\label{plabic1236}
\end{figure}

Let us first consider the case $k=4$.
Consider the plabic graph $G$ in Figure~\ref{plabic1236}.
The 12 faces of $G$ form a weakly separated collection $C=F(G)$, and one of the
faces (the square face) is labelled by $I=\{1,2,3,6\}$ (which is the minor that appeared in the
proof for the case $k=4$). Consider the four bounded faces in $G$. They
consists of a square face labeled with $I$, and 3 additional hexagonal faces.
We call such a plabic graph the $2\times 2$ honeycomb.
(We will show later how to generalize it.)
One way to complete $C=F(G)$ to a maximal weakly separated collection $C'$ in ${[8]\choose 4}$
is $C'=C \cup \{\{1,2,3,4\},\{4,5,6,7\},\{1,6,7,8\},\{1,2,7,8\},\{1,3,7,8\}\}$.

Assign the variable $T$ to the Pl\"ucker coordinates associated with the 3 hexagonal faces mentioned above, and
assign the value 1 to the Pl\"ucker coordinates of the rest of the faces in $C'$.
Using the software {\tt Mathematica}, we expressed all the other Pl\"ucker coordinates
$\Delta_K$, $K\in {[8]\choose 4}\setminus C'$, as functions (positive Laurent polynomials)
of $T$.  We checked that, for all $K \in {[8]\choose 4}$ such that $K \neq J=\{4,5,7,8\}$, the
Laurent polynomials that corresponds to $\Delta_K$ has either the summand $1$
or the summand $T$.
Therefore, if we require $T \geq 1$, then $\Delta_K \geq 1$
for all $K \neq \{4,5,7,8\}$. Finally, $\Delta_{\{4,5,7,8\}}=\frac{6}{T}$.
Therefore, by choosing $T=6$, we get an element in $Gr^+(4,8)$ for which
$\Delta_I=\Delta_J=1$.

The matrix in the proof is exactly the matrix that
corresponds to the construction described above. Moreover, the collection of
smallest minors in this matrix consists of 15 minors that correspond to
$C'\setminus \{\{2,3,5,6\},\{2,3,6,8\},\{3,5,6,8\}\} \cup \{4,5,7,8\}$.
We verified that this is a maximal arrangement of smallest minors.

\begin{remark}
Conjecture~\ref{conj:max_smallest} states that, for $k=4$, $n=8$ any maximal (by
size) arrangement of smallest minors is weakly separated and has 17 elements.
Here we constructed a maximal (by containment, but not by size) arrangement of smallest
minors $C'\setminus \{\{2,3,5,6\},\{2,3,6,8\},\{3,5,6,8\}\} \cup \{4,5,7,8\}$
that has 15 elements and contains a pair $I,J$, which is not weakly separated.
\end{remark}

\subsection{Mutation distance and chain reactions}

%The following notion is related to the above construction.
%, but it is also interesting on its own.

\begin{definition}
Let $I,J\in {[n]\choose k}$ be any two $k$-element subsets in $[n]$.
Define the {\it mutation distance\/} $D(I,J)$ as the minimal number of square moves (M1) needed
to transform a plabic graph $G$ that contains $I$ as a face label into a plabic graph $G'$ that
contains $J$ as a face label.  (The moves (M2) and (M3) do not contribute to the mutation distance.)
\end{definition}

Clearly, $D(I,J)=0$ if and only if $I$ and $J$ are weakly separated.
Indeed, any two weakly separated $k$-element subsets can appear as face labels in the same plabic graph.
The number $D(I,J)$ measures how far $I$ and $J$ are from being weakly separated.

Below we give several examples of pairs $I,J$ and shortest chains of square moves
between plabic graphs containing $I$ and $J$, respectively.

%\begin{problem} How to calculate the mutation distance between any  $I$ and $J$, and how to
%find a shortest chain of square moves between plabic graphs containing these subsets?
%\end{problem}

%Note that this problem can be extended in a general setting of cluster algebras as finding the
%mutation distance between two given cluster variables.

%\medskip

\begin{example}
In the previous subsection, we constructed an arrangement of smallest minors that included the
non weakly separated pair $I=\{1,2,3,6\}$ and $J=\{4,5,7,8\}$. In order to calculate $D(I,J)$,
let us describe a shortest chain of square moves between a pair of plabic graphs that contain
$I=\{1,2,3,6\}$ and $J=\{4,5,7,8\}$, respectively.  Since $I$ and $J$ are not weakly
separated, they cannot appear as face labels of the same plabic graph.  We
start with the plabic graph shown in Figure~\ref{plabic1236} (the $2\times 2$
honeycomb) that contains $I$ as the label of its square face.  We want to transform
it into another plabic that contains $J$ as a face label using minimal possible number
of square moves.  In order to do this, we first apply a square move (M1) on the
face $I=\{1,2,3,6\}$.  Then apply square moves on faces $\{2,3,4,6\}$ and
$\{2,3,6,7\}$ (those faces become squares after appropriate moves of type (M2),
so it is possible to apply a square move on them). Finally, apply a square move
on the face $\{3,4,6,7\}$.  The result is exactly $J=\{4,5,7,8\}$.

We verified, using a computer, that this is indeed a shortest chain of moves
that ``connects'' $I$ with $J$.  Moreover, this is the only shortest chain of moves
for this pair of subsets.
Therefore, $D(I,J)=4$ in this case.
\end{example}

The sequence of moves in the above example can be generalized as follows.   Pick a pair $I,J$ with
length parameters $(\alpha_1,\beta_1,\alpha_2,\beta_2)$ as in case (2) of
Conjecture~\ref{conjecture:minimal_minors} such that $\alpha_2=1$.  Consider
the $\beta_1\times \beta_2$ honeycomb. The structure of such honeycomb
should be clear from examples on Figures~\ref{plabic1236}, \ref{chainreaction}, and~\ref{larghoneycomb}.
This $\beta_1\times \beta_2$ honeycomb has $\beta_1\cdot \beta_2 - 1$ hexagonal
faces, and one square face on the bottom with label $I$.  The square face
serves as a ``catalyst'' of a ``chain reaction'' of moves.  First, we apply a
square move (M1) for $I$.  This transforms the neighbouring hexagons into
squares (after some (M2) moves).  Then we apply square moves for these new
squares, which in turn transforms their neighbours into squares, etc.   In the
end, we obtain a new honeycomb with all hexagonal faces except one square face
on the top with label $J$.

% would like to understand the relation between
%$J=\{4,5,7,8\}$ and the honeycomb in Figure~\ref{plabic1236}.
%First of all, note that the length parameters of $P(I,J)$ are
%$(\alpha_1,\beta_1,\alpha_2,\beta_2)=(3,2,1,2)$, and keep in mind that the
%honeycomb is of order $\beta_1 \times \beta_2$.
%Consider again the graph $G$ in Figure~\ref{plabic1236} and the $2\times 2$
%honeycomb appears in it.

\begin{example}
Figure~\ref{chainreaction} presents an example for the pair $I=\{1,2,3,4,8\}$
and $J=\{5,6,7,9,10\}$.  The length parameters are
$(\alpha_1,\beta_1,\alpha_2,\beta_2)=(4,3,1,2)$.  In this example, the face $A$
of the first honeycomb has label $I$ and the face $F'$ of the last honeycomb
has label~$J$.    Figure~\ref{chainreaction}  shows a shortest chain of square moves
of length $D(I,J)=6$ ``connecting'' $I$ and $J$.
\end{example}

\begin{figure}[h!]
\includegraphics[height=7in]{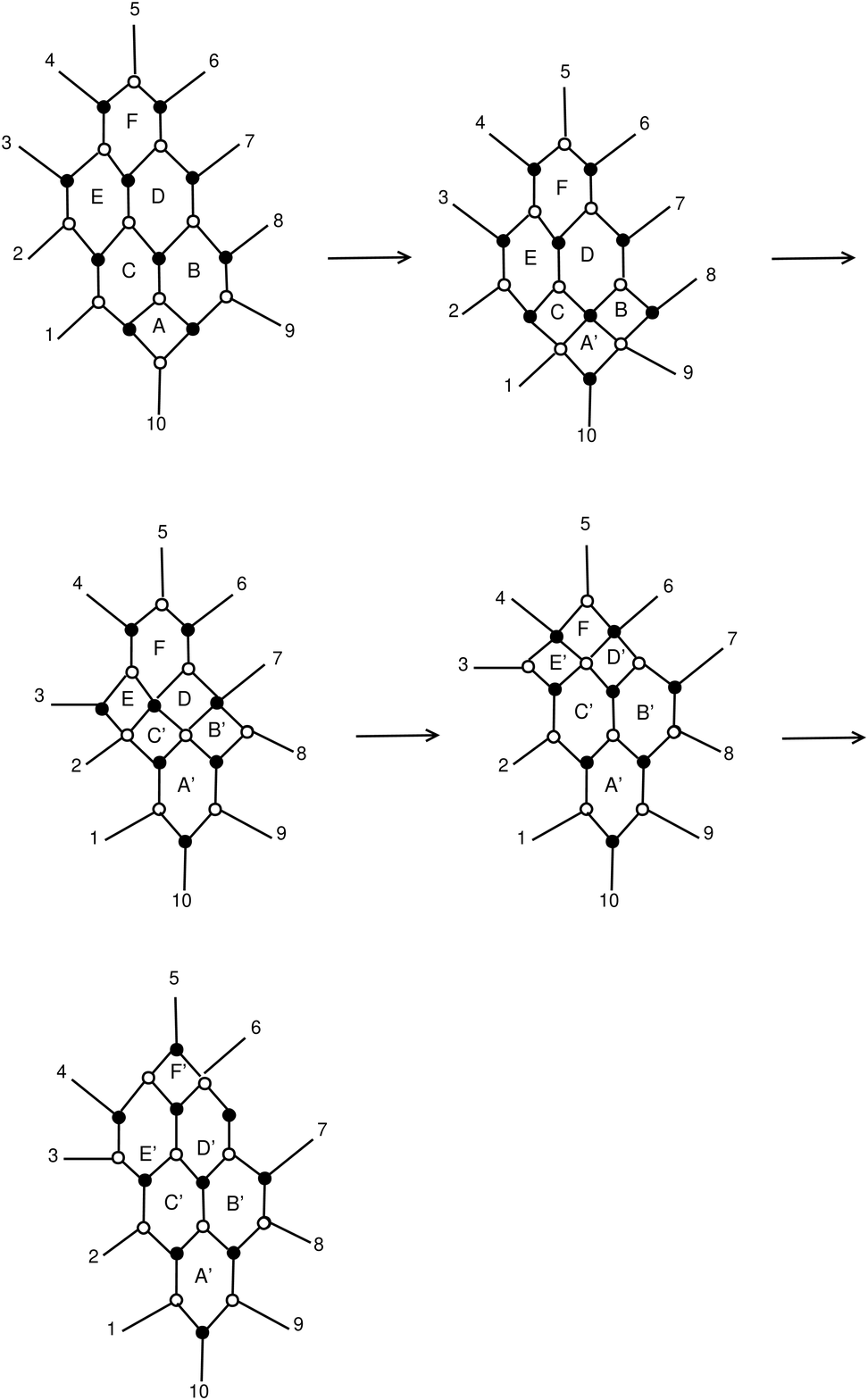}
\caption{The chain reaction in the $3\times 2$ honeycomb.}
\label{chainreaction}
\end{figure}

\begin{conjecture}
Let $G$ be a reduced plabic graph with the strand permutation
$\pi_G(i)=i+k\pmod n$ that contains an $a\times b$ honeycomb $H$
%(as described above)
as a subgraph.  Let $I$ be the label of the square face of the honeycomb
$H$, and $J$ be the label of the square faces of the honeycomb $H'$ obtained
from $H$ by the chain reaction.  Assign the value  $T$ to the Pl\"ucker coordinates
corresponding to the hexagons in the honeycomb $H$, and the value 1
to the Pl\"ucker coordinates of the rest of the faces of $G$ (including $\Delta_I=1$).
Express any other Pl\"ucker coordinate $\Delta_K$ as a Laurent polynomial in $T$ with positive
integer coefficients.
Then the degree of the Laurent polynomial, for any $\Delta_K$, $K\ne J$, is at least 0;
that is, it contains at least one term $T^a$ with $a\geq 0$.
Also the degree of the polynomial for $\Delta_J$ is at most $-1$;
that is, it only contains terms $T^b$ with $b\leq -1$.
\label{conj:smallest_equal_pairs}
\end{conjecture}

This conjecture means that there exists a unique positive value of $T$ such
that $\Delta_I=\Delta_J=1$, and all the other Pl\"ucker coordinates $\Delta_K\geq
1$.  This provides a construction of matrix for an arrangement of smallest
minors containing $I$ and $J$, for any pair $I,J$ as in part (2) of
Conjecture~\ref{conjecture:minimal_minors} with $\alpha_2=1$.

%transition from $I$ to $J$ can be generalized for any pair $I,J$ that satisfies
%condition (2), and for which $\beta_1=1$. We now present the chain reaction
%explicitly, for a larger honeycomb. Consider the honeycomb that appears in the
%upper-left part of Figure ~\ref{chainreaction}. In this plabic graph, the
%labeling of the face $A$ is $\{1,2,3,4,8\}$. We will show that after applying
%the chain reaction, we obtain the face $\{5,6,7,9,10\}$. Note that the length
%parameters of $P_T(\{1,2,3,4,8\},\{5,6,7,9,10\})$ are $(4,3,1,2)$, so in this
%case the honeycomb is of order $\beta_1\times \beta_2$ = $3\times 2$. Let us
%describe the 5 steps depicted in Figure~\ref{chainreaction}.

%In the first step, we apply square move (M1) on the face $A$ and obtain the
%face $A'$. Then, in the second step, we apply square move on the faces $B$ and
%$C$ (possibly after some moves of the type (M2),(M3)) to obtain the faces
%$B',C'$. We then apply square moves on faces $E$ and $D$, and obtain the faces
%$E'$ and $D'$. Finally, we apply square move on the face $F$, and obtain the
%face $F'$. The labeling of $F'$ is exactly $\{5,6,7,9,10\}$. Note that there is
%a symmetry in the process. We could start from the plabic graph in the final
%step, and apply the chain backwards.

\begin{example}
Let us give another example for the case $\alpha_2=1$.  The $4\times 3$ honeycomb that
appears in Figure~\ref{larghoneycomb} corresponds to the pair
$I=\{1,2,3,4,5,6,11\}$ and $J=\{7,8,9,10,12,13,14\}$.  The length parameters of
$P(I,J)$ are $(6,4,1,3)$.   In this case we need $D(I,J)=12$ mutations.
%%so the order of this honeycomb is $4\times 3$.
\end{example}

\begin{figure}[h!]
\includegraphics[height=2.5in]{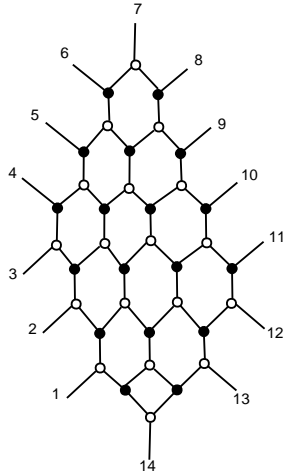}
\caption{The $4\times 3$ honeycomb.}
\label{larghoneycomb}
\end{figure}

\begin{figure}[h!]
\includegraphics[height=1.6in]{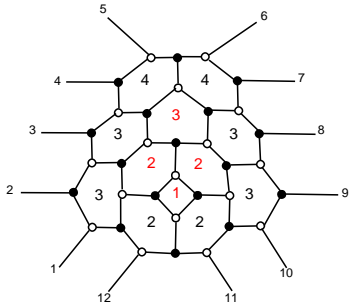}
\caption{A honeycomb with one layer.}
\label{fig:twolayerdhoneycomb}
\end{figure}

\begin{example}
Let us give an example for the case $\alpha_2=2$. Consider the
pair $I=\{1,2,3,4,8,9\}$, $J=\{5,6,7,10,11,12\}$. The length parameters of
$P(I,J)$ are $(\alpha_1,\beta_1,\alpha_2,\beta_2)=(4,3,2,3)$. We can obtain
the face $J$ via a chain reaction that starts with a plabic graph that contains
the face $I$ as follows.

Consider the plabic graph in Figure~\ref{fig:twolayerdhoneycomb}.  This plabic
graph consists of a $2\times 2$ honeycomb surrounded with one ``layer'' of
hexagonal faces. In this plabic graph, the square face (denoted by 1) has label
$I$. The chain reaction that enables us to obtain the face $J$ is the
following.
First, apply a square move on face 1.
Then (after some moves of type (M2)) apply square moves on faces 2 (in any
order). We continue with square moves on the faces denoted by 3 and then the faces denoted by 4. After this iteration, we
apply the chain reaction again, this time only on the internal faces
(with red labels).
Then the face denoted by 3 (in red color) will
have the label $J$.   We need $D(I,J)=16$ square moves.

In order to obtain an arrangement of smallest minors that
contains both $I$ and $J$, one can complete the graph $G$ in
Figure~\ref{fig:twolayerdhoneycomb} to a maximal weakly separated set and assign
the following values to its Pl\"ucker coordinates.
Assign the value 1 to all the coordinates that do not appear in $G$, and also to the square face of $G$.
Assign the value $T$ to the coordinates in $G$ that correspond to the ``layer.''
Assign the value $T^2$ to the coordinates of the $2\times 2$ honeycomb (shown in red),
excluding the square face.
We checked, using a computer, that there exists a unique $T$ for which
$\Delta_I$ and $\Delta_J$ are equal and minimal.
\end{example}

\subsection{Square pyramids and the octahedron/tetrahedron moves}

We conclude with a brief discussion of an alternative
geometric description for the chain reactions of honeycomb plabic graphs.
The objects described below are special cases of {\it membranes\/}
from the forthcoming paper \cite{LP2}. They are certain surfaces
associated with plabic graphs.

%In this work, plabic graphs are
%associated with certain 2-dimensional surfaces, called {\it membranes,} in $\R^m$.

%, and plabic graph (or
%more accurately, graphs that called the dual graphs of plabic graphs) are a
%special case of such surfaces.\\

Define the following map $\pi_I$ from weakly separated sets to $\R^4$.
Subdivide $[n]$ into a disjoint union of four intervals
$[n]=T_1\cup T_2\cup T_3 \cup T_4$ such that
$T_1 =[1,a]$, $T_2 = [a+1, b]$, $T_3=[b+1,c]$, $T_4=[c+1,n]$, for some $1\leq a< b< c< n$.
Assume that $I=T_1\cup T_3\in {[n]\choose k}$.  Then $[n]\setminus I=T_2\cup T_4$.
Let $\pi_I: {[n]\choose k} \rightarrow \R^4$ be the projection given by
$$
\pi_I(W)=(|W \cap T_1|,|W \cap T_2|,|W \cap T_3|,|W \cap T_4|).
$$
For example, $\pi_I(I)=(a,0,c-b,0)$.

%Since for all $W \in {[n]\choose k}$, the sum of the elements in $\pi_I(W)$ is
%$k$.
The image of $\pi_I(W)$ belongs to the 3-dimensional hyperplane
$\{x_1+x_2+x_3+x_4=k\}\simeq \R^3$ in $\R^4$.

For a plabic graph $G$ (whose face labels $W\in {[n]\choose k}$ form a weakly
separated set $F(G)$), the map $\pi_I$ maps the elements $W\in F(G)$ into integer
points on a 2-dimensional surface in $\R^3$.

\begin{figure}[h!]
\includegraphics[height=1.3in]{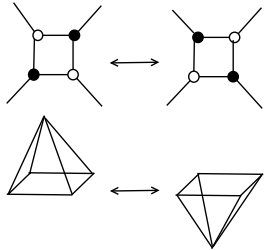}
\caption{the octahedron move}
\label{octahedronmove}
\end{figure}

\begin{figure}[h!]
\includegraphics[height=1.3in]{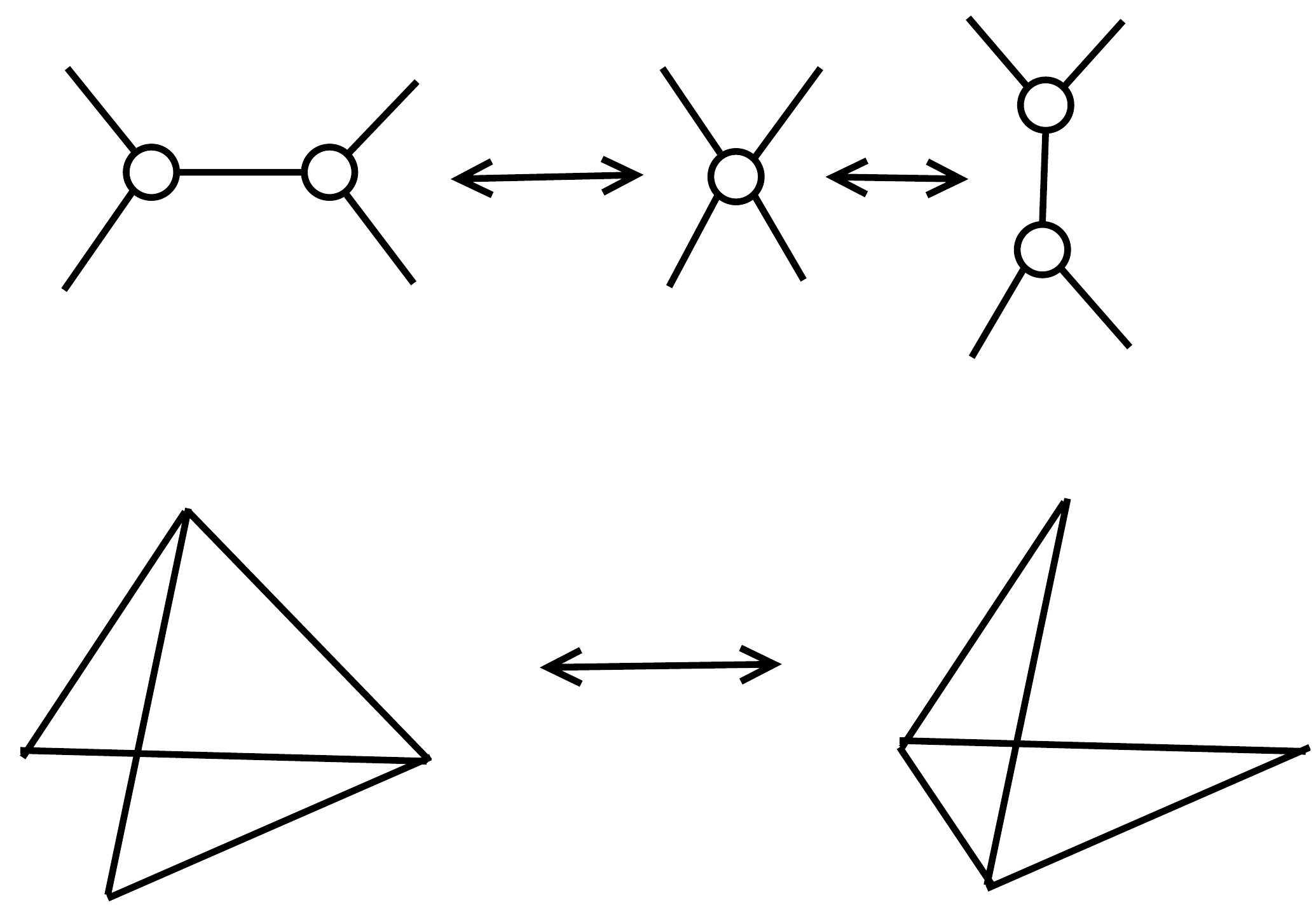}
\caption{the tetrahedron move}
\label{tetrahedronaction}
\end{figure}

The map $\pi_I$ transforms the moves (M1) and (M2) of plabic graphs
to the ``octahedron move'' and the ``tetrahedron move'' of the corresponding
2-dimensional surfaces, as shown on Figures~\ref{octahedronmove}
and~\ref{tetrahedronaction}.  For example, the ``octahedron move'' replaces a
part of the surface which is the upper boundary of an octahedron by the lower
part of the octahedron.  (This construction is a special case of a more general
construction that will appear in full details in \cite{LP2}.)

 \begin{figure}[h!]
    \subfloat[Initial surface]{{\includegraphics[height=1.6in]{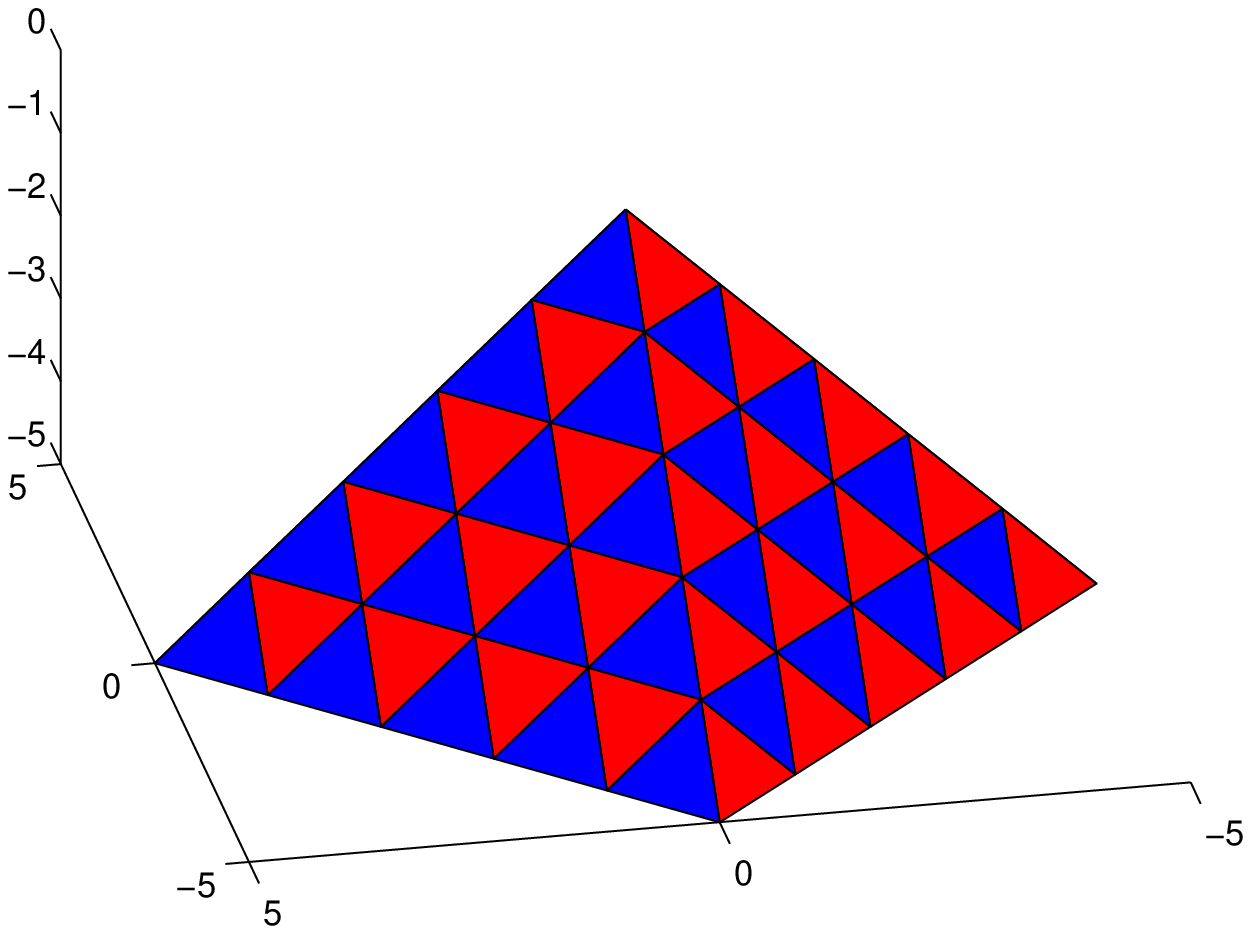} }}%
    \qquad
    \subfloat[1 octahedron move]{{\includegraphics[height=1.6in]{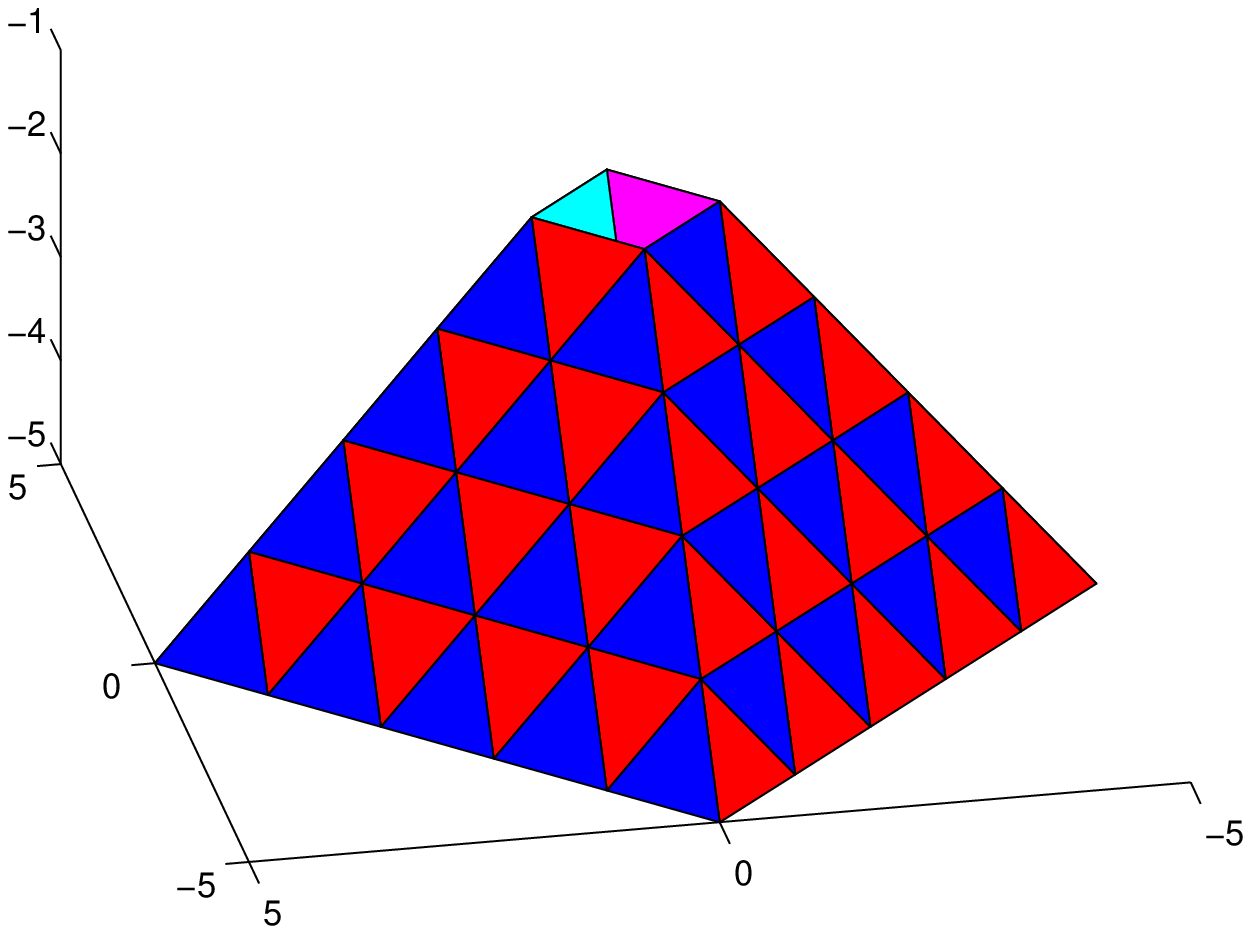} }}%
     \qquad
    \subfloat[3 tetrahedron moves]{{\includegraphics[height=1.6in]{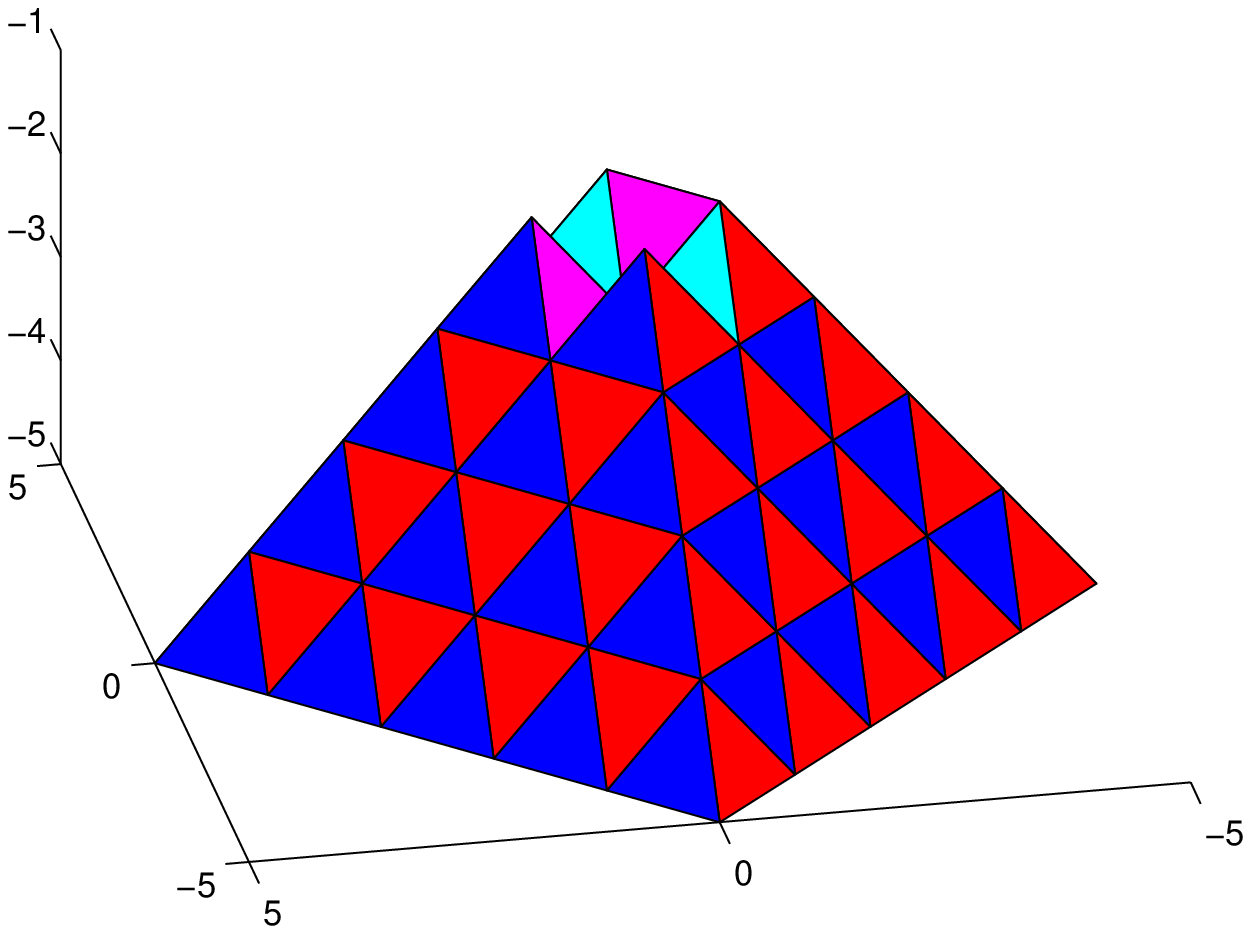} }}%
     \qquad
    \subfloat[2 octahedron moves]{{\includegraphics[height=1.6in]{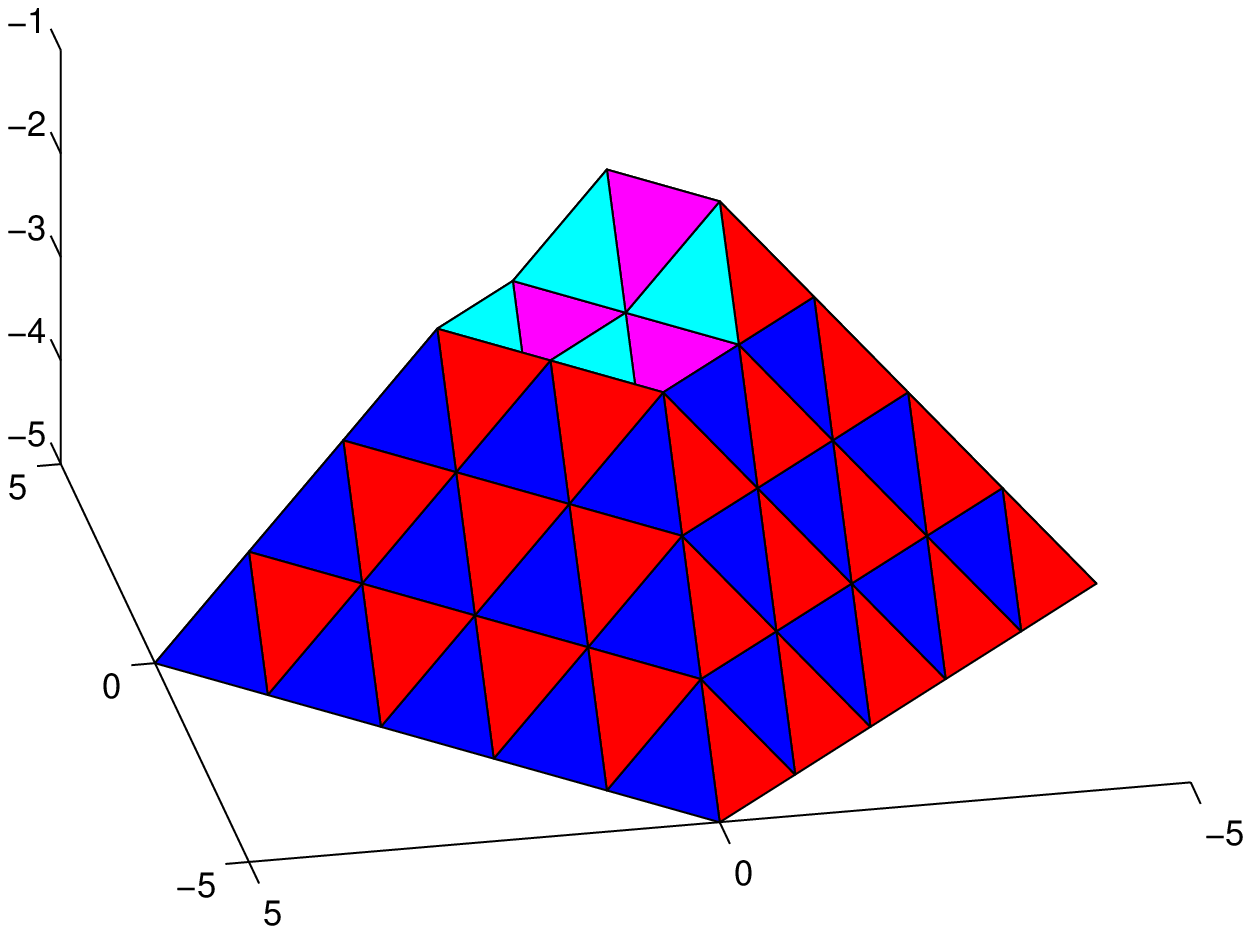} }}%
     \qquad
    \subfloat[4 tetrahedron moves]{{\includegraphics[height=1.6in]{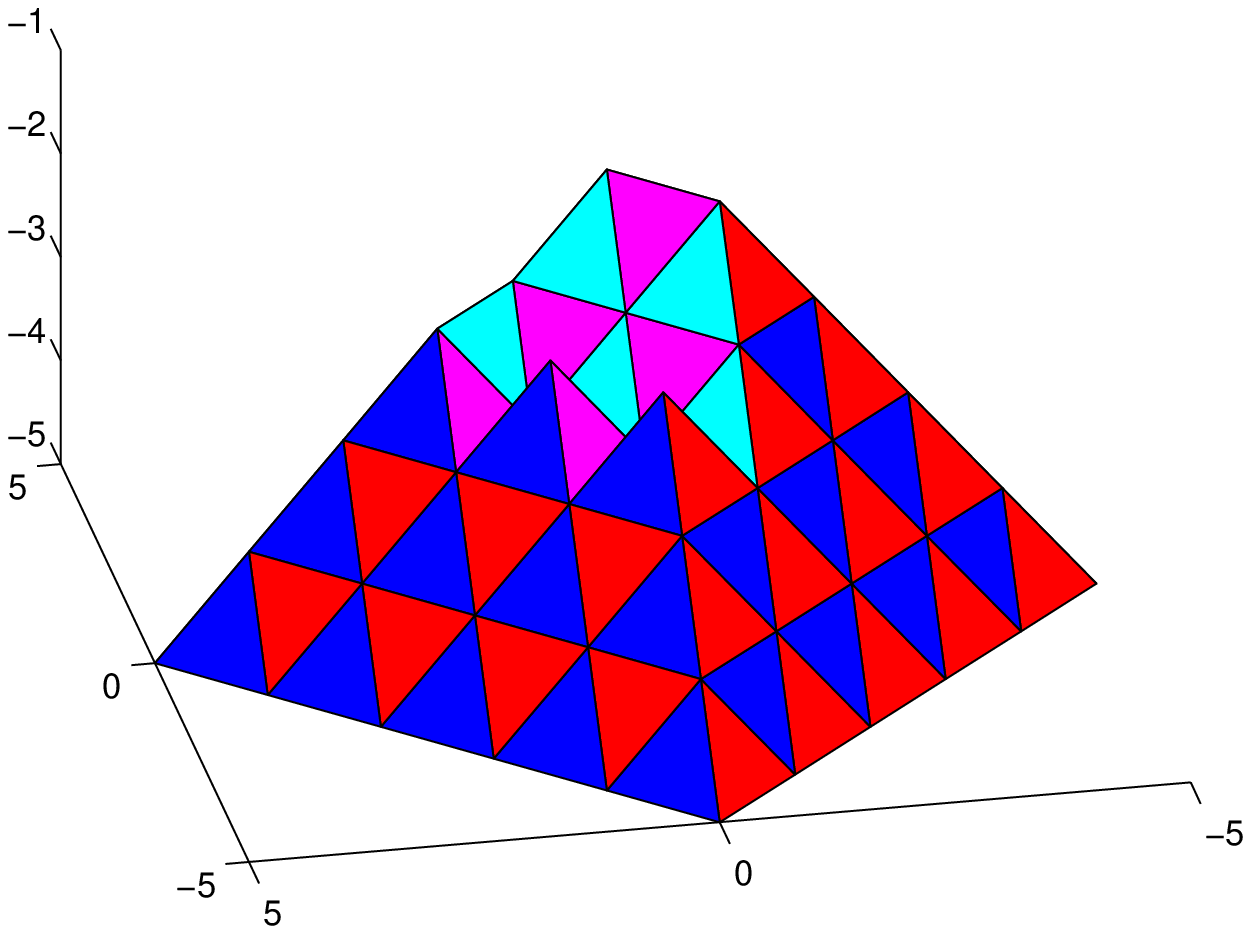} }}%
     \qquad
    \subfloat[2 octahedron moves]{{\includegraphics[height=1.6in]{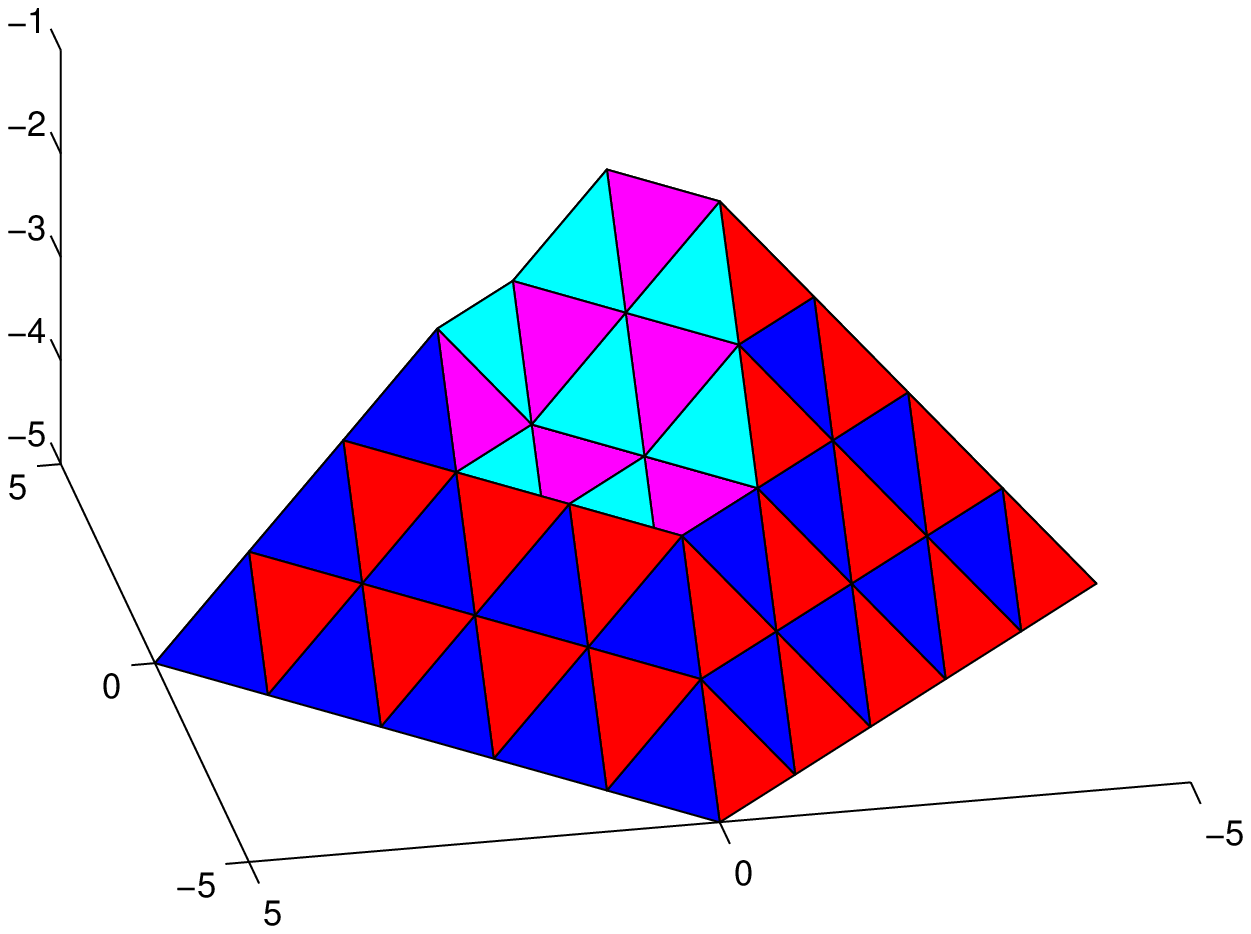} }}%
     \qquad
    \subfloat[3 tetrahedron moves]{{\includegraphics[height=1.6in]{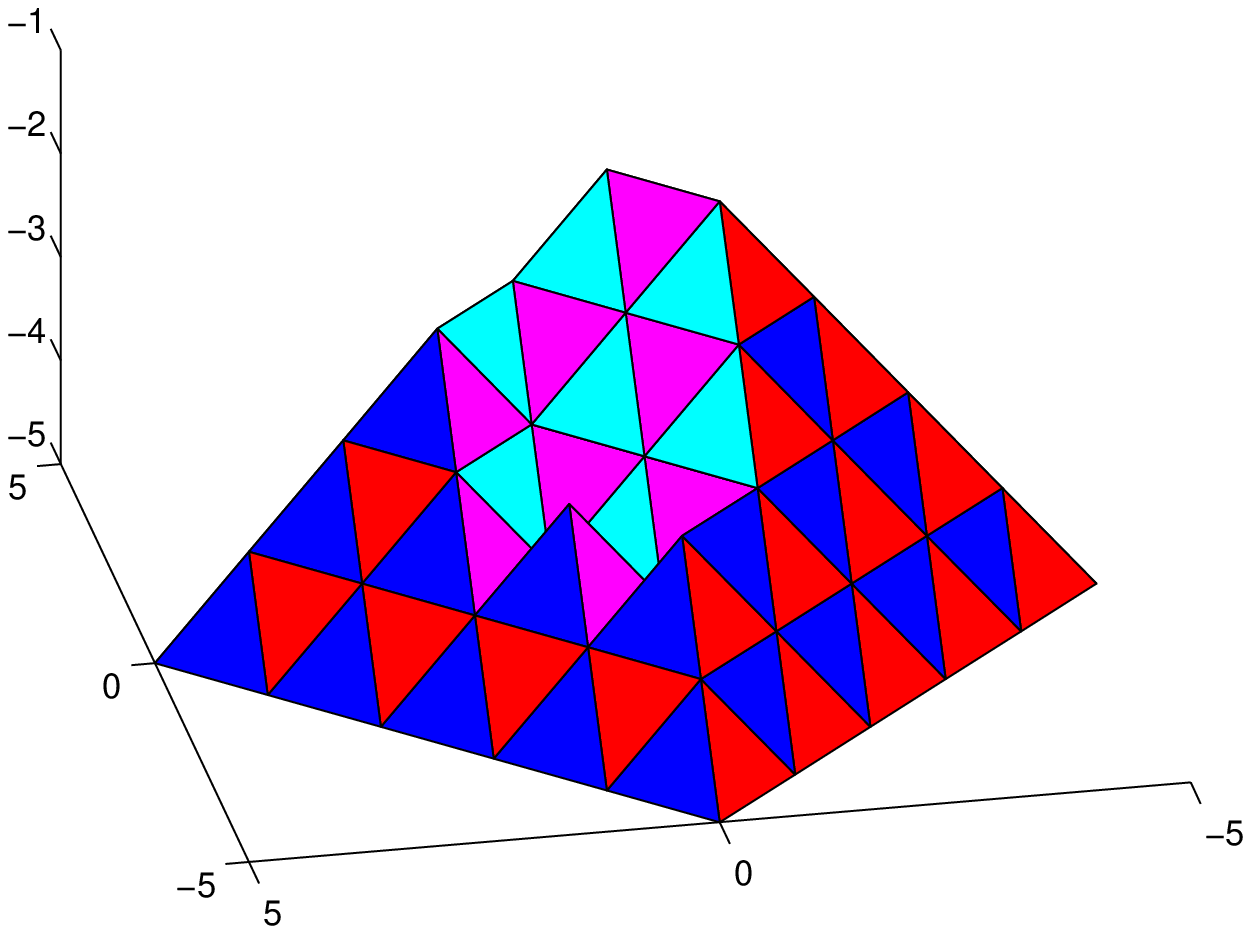} }}%
     \qquad
    \subfloat[1 octahedron move]{{\includegraphics[height=1.6in]{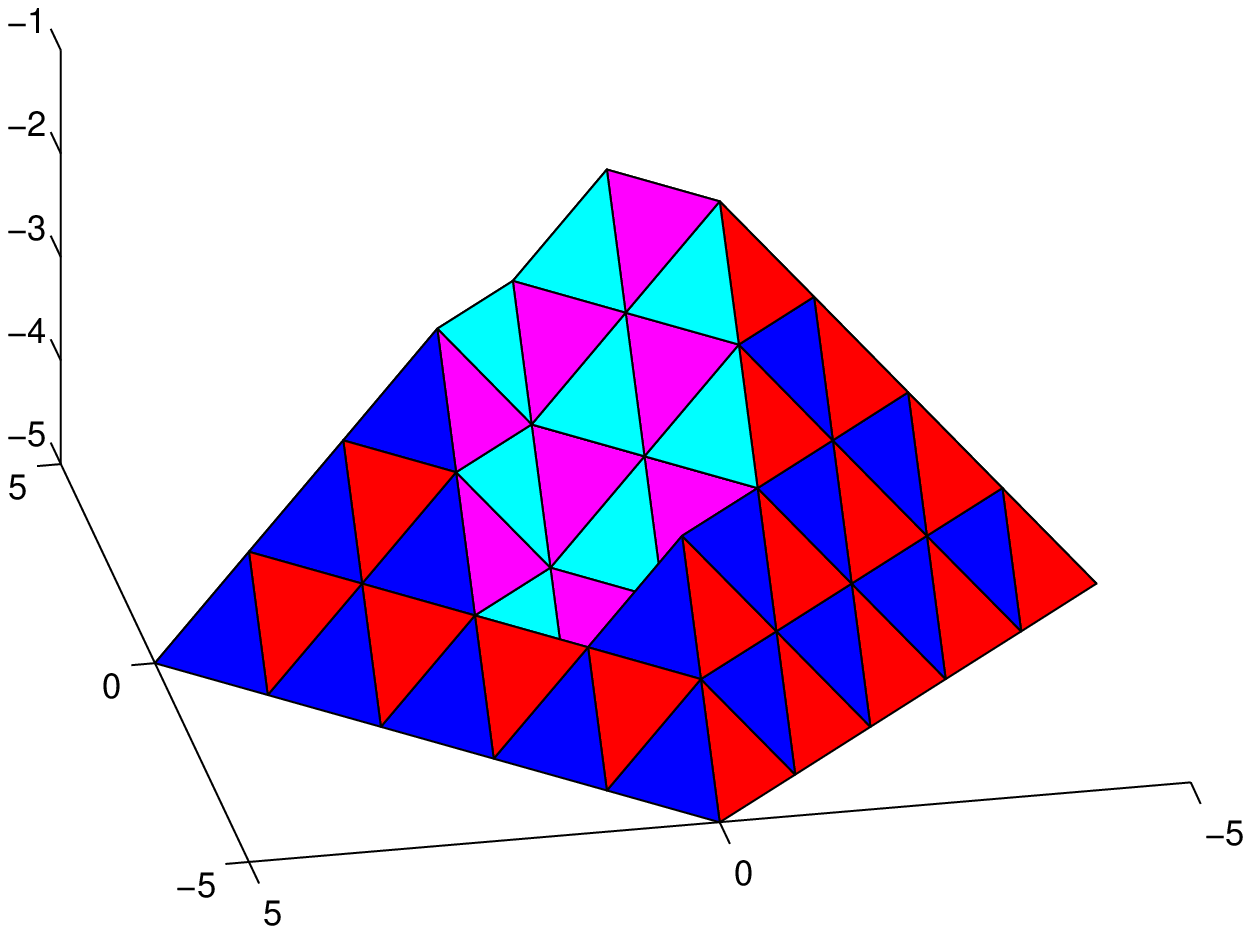} }}%
    \caption{The chain reaction in a $3\times 2$ honeycomb,
described using octahedron and tetrahedron moves.}%
    \label{chainreaction2}
\end{figure}

\begin{figure}[h!]
    \subfloat{{\includegraphics[height=1.6in]{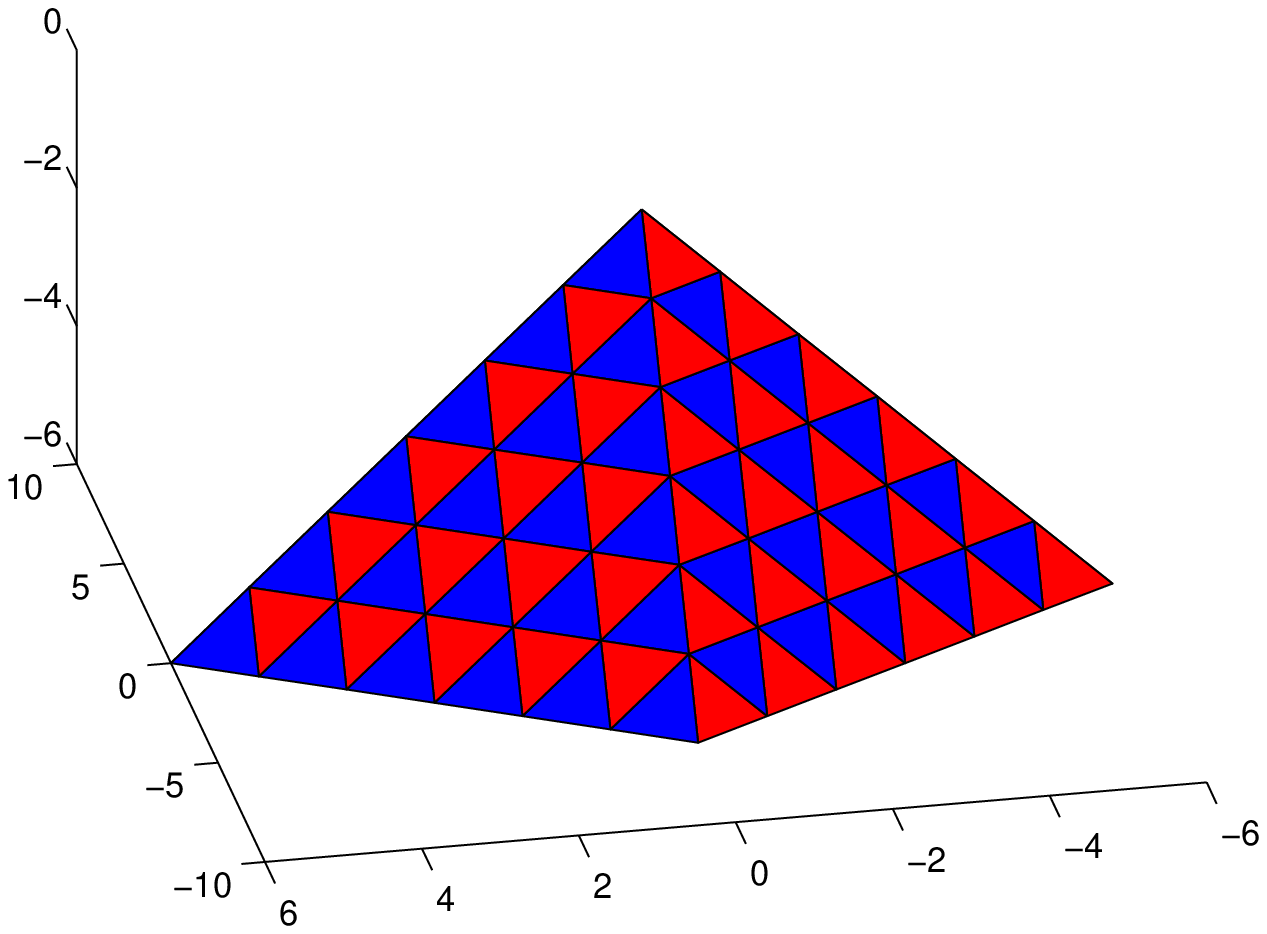} }}%
    \qquad
    \subfloat{{\includegraphics[height=1.6in]{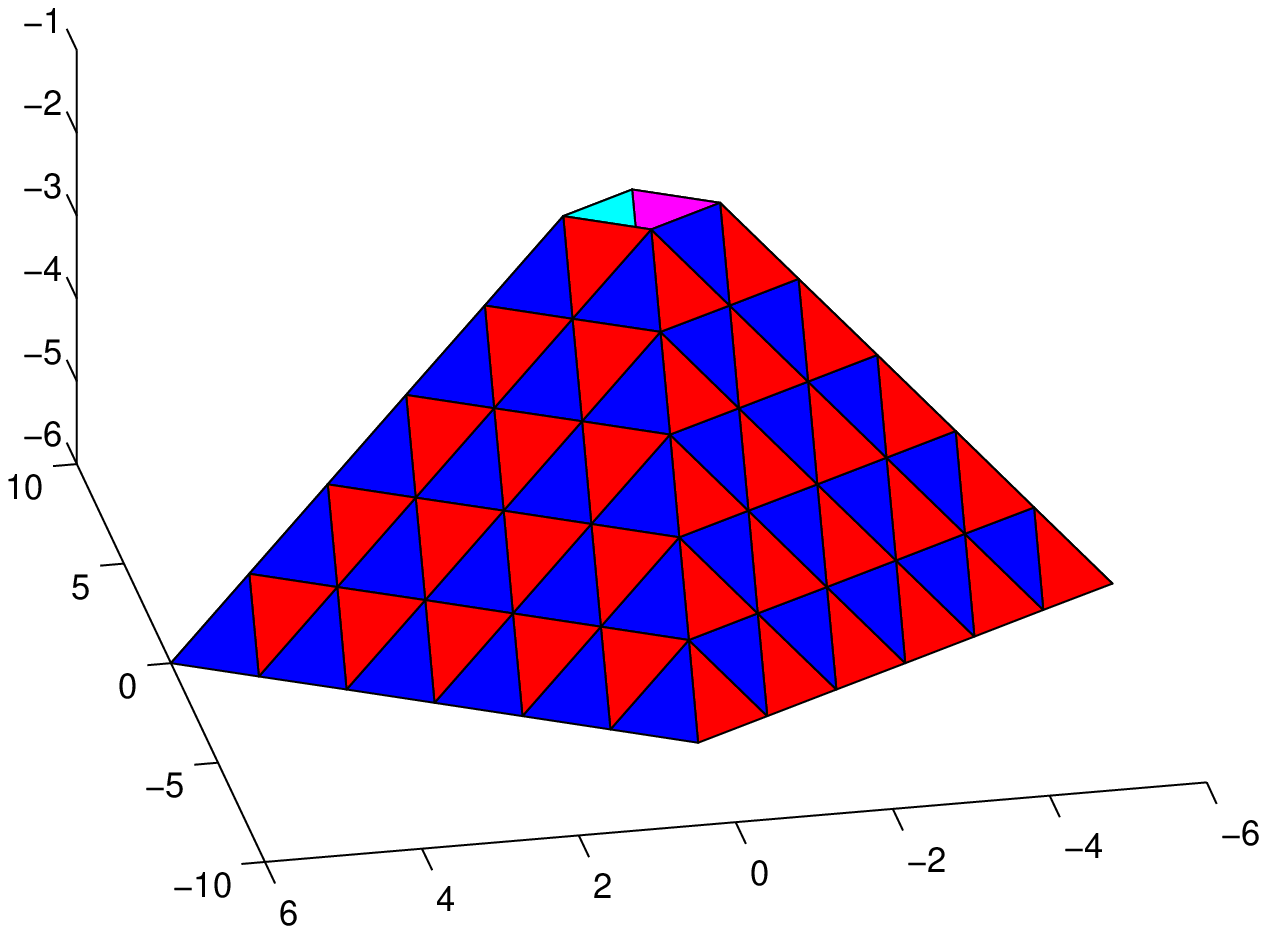} }}%
     \qquad
    \subfloat{{\includegraphics[height=1.6in]{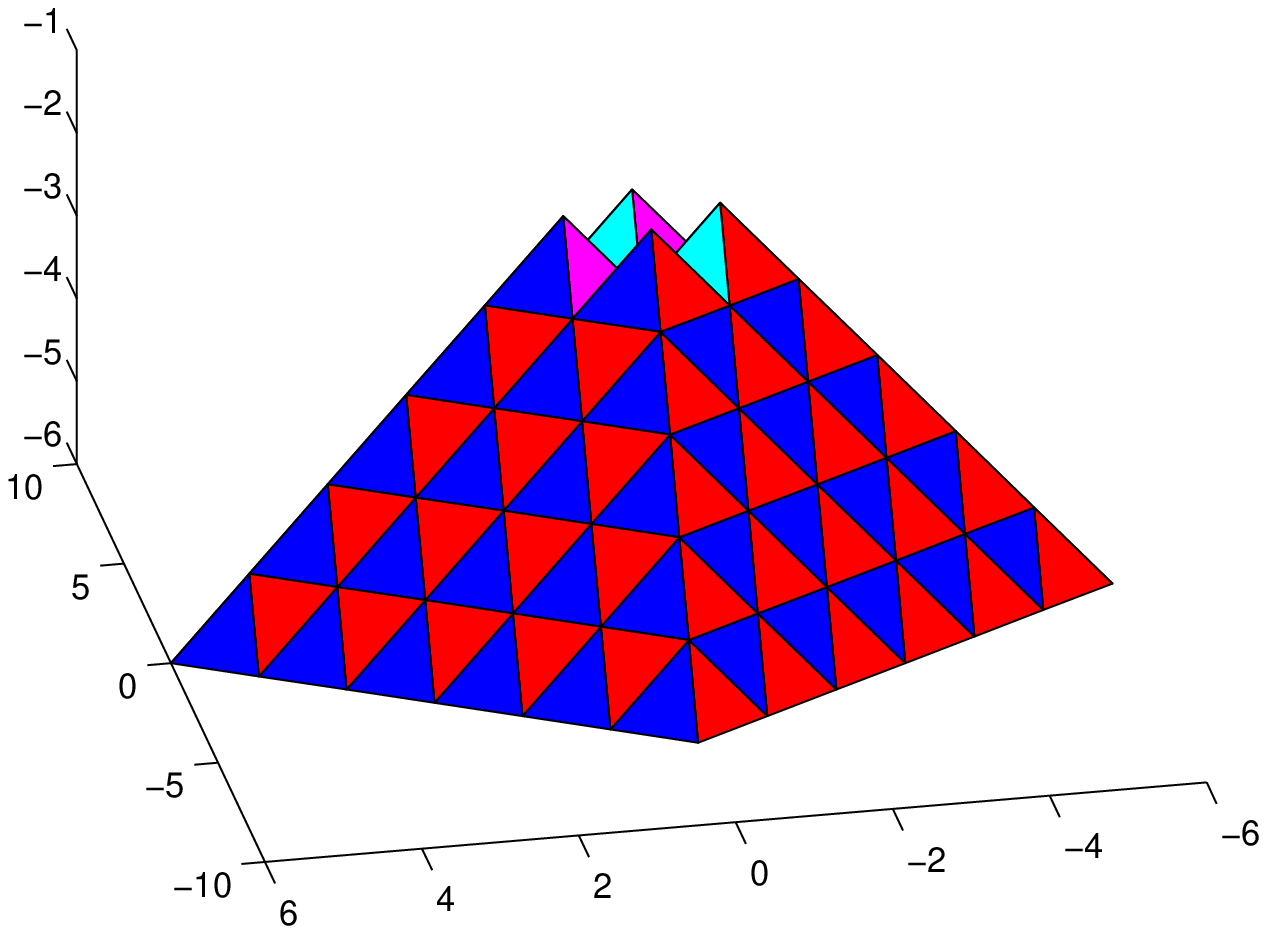} }}%
     \qquad
    \subfloat{{\includegraphics[height=1.6in]{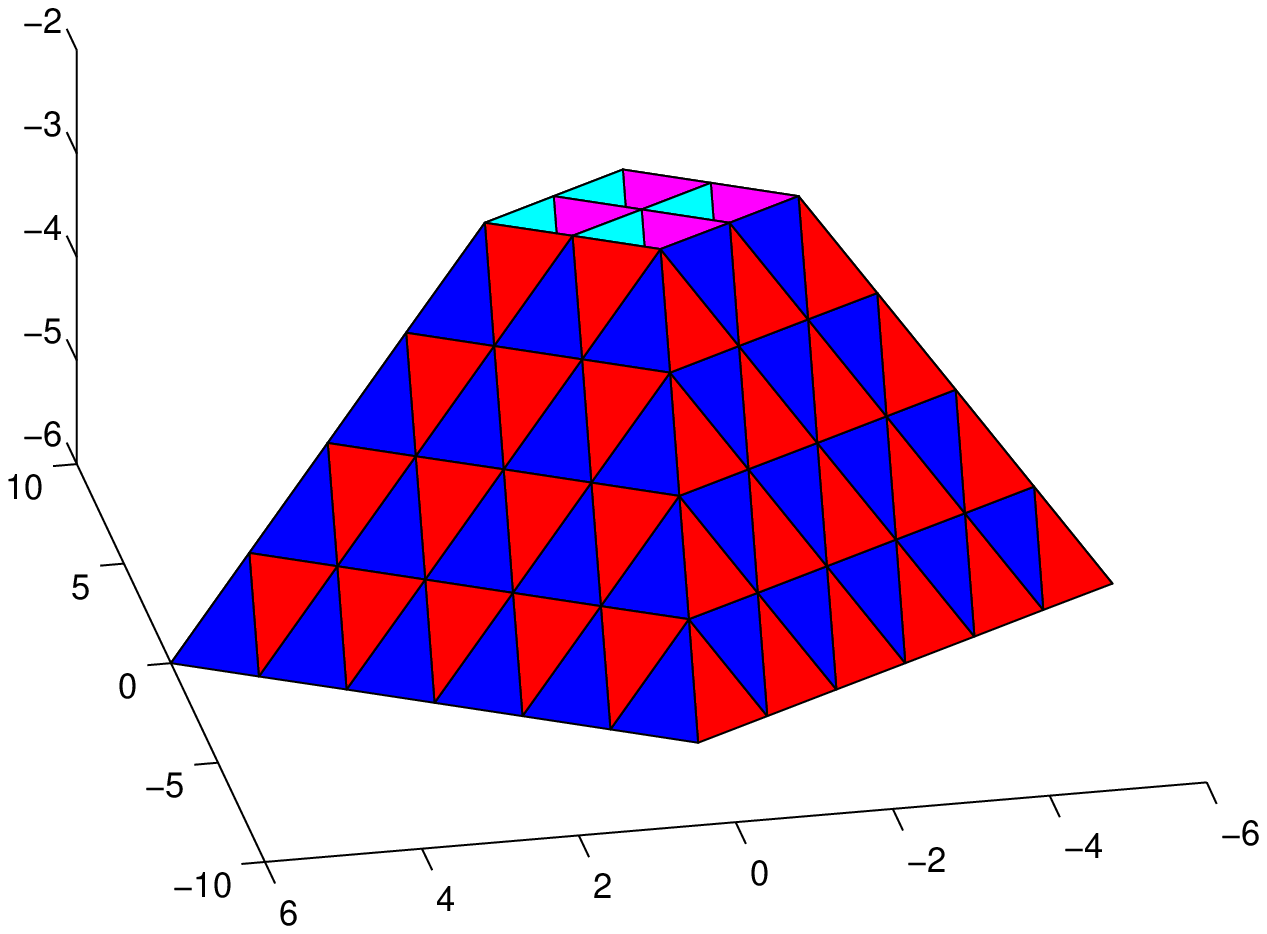} }}%
     \qquad
    \subfloat{{\includegraphics[height=1.6in]{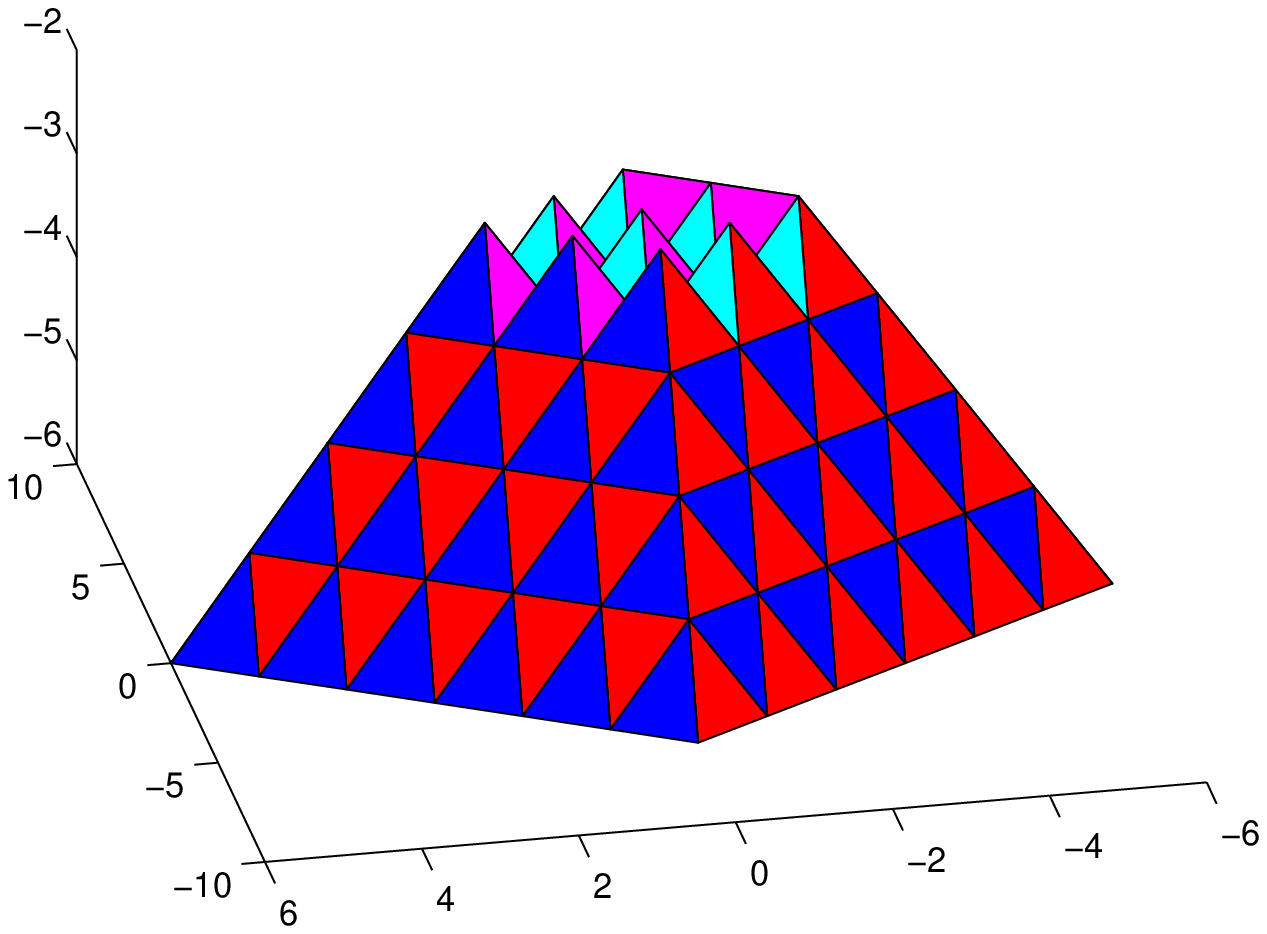} }}%
     \qquad
    \subfloat{{\includegraphics[height=1.6in]{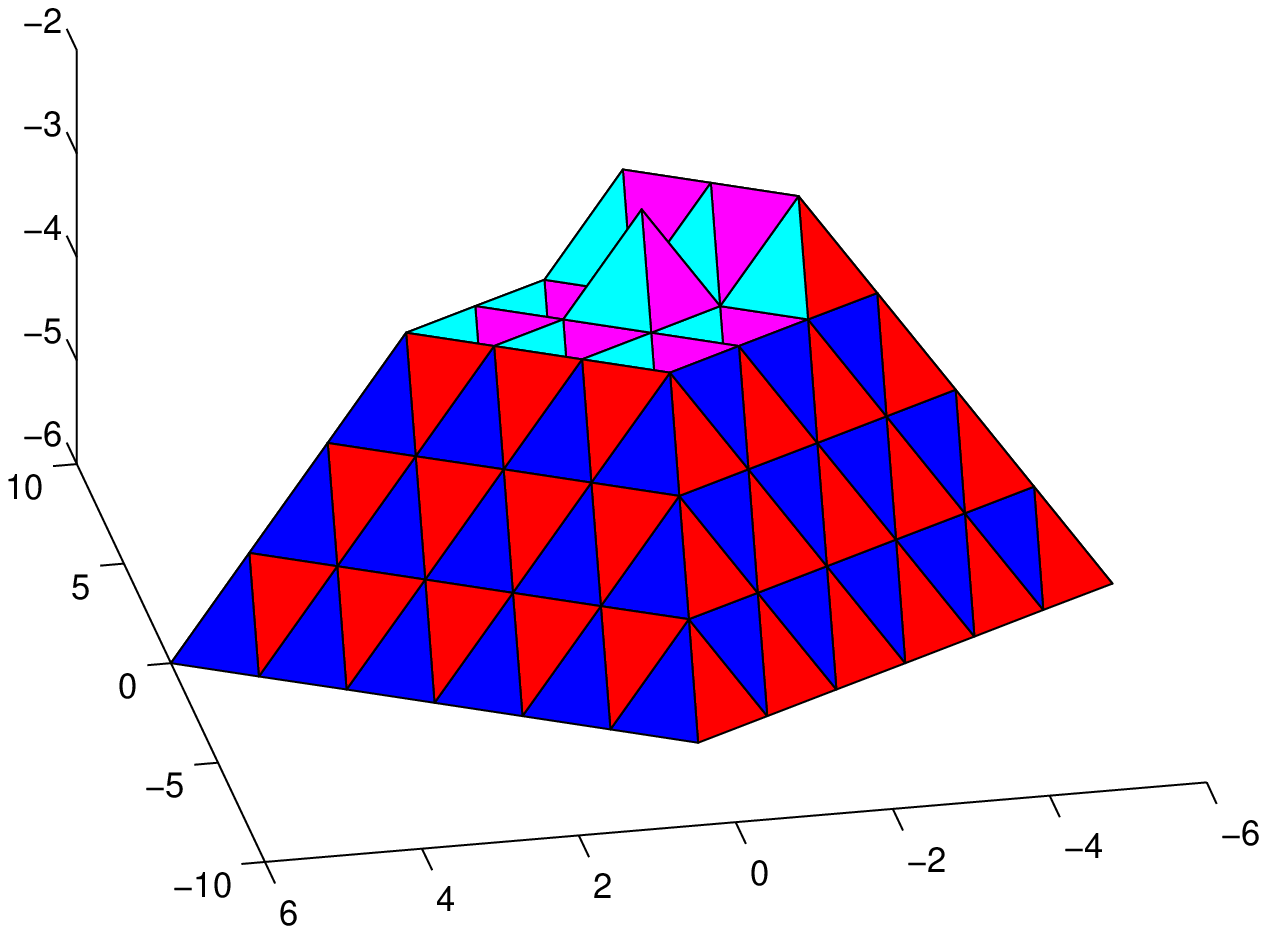} }}%
     \qquad
    \subfloat{{\includegraphics[height=1.6in]{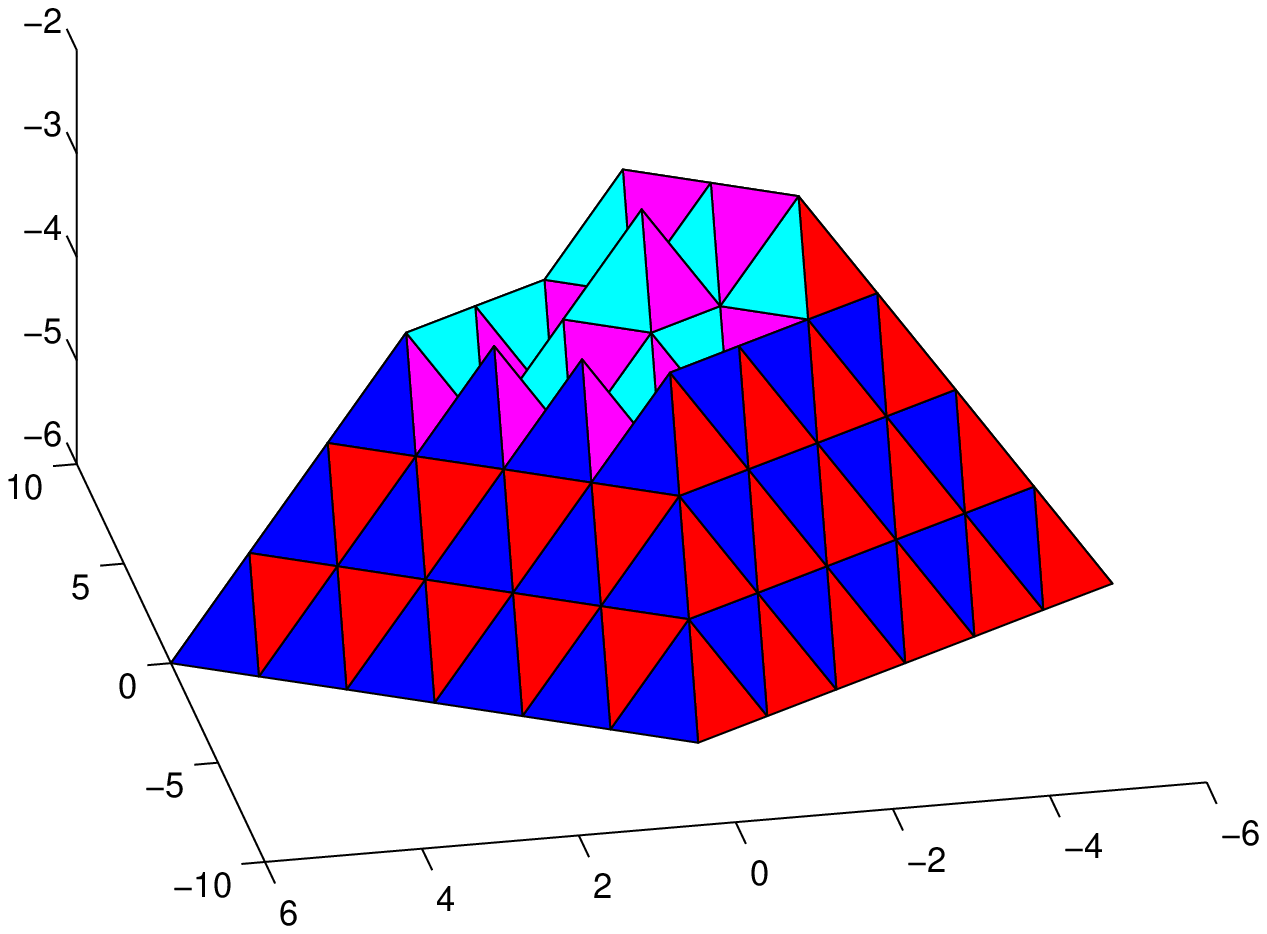} }}%
     \qquad
    \subfloat{{\includegraphics[height=1.6in]{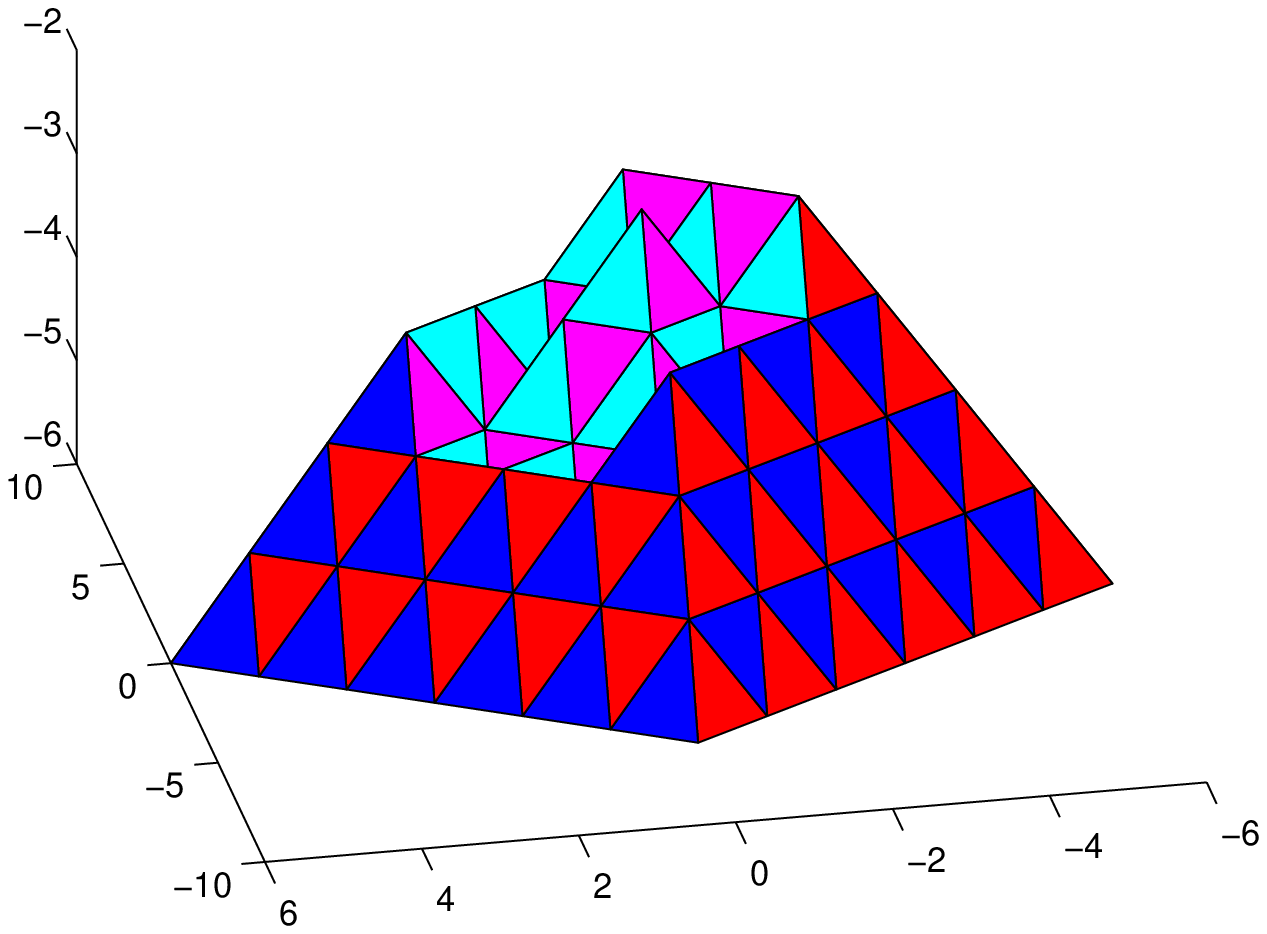} }}%
    \caption{First 8 steps in the chain reaction}%
    \label{pyramidreaction1}%
\end{figure}

\begin{figure}[h!]
    \subfloat{{\includegraphics[height=1.6in]{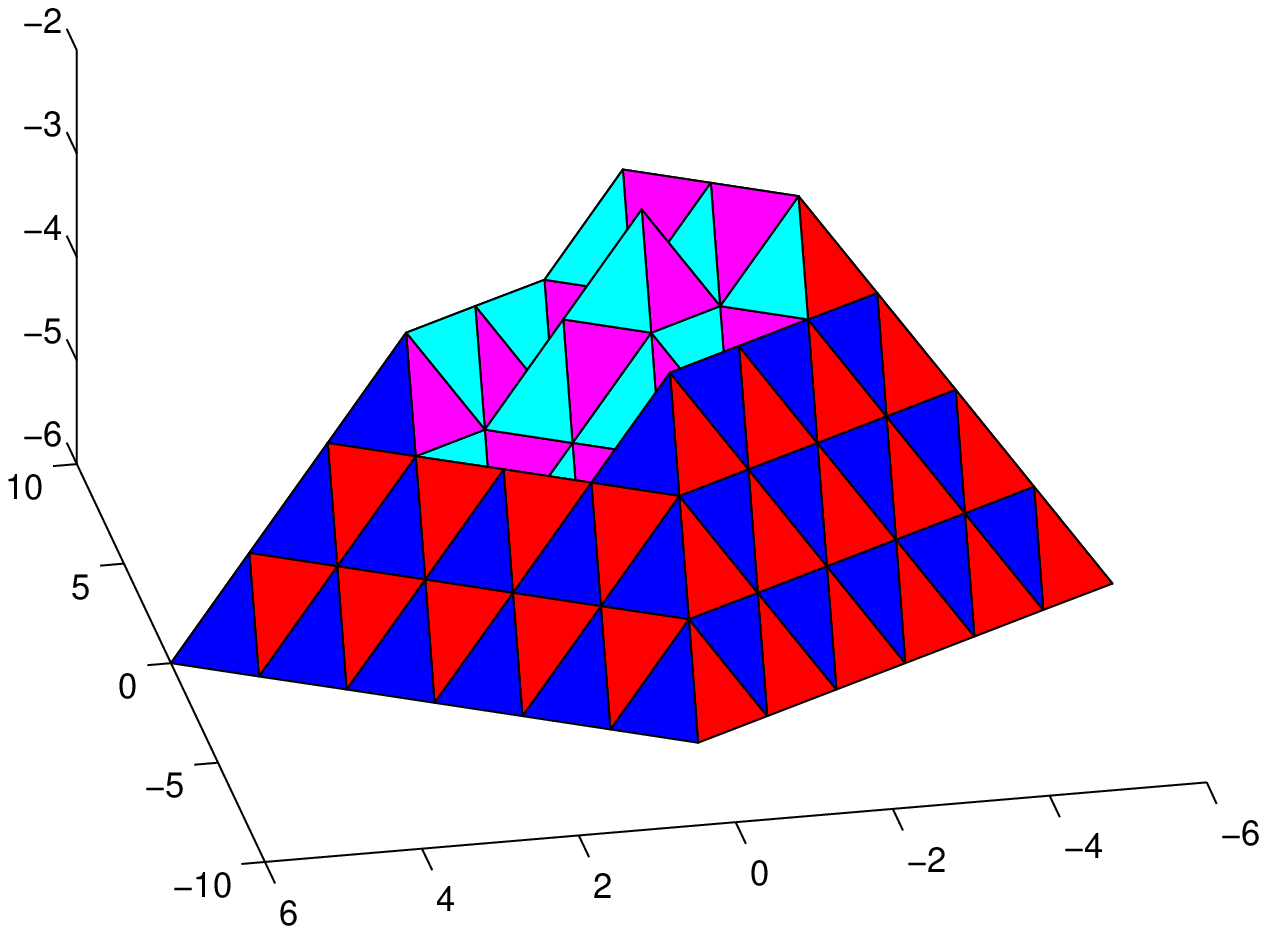} }}%
    \qquad
    \subfloat{{\includegraphics[height=1.6in]{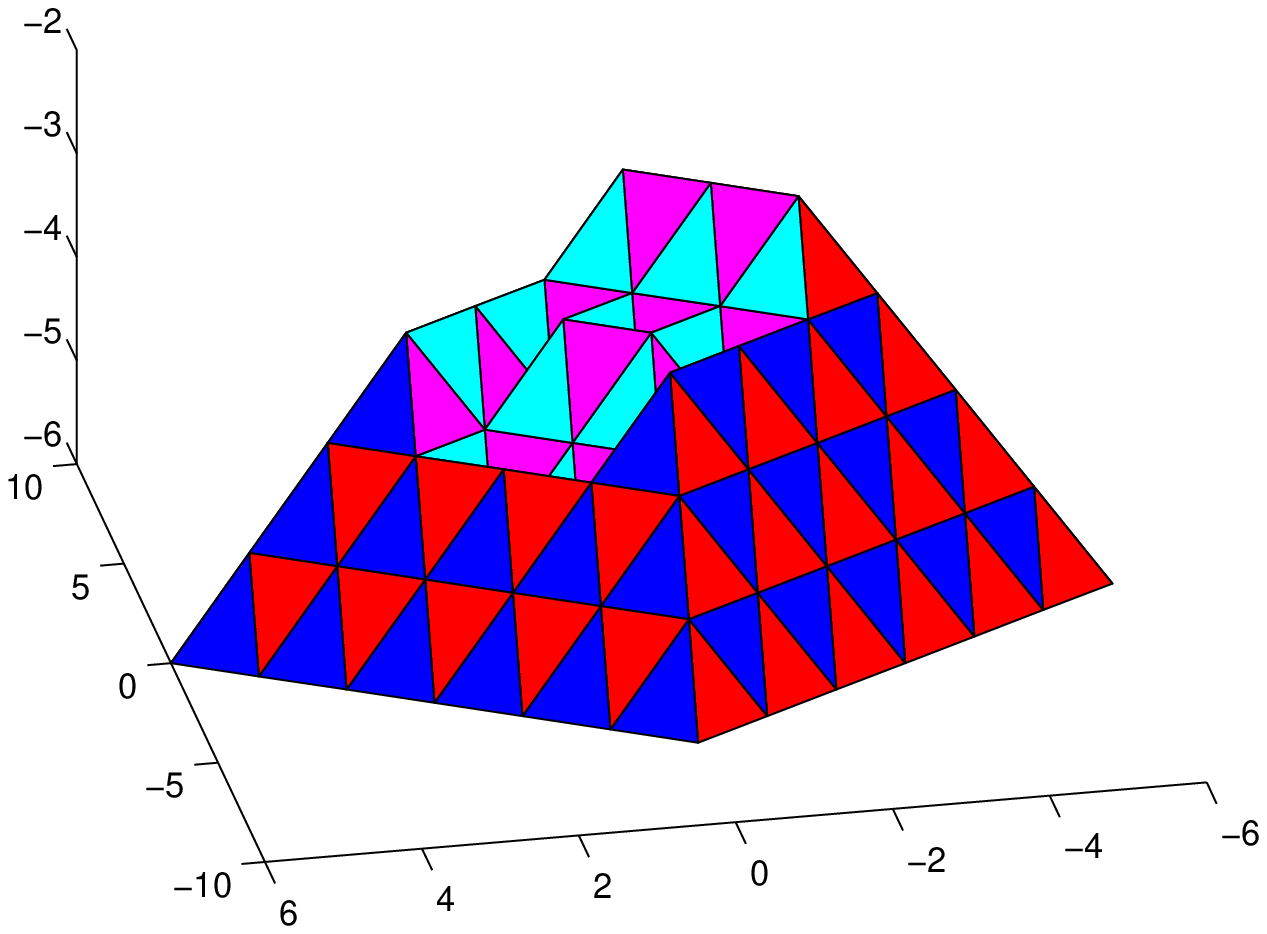} }}%
     \qquad
    \subfloat{{\includegraphics[height=1.6in]{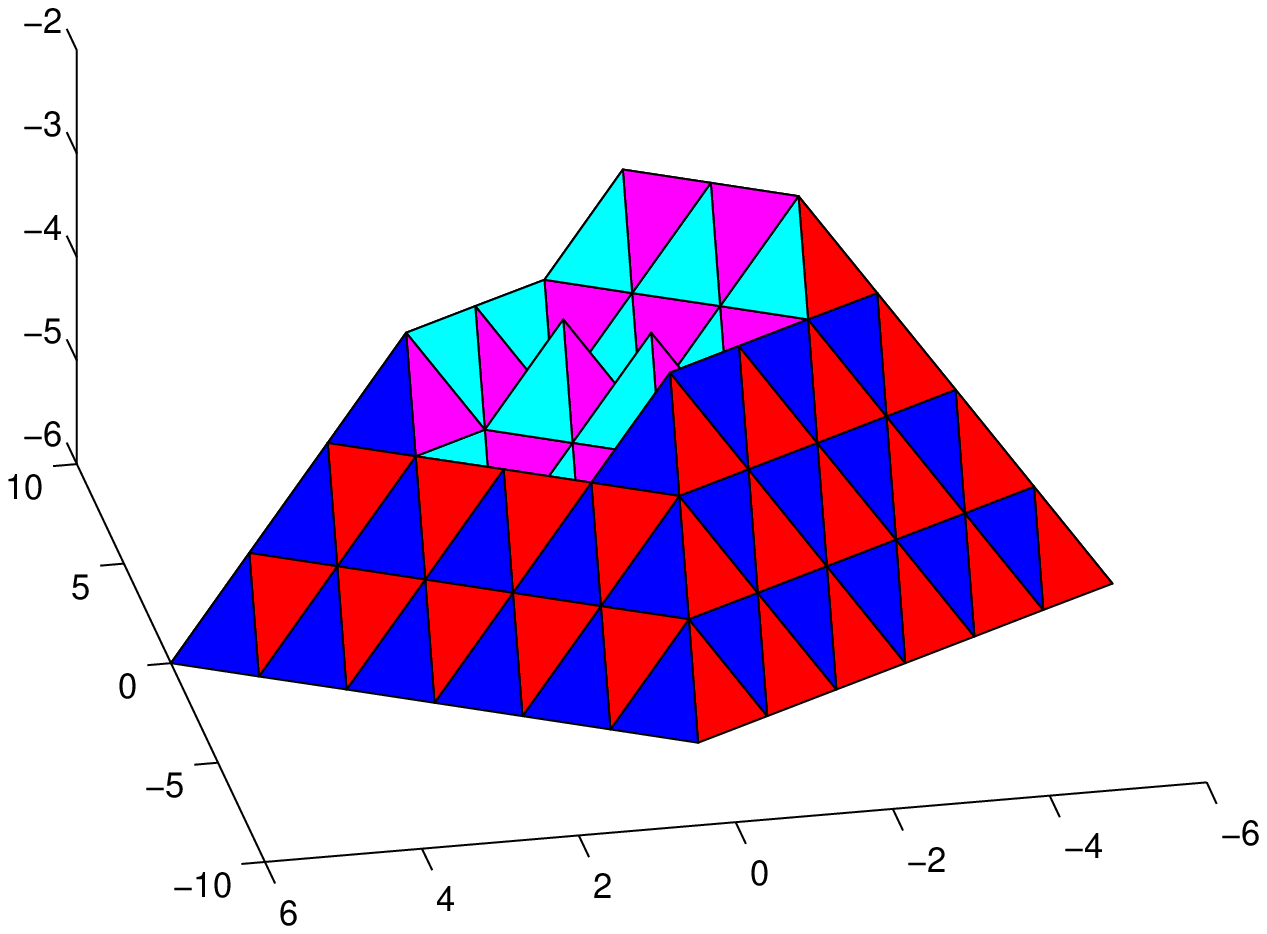} }}%
     \qquad
    \subfloat{{\includegraphics[height=1.6in]{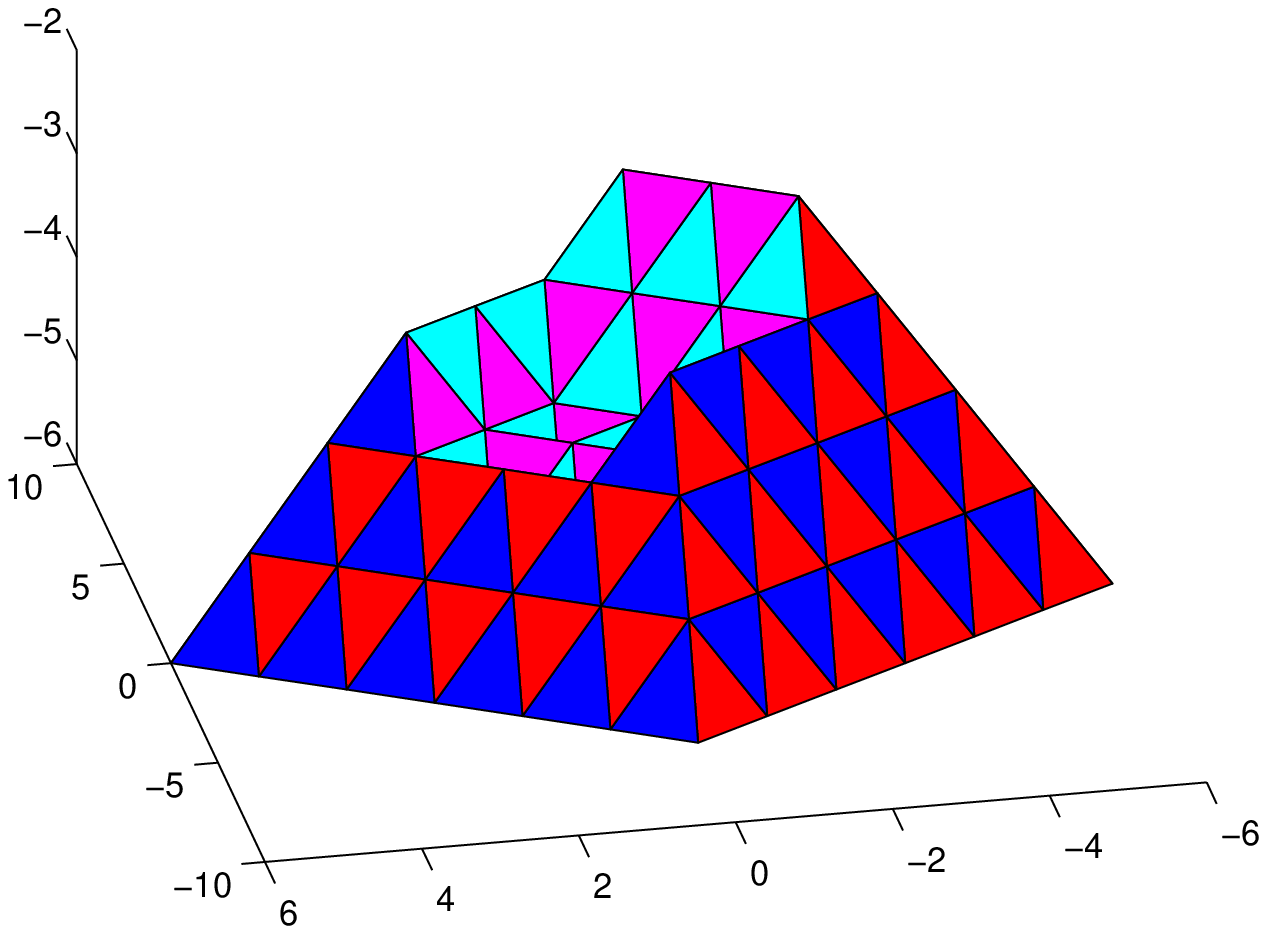} }}%
     \qquad
    \subfloat{{\includegraphics[height=1.6in]{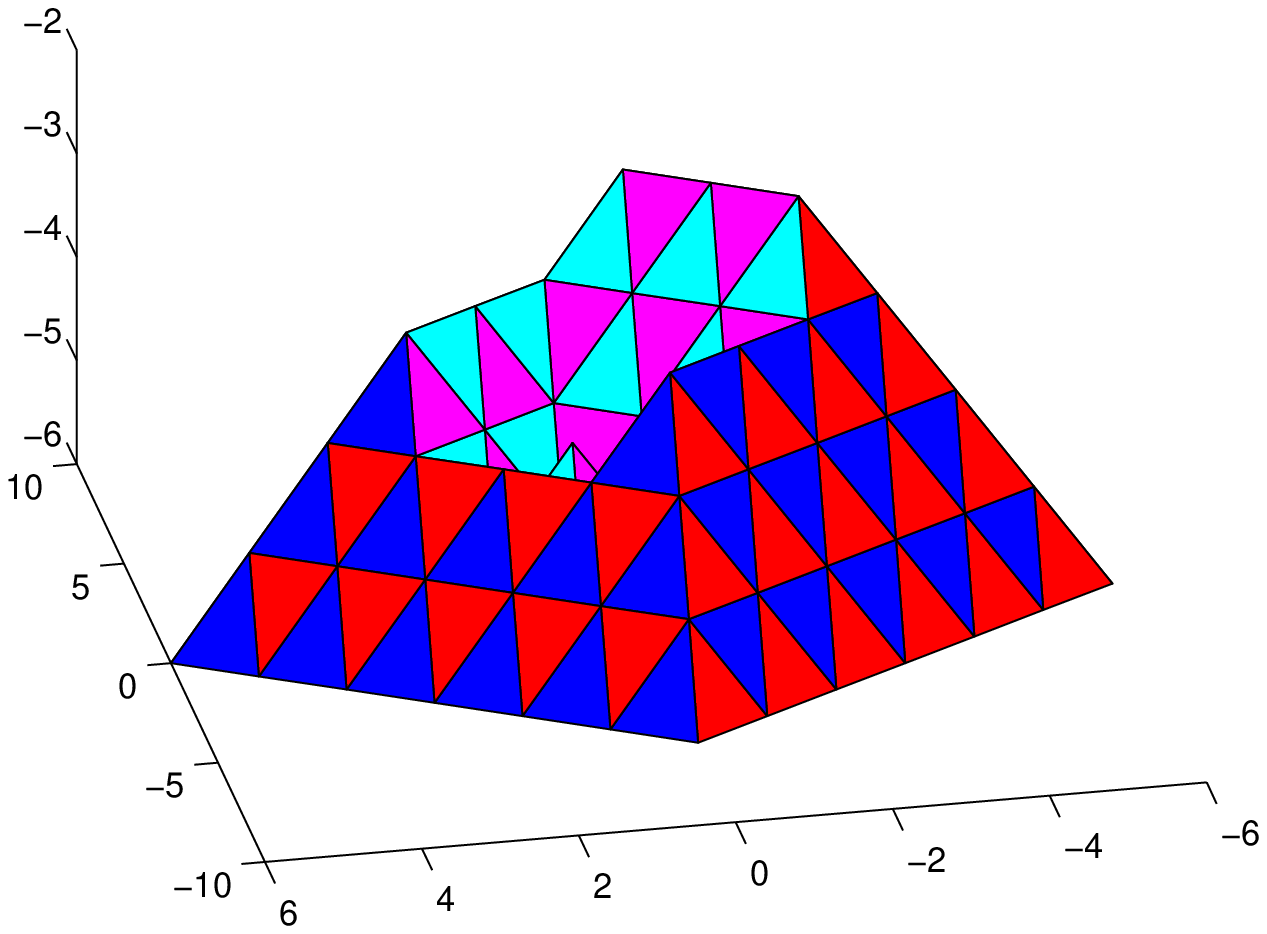} }}%
     \qquad
    \subfloat{{\includegraphics[height=1.6in]{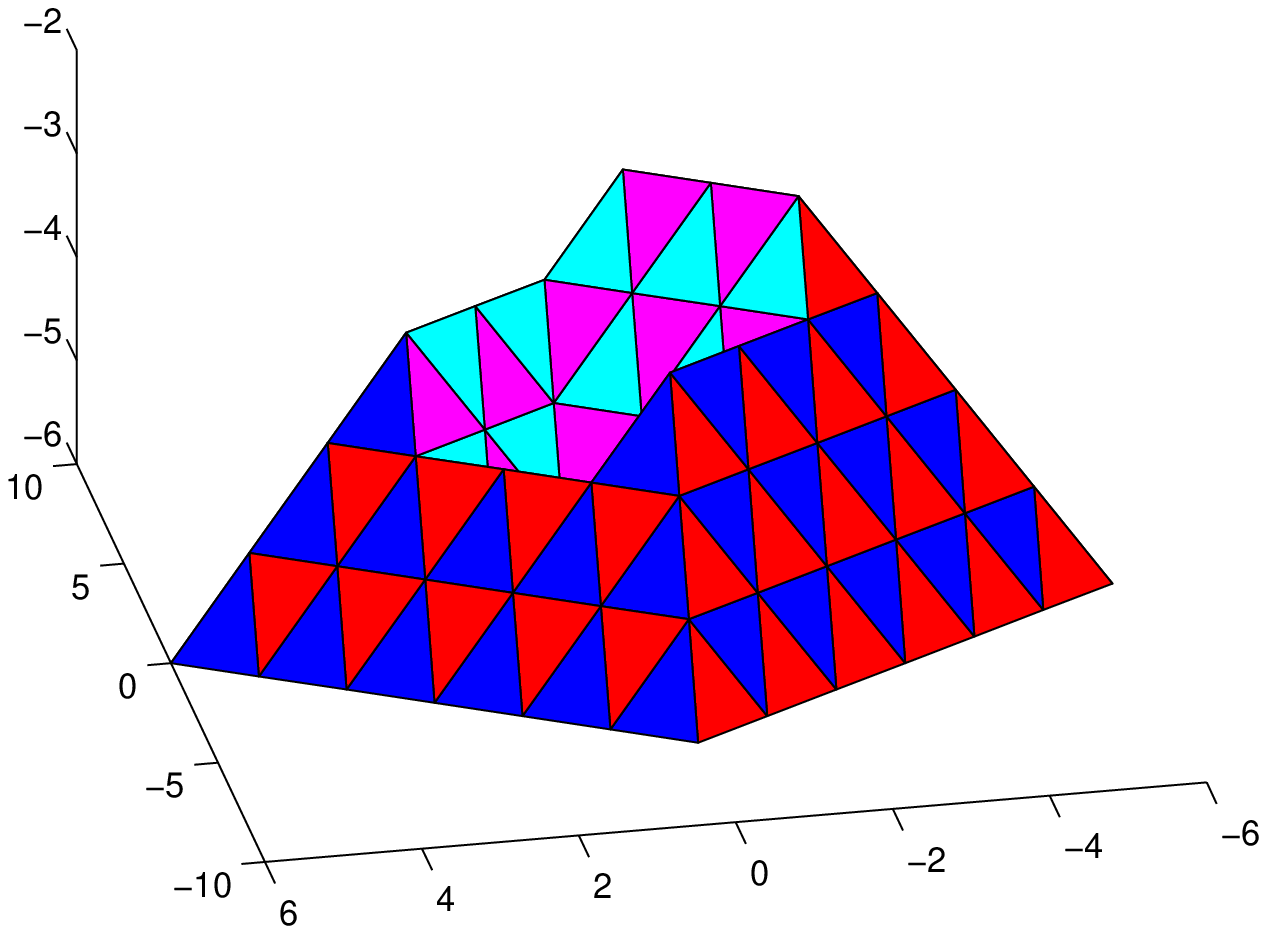} }}%
    \caption{Final 6 steps in the chain reaction}%
    \label{pyramidreaction2}%
\end{figure}

As an example, consider the sequence of plabic graphs in the chain reaction
shown on Figure~\ref{chainreaction}.   In this case $I=\{1,2,3,4,8\}$ and
$J=\{5,6,7,9,10\}$.  Let $G$ and $H$ the first and the last plabic graphs
(respectively) in this chain reaction.  Then $I \in F(G)$ and $J \in F(H)$. The
image $\pi_I(F(G))$ consists of integer points on the upper boundary of a square
pyramid with top vertex $\pi_I(I)$ (see part (A) of
Figure~\ref{chainreaction2}).

%while the image of
%$\pi_J(F(H))$ is embedded in a square pyramid that is inverted upside down, and
%its vertex is $\pi_J(J)$.

%For example,
%since the move (M1) involve 5 faces, the
%corresponding ``octahedron move'' will involve 5 points in $\R^3$.

The map $\pi_I$ transforms the chain reaction shown in
Figure~\ref{chainreaction} into the sequence of 2-dimensional surfaces in $\R^3$
shown in Figure~\ref{chainreaction2}.  These surfaces are the upper boundaries of the
solids obtained from the square pyramid by repeatedly removing little octahedra and tetrahedra,
as shown in the figure.

Similarly, Figures~\ref{pyramidreaction1} and~\ref{pyramidreaction2} show the surfaces
for the chain reaction that corresponds to the plabic graph from
Figure~\ref{fig:twolayerdhoneycomb}.

\section{Final remarks}
\label{sec:final_remarks}

%\subsection{The octahedron recurrence}

%\subsection{Membranes}

\subsection{Arrangements of t-th largest minors}
In the current work, we discussed arrangements of smallest and largest minors.
A forthcoming paper \cite{FM} gives some results regarding arrangements of
$t$-th largest minors, for $t \geq 2$. As in the case of the largest minors,
those arrangements are also closely related to the triangulation of the
hypersimplex.

\subsection{Schur positivity}
Skandera's inequalities \cite{Ska} for products of minors discussed in
Section~\ref{sec:inequalities_products_minors} and also results of
Rhoades-Skandera \cite{RS} on immanants are related
%to monomial positivity and
to {\it Schur positivity\/} of expressions in terms of the Schur functions the
form $s_\lambda s_\mu - s_\nu s_\kappa$.  In \cite{LPP}, several Schur
positivity results of this form were proved.
%, including
%Fomin-Fulton-Li-Poon's conjecture \cite{FFLP} and Okounkov's conjecture
%\cite{Oko}.
%Schur functions were shown to satisfy a certain {\it Schur-log-concavity\/} property.
%In a sense, Theorem~\ref{thm:sorted} is also
%a manifestation of a similar log-concavity principle.
There are some parallels
between the current work on arrangements of equal minors and constructions from
\cite{LPP}.  It would be interesting to clarify this link.

% This might lead to more general Schur positivity results.

\FloatBarrier

\end{document}